\documentclass[11pt]{scrartcl}

\usepackage{amsfonts,amsmath,amsthm,amssymb}
\usepackage{mathrsfs,mathtools}
\usepackage[shortlabels]{enumitem}
\usepackage{enumitem}
\usepackage{hyperref}
\usepackage{esint}
\usepackage{graphicx}
\usepackage{bm}
\usepackage{commath}
\usepackage{esint}
\usepackage{color}

\usepackage{float}
\usepackage{algorithm}
\usepackage{algorithmic}

\numberwithin{equation}{section}

\newtheorem{definition}{Definition}

\newtheorem{theorem}[definition]{Theorem}
\newtheorem{remark}[definition]{Remark}

\newcommand{\curl}{\mathop{\mathbf{curl}}\nolimits}
\newcommand{\di}{\mathop{\mathrm{div}}\nolimits}
\newcommand{\R}{\mathbb{R}}
\newcommand{\N}{\mathbb{N}}
\newcommand{\nd}{n_p}
\newcommand{\dd}{d}

\newcommand{\Dt}{\tau}
\newcommand{\Ds}{\kappa}

\renewcommand{\O}{\Omega}

\newcommand{\rL}{\mathrm{L}}
\newcommand{\rH}{\mathrm{H}}

\newcommand{\partialt}{d_t}

\newcommand{\bv}[1]{{\boldsymbol #1}}

\newcommand{\balpha}{\bv{\alpha}}
\newcommand{\btheta}{\bv{\theta}}
\newcommand{\bphi}{\bv{\phi}}

\newcommand{\bbalpha}{\bar{\bv{\alpha}}}
\newcommand{\bbtheta}{\bar{\bv{\theta}}}
\newcommand{\bbphi}{\bar{\bv{\phi}}}

\newcommand{\balphaI}{\balpha_{int}}
\newcommand{\bthetaI}{\btheta_{int}}

\newcommand{\hD}{\widehat{D}}
\newcommand{\bX}{\mathbf{X}}
\newcommand{\bx}{\boldsymbol{x}}
\newcommand{\br}{\boldsymbol{r}}
\newcommand{\bvarphi}{\boldsymbol{\varphi}}
\newcommand{\brho}{\boldsymbol{\rho}}

\newcommand{\hbx}{\widehat{\bx}}
\newcommand{\bvel}{\mathbf{f}}
\newcommand{\rmvel}{\mathrm{f}}
\newcommand{\hbvel}{\widehat{\rmvel}}
\newcommand{\eps}{\varepsilon}

\newcommand{\calH}{\mathcal{H}}
\newcommand{\calU}{\mathcal{U}}
\newcommand{\calT}{\mathcal{T}}
\newcommand{\calV}{\mathcal{V}}
\newcommand{\calJ}{\mathcal{J}}
\newcommand{\calF}{\mathcal{F}}
\newcommand{\calC}{\mathcal{C}}

\newcommand{\hbP}{\widehat{\mathbb{B}}}

\newcommand{\bD}{\bv{D}}
\newcommand{\bR}{\bv{R}}

\newcommand{\bH}{\bv{h}}

\newcommand{\mD}{\mathbb{D}}
\newcommand{\mR}{\mathbb{R}}

\newcommand{\weakto}{\rightharpoonup}
\DeclareMathOperator*{\argmin}{arg\,min}

\newcommand{\pair}[1]{\left\langle #1 \right\rangle}

\newcommand{\abssec}[1]{\noindent\normalsize {\bfseries #1\quad }\ignorespaces}
\renewenvironment{abstract}{\abssec{Abstract}}{\par\vspace{.1in}}
\newenvironment{keywords}{\abssec{Key Words}}{\par\vspace{.1in}}
\newenvironment{AMS}{\abssec{AMS subject
		classification}}{\par\vspace{.1in}}

\DeclareMathAlphabet{\mathpzc}{OT1}{pzc}{m}{it}

\usepackage[textsize=small]{todonotes}
\numberwithin{equation}{section}

\title{
    Controlling the Kelvin Force: Basic Strategies and Applications to Magnetic Drug Targeting
\thanks{The work of  H. Antil has been partially supported by NSF grants DMS-1109325 and DMS-1521590.
 R.H. Nochetto has been partially supported by NSF grants DMS-1109325 and DMS-1411808 and
  P. Venegas  has been  supported by NSF grant DMS-1411808 and FONDECYT project 11160186.}
}

\author{Harbir Antil\thanks{Department of Mathematical Sciences, George Mason University, Fairfax, VA 22030, USA. \texttt{hantil@gmu.edu}}
\and
Ricardo H. Nochetto\thanks{Department of Mathematics and Institute for Physical
           Science and Technology, University of Maryland College Park, MD 20742, USA.
\texttt{rhn@math.umd.edu}}
\and 
Pablo Venegas\thanks{GIMNAP, Departamento de Matem\'atica, Universidad del B\'io-B\'io, Chile.
\texttt{pvenegas@ubiobio.cl}}
}

\begin{document}

\maketitle

\begin{abstract}
Motivated by problems arising in magnetic drug targeting, we propose
to generate an almost constant Kelvin (magnetic) force in a target subdomain,
moving along a prescribed trajectory. This is carried out by solving a
minimization problem with a tracking type cost functional. The
magnetic sources are assumed to be dipoles and the control
variables are the magnetic field intensity, the source location and
the magnetic field direction. The resulting magnetic field is shown
to effectively steer the drug concentration, governed by a
drift-diffusion PDE, from an initial to a desired location with limited spreading.
\end{abstract}



\section{Introduction}
Magnetic drug targeting (MDT) is an important application of ferrofluids where drugs,
with ferromagnetic particles in suspension, are injected into the body.
The external magnetic field then can control the drug and subsequently the drug can target the relevant areas, for example, solid tumors
(see, for instance, \cite{Lubbe1996}).
Mathematically this process can be modeled as follows: let the concentration of magnetic nanoparticles  confined in a domain $\widetilde\Omega\subset \R^\dd, \dd=2,3$.
Let $c$ be the drug concentration and $\bH$ the magnetic field, then the evolution
 of  $c$ by the applied magnetic field is given by the following
  convection-diffusion model \cite{GR2005,KS2011}:
\begin{align}\label{eq:strong_1}
\partial_t c+\di \left( -A\nabla c   +( \bv{u} +
  \gamma \bv{f}(\bH))c\right)=0 \quad \mbox{in } \widetilde\Omega\times(0,T)&\\
(-A\nabla c   +(\bv{u} +
  \gamma \bv{f}(\bH))c )\cdot\boldsymbol{n}=0\quad \mbox{on } \partial\widetilde\Omega\times(0,T)\qquad c(x, 0) = c_0
\quad \mbox{in } \widetilde\Omega&\label{eq:strong_2}\\
 \curl \bH=\bv{0} \quad \mbox{in } \widetilde\Omega\qquad   \di \del{\mu\bH}=0
 \quad \mbox{in } \widetilde\Omega\label{eq:strong_max}
\end{align}
where $\gamma$ is a constitutive constant,  $A$ is a diffusion coefficient matrix,  
$\bv{u}$ is a fixed velocity vector, $\boldsymbol{n}$ is the outward unit normal, the 
magnetic permeability $\mu$ is assumed to be constant, 
and $\bv{f}$ is the \textit{Kelvin force}  applied to a magnetic particle:
$$
\bv{f}(\bH)= \nu \nabla|\bH|^2 ,
$$
where  $\nu>0$ is a constant depending on the volume of the particle and its
permeability (see \cite{KS2011} for details).
The two fundamental units that determine
the evolution of concentration $c$ are transport and diffusion (cf.~\eqref{eq:strong_1}).

The aim of magnetic drug delivery is to move a drug, initially
concentrated in a small subdomain $D_0$ and described by the initial concentration
$c_0$, to another suddomain $D_t$ (desired location) while minimizing the
spreading which is characterized by the concentration $c(t)$ at time
$t$. This is a challenge because magnetic fields inherently
tend to disperse magnetic particles. It is known that it is not
possible to concentrate magnetic nanoparticles with a static magnetic source
(Earnshaw's theorem). To understand this essential obstruction,
let $D\subset\widetilde{\Omega}$ be a subdomain and observe that
\[
\di \bv{f} (\bv{h})=\di(\nabla|\bH|^2)=2\Delta \bH+2|\nabla \bH|^2=2|\nabla \bH|^2\geq 0
\quad\textrm{in }\widetilde{\Omega}
\]
implies
\begin{equation}\label{E:defocussing}
\int_{\partial D} \bv{f}(\bH)\cdot\boldsymbol{n} \,
ds=\int_{D}\di\bv{f}(\bH) \, dx\geq 0.
\end{equation}
This prevents the magnetic force from focussing
$\bv{f}(\bH)\cdot\boldsymbol{n}\le0$ on $\partial D$, a dispersion
effect we could compensate for. We are thus
confronted with the following question:
\begin{equation}\label{E:key-question}
\begin{minipage}{0.85\linewidth}
\emph{
Is it possible to generate an appropriate magnetic force $\bv{f}(\bH)$
to steer a concentration to a desired location while minimizing
the spreading inherent to magnetic forces?
}
\end{minipage}
\end{equation}
A first approach to tackle this problem was proposed in \cite{KS2011}.
A feedback control of an arrangement of magnetic sources (with fixed
positions) is proposed to
drive a spot of distributed ferrofluid from an initial point to a target position.
The authors focus on the center of mass of the concentration
and its covariance matrix which measures the spreading. This leads
to a constrained minimization problem where
the dynamics of \eqref{eq:strong_1}-\eqref{eq:strong_2} are approximated by two coupled
ordinary differential equations.

Our approach to \eqref{E:key-question} aims to design and study an
optimization framework for steering a \emph{target subdomain} $D_t$ from an initial
to a desired location via a magnetic force $\bv{f}(\bH)$
which is almost constant in space within $D_t$ and thus minimizes the drug
spreading outside $D_t$ due to \eqref{E:defocussing}. We stress that
the relative size of the diffusion coefficient matrix $A$ also
influences the spreading. This paper continues the program started in
\cite{ANV2016} upon adding the following features:
\begin{enumerate}[$\bullet$]
\item
\emph{Controls:} The control variables are now the magnetic field
intensity, the magnetic field direction, and the position of the
magnetic dipoles, instead of just the former \cite{ANV2016}.
This provides greater flexibility and controllability
of the process.

\item
\emph{Final configurations:}
We present several non-trivial, though extremely simplified,
examples for $d=2$ that allow us to assess the question of feasibility
of MDT and of our approach for its development and automation. For
instance, we study the issue of moving the concentration around
obstacles and of magnetic injection while minimizing the spreading.

\item
\emph{Advection-diffusion PDE:}
For the magnetic injection example, we solve \eqref{eq:strong_1} with a realistic Neumann
boundary condition \eqref{eq:strong_2}, instead of the Dirichlet boundary
conditions as in \cite{ANV2016}. We also consider a smaller domain than the one considered in
\cite{KS2011,ANV2016} which may lead to numerical oscillations as the concentration tends to go to the boundary
(see \cite{NBBS2011}).
In order to deal with the advection-dominant case
we have adopted the edge-averaged finite element method of \cite{XZ1999} but other alternatives
such as the implicit scheme with SUPG stabilization or the explicit schemes are equally valid.
Even though we do not observe
oscillations in the solution, we notice a large amount of diffusion due to the upwinding nature of
the scheme. This behavior makes it difficult to control the concentration
in complicated geometries, 
for instance, the one needed to control
the flow around an obstacle.

For the second example we intend to move the concentration around an obstacle.
Given that a small diffusion parameter is needed to achieve this goal, homogeneous
Dirichlet boundary condition (instead of Neumann condition) is meaningful if the concentration
stays away from the boundary $\partial\widetilde\Omega$ (see Section~\ref{s:explicit} for details).
Moreover, in order to deal with the numerical diffusion, we consider finite element method for space discretization 
and the explicit Euler scheme with mass lumping for time discretization, including a correction to tackle the dispersive effects of mass lumping \cite{GP2013}. The resulting scheme works better than the upwind scheme.

\end{enumerate}

It is worth mentioning that this approach is not only applicable for
magnetic drug targeting but for
many other applications where magnetic force plays a role such as:
gene therapy \cite{D2006}, magnetized
stem-cells \cite{SKL2008}, magnetic tweezers \cite{HJBTF2003},
lab-on-a-chip systems that include
magnetic particles or fluids \cite{GLL2010} or
separation of particles \cite{LYL2014}, just to name a few.

The paper is organized as follows: in section~\ref{s:da} we describe the dipole approximation of the magnetic sources to be used throughout the remaining paper. Our main work begins in section~\ref{sec:model} where we establish the optimization problems. In section~\ref{s:disc} we discuss the numerical approximation and convergence of our discrete scheme. We provide numerical illustrations for the optimization problem in section~\ref{sec:example1} and conclude with a numerical scheme and several illustrative examples for \eqref{eq:strong_1}-\eqref{eq:strong_max} in section~\ref{s:addiff_numerics}.

\section{Dipole approximation}\label{s:da}
 Let $\O\subset \R^d$, $d=2,3$
be open bounded fictitious domain where we intent to control the magnetic force and $T>0$ be the final time.
 We assume that the magnetic sources lies outside $\overline{\Omega}$, then, from
 the  Maxwell equations it follows that the magnetic field $\bH$ satisfies:
\begin{align}\label{eq:control_h1}
\curl \bH&=\bv{0}, \quad   \di \bH=0 \quad \mbox{in } \Omega,
\end{align}
where the last equation follows by assuming a linear relation between the magnetic induction
 and the magnetic field. The magnetic field generated by a current distribution
and a permanent magnet  can be modeled by the Biot-Savart law, which is a magnetostatic approximation.
However, for simplicity, in our case we consider a dipole approximation to the magnetic
source, which provides a concise and easily tractable representation
 of the magnetic field (see \cite{PA2013} for a quantification  of the error associated
with the dipole approximation). This approximation is commonly used
 for localization of objects in applications ranging from medical imaging to
military. It is also extensively used in real-time control of magnetic devices in medical sciences
\cite{FKA2010,MA2011,NKA2010}.

From now on we will assume that the magnetic field is modeled by the superposition
of a fixed number $\nd$ of dipoles, namely
\begin{equation}\label{eq:h_mag_1}
\bH=\sum_{i=1}^{\nd}\alpha_i(t)\del{\dd\dfrac{\br_i\br_i^\top}{|\br_i|^2}
-\mathbb{I}}\dfrac{\widehat{\bv{d}_i}}{|\br_i|^\dd},
\end{equation}
where $\br=\bx-\bx_i\in \R^{\dd}$ and $\mathbb{I}\in \R^{\dd\times \dd}$ is the identity matrix. In addition,
$\widehat{\bv{d}_i}\in\R^\dd$ and $\bx_i\in\R^\dd\setminus\overline{\O}$, $i=1,\ldots,\nd$, denote
unit vectors and dipole positions, respectively.
It is straightforward to show that the magnetic field given by \eqref{eq:h_mag_1}
satisfies \eqref{eq:control_h1}.

%
\section{Model optimization problem}\label{sec:model}

To fix ideas, we assume the time dependent domain $D_t$  moves along a pre-specified curve smooth curve $\calC$
and
 is strictly contained in $\Omega$ for every $t\in[0,T]$.
Our goal here is then to approximate a vector field $\bvel \in [\rL^{2}(0,T;\rL^2(\Omega))]^d$ in $D_t$ by the
 so--called Kelvin force generated by a configuration of magnetic sources.
 In particular, for MDT, $\bvel$ could be considered uniform on $D_t$.
This vector field will then replace $\nabla| \bH|^2$  in \eqref{eq:strong_1} and will drive the concentration to a desired location.

In \cite{ANV2016} we propose an optimal control problem in order to study the feasibility of the dipole approximation,
namely, whether we can create  a Kelvin force field that allows us to control, for instance, the concentration
of nanoparticles in a bounded domain $\Omega$.
With this aim, a desired vector  field  $\bvel$ and a time dependent domain $D_t\subset\Omega$, $t\in[0,T]$
 are assumed to be known (see Figure~\ref{fig:3config} (left)).
The field $\bvel$ is approximated in $D_t\subset\Omega$ by the Kelvin force generated by a configuration of $\nd$
dipoles with fixed directions $\widehat{\bv{d}}_i$ and positions $\bv{x}_i$.  Here
the control is the vector magnetic field's magnitude $\balpha(t):=(\alpha_1(t),\ldots,\alpha_{\nd}(t))^T\in\R^{\nd}$.

In this work we will approximate $\bvel:=(\rmvel_{1},\ldots, \rmvel_{\dd})^\top$ with two different dipole configurations.

\subsection{Controlling intensities and directions}

\begin{figure}[h!]
\centering
\includegraphics[height=3.2cm]{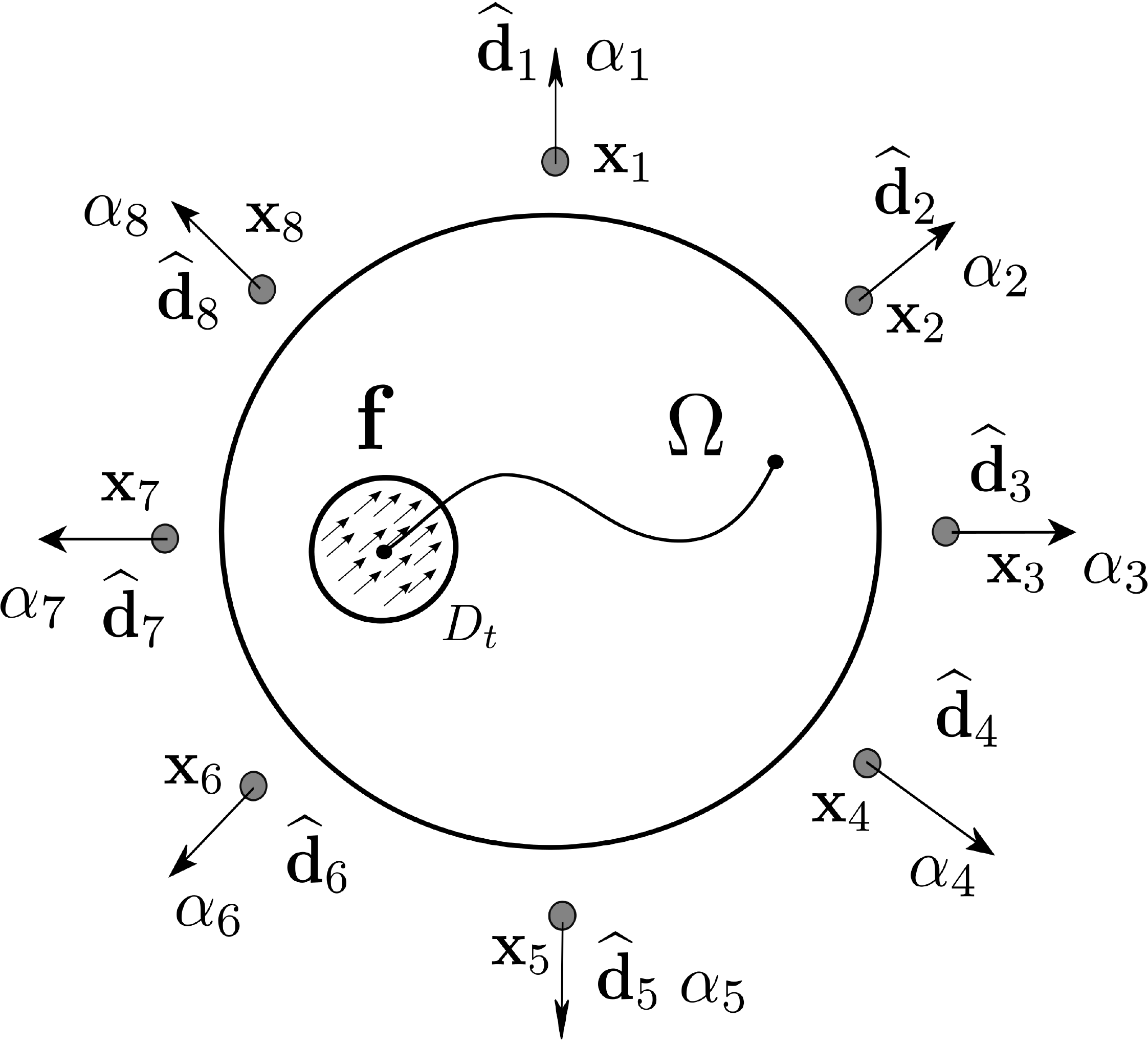}
\includegraphics[height=3.2cm]{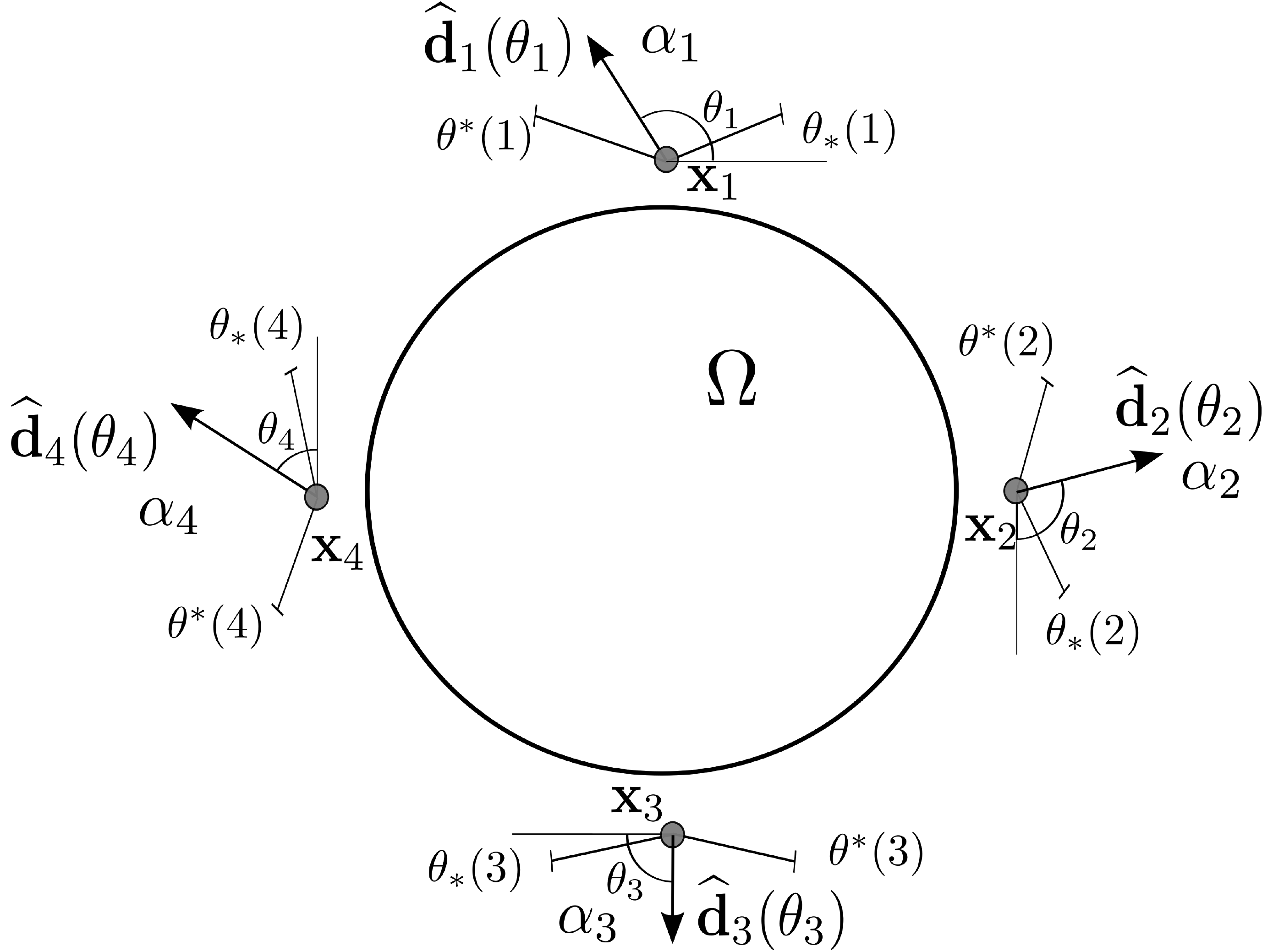}
\includegraphics[height=3.2cm]{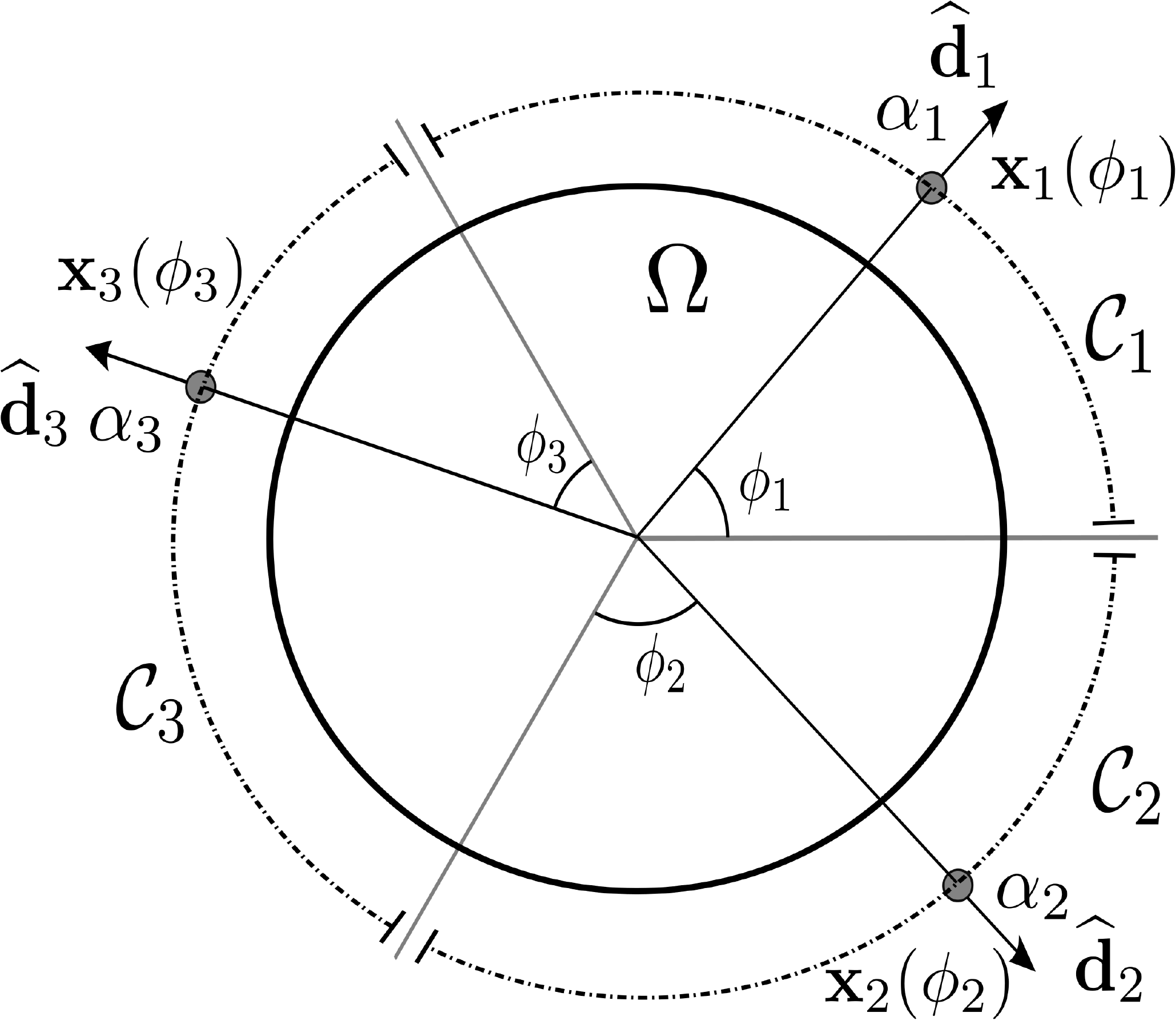}
\caption{Left: Configuration of $\nd=8$ dipoles surrounding a computational domain $\Omega\subset \R^2$ and moving
domain $D_t$.
 Here the target vector field $\bvel$, represented with arrows, is shown only in $D_t$. The center panel shows
 a configuration of $\nd=4$ dipoles with variable directions.
 On the right  panel we show the positions  $\{\bx_i\}_{i=1}^{3}$ of the dipoles moving along the curves
 $\calC_i, i=1,2,3$, respectively.
 Each dipole is characterized by
 its position $\bx_i$ (represented by a dot outside $\O$), direction $\bv{d}_i$
(represented by an arrow)
and the magnitude $\balpha_i$, for $i=1,\ldots,\nd$. }
\label{fig:3config}
\end{figure}

 First we consider  fixed dipole positions $\left\{\bx_i\right\}_{i=1}^{\nd}$  and
 the controls are the magnitudes $\left\{\alpha_i\right\}_{i=1}^{\nd}$ and  directions
 $\big\{\widehat{\bv{d}}_i\big\}_{i=1}^{\nd}$ parametrized by one degree of freedom per dipole
 (see Figure~\ref{fig:3config} (center)), namely, if
$\balpha(t):=(\alpha_1(t),\ldots,\alpha_{\nd}(t))^\top\in\R^{\nd}$ and
$\btheta(t):=(\theta_1(t),\ldots,\theta_{\nd}(t))^\top\in\R^{\nd}$ denote the vectors of magnetic field's
magnitude and ``angles'', such that $\widehat{\bv{d}}_i=\widehat{\bv{d}}_i(\theta_i)$ for a regular
function $\widehat{\bv{d}}_i$ (for instance, $\widehat{\bv{d}}_i(\theta_i)=\left(\cos(\theta_i),\sin(\theta_i)\right)$
when $\dd=2$),  then we want to study the problem
\begin{subequations}\label{eq:PJ1}
\begin{equation}\label{eq:min_dir}
\min_{(\balpha,\btheta)\in \calH_{ad}\times\calV_{ad} } \calJ(\bH(\balpha,\btheta))
\end{equation}
with
\begin{equation}\label{eq:calJ}
\calJ(\bH(\balpha,\btheta))=\dfrac{1}{2}\int_0^T\|\nabla|\bH|^2-\bvel\|^2_{\rL^2(D_t)}
+\dfrac{\lambda}{2}\int_0^T|\partial_t\balpha|^2 +\dfrac{\eta}{2}\int_0^T|\partial_t\btheta|^2
\end{equation}
and
\begin{equation}\label{eq:h_mag1}
\bH(\balpha,\btheta)=\sum_{i=1}^{\nd}\alpha_i(t)\del{\dd\dfrac{\br_i\br_i^\top}{|\br_i|^2}-\mathbb{I}}
\dfrac{\widehat{\bv{d}_i}(\theta_i)}{|\br_i|^\dd},
\end{equation}
\end{subequations}
and where  $(\balpha,\btheta) \in \calH_{ad}\times\calV_{ad} $ and $\lambda$ and $\eta$ are the costs of the control.
For given  constant vectors $\balpha_0,\btheta_0,\balpha_*,\btheta_*, \balpha^*, \btheta^*\in\R^{\nd}$, we
seek $(\balpha,\btheta)$  in the following admissible convex sets:
\[
\calH_{ad} := \left\{\balpha\in [\rH^1(0,T)]^{\nd}: \balpha(0)=\balpha_0,\,\, \mbox{and}\,\, \balpha_* \leq
\balpha(t)\leq \balpha^*,\,\, \forall t\in [0,T]  \right\}
\]
and
\[
\calV_{ad} := \left\{\btheta \in [\rH^1(0,T)]^{\nd} : \btheta(0)=\btheta_0,\,\, \mbox{and}\,\, \btheta_*\leq
\btheta(t)\leq \btheta^*, \forall \, t\in [0,T]  \right\}.
\]
Notice that, in view of \eqref{eq:h_mag1},
we can rewrite $\bH$ as $\bH(\balpha,\btheta)=\mD(\btheta)\balpha $
where
\[
\mD(\btheta)=\big(\bD_1(\theta_1)\, \bD_2(\theta_2)
\,\ldots\,\bD_{\nd}(\theta_{\nd})\big)\in
   [\rH^1(0,T;C^{\infty}(\overline{\Omega}))]^{\dd\times {\nd}}
\]
is such that
\begin{equation*}
\bD_i(\theta_i):=\del{\dd\dfrac{(\bx-\bx_i)(\bx-\bx_i)^\top}{|\bx-\bx_i|^2}-\mathbb{I}}
\dfrac{\widehat{\bv{d}_i}(\theta_i)}{|\bx-\bx_i|^\dd}=\mathbb{M}_i(\bx)\widehat{\bv{d}_i}(\theta_i) .
\end{equation*}
Thus, the Kelvin force in terms of $\balpha$ and $\btheta$ is given by
\begin{align}\label{eq:F_theta}
\nabla|\bH|^2&=\nabla\left(\balpha(t)^\top(\mD(\btheta)^\top\mD(\btheta))\balpha(t)\right)\nonumber \\
&=\begin{pmatrix}\balpha(t)^\top\partial_{x_1}(\mD(\btheta)^\top\mD(\btheta))\balpha(t)\\
\vdots
\\
\balpha(t)^\top\partial_{x_\dd}(\mD(\btheta)^\top\mD(\btheta))\balpha(t)\end{pmatrix}
=\begin{pmatrix}\balpha^\top\mathbb{B}_1(\btheta)\balpha\\
\vdots
\\
\balpha^\top\mathbb{B}_\dd(\btheta)\balpha\end{pmatrix}
\end{align}
with $\mathbb{B}_k(\btheta):=\partial_{x_k}(\mD^\top(\btheta)\mD(\btheta))
\in [\rH^1(0,T;C^{\infty}(\overline{\Omega}))]^{{\nd}\times {\nd}}$.

\subsection{Controlling intensities and positions}

For the second configuration we consider  fixed dipole directions $\big\{\widehat{\bv{d}}_i\big\}_{i=1}^{\nd}$
  and  the controls are the magnitudes $\left\{\alpha_i\right\}_{i=1}^{\nd}$ and  positions $\left\{\bx_i\right\}_{i=1}^{\nd}$.
 In this case, we assume that each dipole $i$ moves along a prescribed trajectory $\calC_i$ parameterized
by a $C^1([0,T])$ curve $\brho_i: [\bphi_{*}(i),\bphi^{*}(i)]\to \calC_i$, $i=1,\ldots,\nd$   where
$\bphi_*, \bphi^*\in\R^{\nd}$ (see Figure~\ref{fig:3config} (right)). Thus, we study the problem
\begin{subequations}\label{eq:PJ2}
\begin{equation}\label{eq:min_pos}
\min_{(\balpha,\bphi)\in \calH_{ad}\times\calU_{ad} }  \calF(\bH(\balpha,\bphi))
\end{equation}
with
\begin{equation}\label{eq:calF}
\calF(\bH(\balpha,\bphi))=\dfrac{1}{2}\int_0^T\|\nabla|\bH|^2-\bvel\|^2_{\rL^2(D_t)}
+\dfrac{\lambda}{2}\int_0^T|\partial_t\balpha|^2 +\dfrac{\beta}{2}\int_0^T|\partial_t\bphi|^2
\end{equation}
and
\begin{equation}\label{eq:h_mag2}
\bH(\balpha,\bphi)=\sum_{i=1}^{\nd}\alpha_i(t)\del{\dd\dfrac{\br_i(\phi_i)
\br_i(\phi_i)^\top}{|\br_i(\phi_i)|^2}-\mathbb{I}}
\dfrac{\widehat{\bv{d}_i}}{|\br_i(\phi_i)|^\dd},
\end{equation}
\end{subequations}
where $\br_i(\phi_i)=(\bx-\brho_i(\phi_i))$,   $\bphi=(\phi_1,\ldots,\phi_{\nd})$, $(\balpha,\bphi) \in \calH_{ad}\times\calU_{ad} $ and, for a given  constant vector $\bphi_0\in\R^{\nd}$, we
consider the admissible set
$ \calU_{ad} := \big\{\bphi\in [\rH^1(0,T)]^{\nd}: \bphi(0)=\bphi_0,\,\, \mbox{and}\,\, \bphi_* \leq
\bphi(t)\leq \bphi^*,\,\, \forall t\in [0,T]  \big\}$.
We first notice that, in view of \eqref{eq:h_mag2}, we can rewrite $\bH$ as
$\bH(\balpha,\bphi)=\mR(\bphi)\balpha$ where
$\mR(\bphi)=\left(\bR_1(\phi_1)\, \bR_2(\phi_2)\,\ldots\,\bR_{\nd}(\phi_{\nd})\right)$
belongs to $[\rH^1(0,T;C^{\infty}(\overline{\Omega}))]^{\dd\times {\nd}}$ and is given by
\begin{equation*}
\bR_i(\phi_i):=\del{\dd\dfrac{\br_i(\phi_i)\br_i(\phi_i)^\top}{|\br_i(\phi_i)|^2}-\mathbb{I}}
\dfrac{\widehat{\bv{d}_i}}{|\br_i(\phi_i)|^\dd},
\end{equation*}
therefore, we can write the Kelvin force in terms of $\balpha$ and $\bphi$ as follows
\begin{align}\label{eq:F_phi}
\nabla|\bH|^2&=\nabla\left(\balpha(t)^\top(\mR(\bphi)^\top\mR(\bphi))\balpha(t)\right) \nonumber \\
&=\begin{pmatrix}\balpha(t)^\top\partial_{x_1}(\mR(\bphi)^\top\mR(\bphi))\balpha(t)\\
\vdots
\\
\balpha(t)^\top\partial_{x_\dd}(\mR(\bphi)^\top\mR(\bphi))\balpha(t)\end{pmatrix}
=\begin{pmatrix}\balpha^\top\mathbb{G}_1(\bphi)\balpha\\
\vdots
\\
\balpha^\top\mathbb{G}_\dd(\bphi)\balpha\end{pmatrix}
\end{align}
with $\mathbb{G}_k(\bphi):=\partial_{x_k}(\mR^\top(\bphi)\mR(\bphi))
\in [\rH^1(0,T;C^{\infty}(\overline{\Omega}))]^{{\nd}\times {\nd}}$, $k=1,\ldots,\dd$.
 We will show existence
of solution to problems \eqref{eq:PJ1} and \eqref{eq:PJ2} via a minimizing sequence argument
(see, for instance, \cite{book:T10}).
The feasibility of the dipole approximation can be studied with the first term of $\calJ(\bH(\balpha,\btheta))$
and $\calF(\bH(\balpha,\bphi))$, because
the computed Kelvin force can be used to control a desired concentration. On the other hand, the last two terms
in \eqref{eq:calJ} (and \eqref{eq:calF}) will enforce a smooth evolution of the intensities and angles
(and positions).
%

\subsection{Existence and optimality conditions}

Next, we focus on the existence of a solution to  the minimization problem \eqref{eq:PJ1},
which is nonconvex.
For notational simplicity, from now on we denote by $\mathrm{V}$ both a Banach space $\mathrm{V}$ and
the Banach tensor product  $\mathrm{V}^{\nd}$.
\begin{theorem}[existence of minimizers]\label{thm:exist_min}
There exist at least one solution $(\bbalpha,\bbtheta)$ to the minimization problem \eqref{eq:PJ1}.
\end{theorem}
\begin{proof}
We apply  the direct method of the calculus of variations. Given that $\calJ$
is bounded below by zero, we deduce that $j = \inf_{(\balpha,\btheta) \in \calH_{ad}\times\calV_{ad}}
\calJ(\balpha,\btheta)$ is finite.
We can thus construct a minimizing sequence $\{ (\balpha_n,\btheta_n) \}_{n\in \mathbb{N}}$ such that
\[
    j = \lim_{n\rightarrow \infty} \calJ(\balpha_n,\btheta_n).
\]
As the sequence $\{ (\balpha_n,\btheta_n) \}_{n\in \mathbb{N}}$ is uniformly bounded in
$\calH_{ad}\times\calV_{ad}\subset [\rH^1(0,T)]^2$,
we can extract a (not relabeled) weakly
convergent subsequence $\{ \balpha_{n} \}_{n\in \N}$ such that
\begin{equation}\label{eq:ex_1}
    \balpha_{n} \weakto \bbalpha \quad \mbox{in } \rH^1(0,T),
                \quad \bbalpha \in \calH_{ad}\quad\mbox{and} \quad
    \btheta_{n} \weakto \bbtheta \quad \mbox{in } \rH^1(0,T),
                \quad \bbtheta \in \calV_{ad} .
\end{equation}
Moreover, according to \cite[Theorem~9.16]{B11} we have
\begin{align}\label{eq:ex_2}
    \balpha_{n} \to \bbalpha \quad \mbox{and }\quad
    \btheta_{n} \to \bbtheta \quad \mbox{in }
C([0,T]).
\end{align}
%
To show the optimality of $\bar{\balpha}$, we first consider \eqref{eq:ex_2} to get that
\begin{align}\label{eq:PJ1_1conv}
\int_0^T\|\nabla|\bH(\balpha_n,\btheta_n)|^2-\rmvel_{i}\|^2_{\rL^2(D_t)}
\to \int_0^T\|\nabla|\bH(\bbalpha,\bbtheta)|^2-\rmvel_{i}\|^2_{\rL^2(D_t)}  \qquad i=1,\ldots,\dd,
\end{align}
which follows  from \eqref{eq:F_theta} and the fact that
\begin{multline*}
\mathbb{B}_k(\btheta_n)_{i,j}-\mathbb{B}_k(\bbtheta)_{i,j}=
 (\widehat{\bv{d}}^\top_i(\btheta_{n}(i))-\widehat{\bv{d}}^\top_i(\bbtheta(i)))
 \left(\partial_{x_k}(\mathbb{M}_i(\bx)^\top\mathbb{M}_j(\bx))\right)
 \widehat{\bv{d}}_j(\btheta_{n}(j)) \\
  \widehat{\bv{d}}^\top_i(\bbtheta(i))\left(\partial_{x_k}(\mathbb{M}_i(\bx)^\top\mathbb{M}_j(\bx))\right)
 (\widehat{\bv{d}}_j(\btheta_n(j))-\widehat{\bv{d}}_j(\bbtheta(j))).
\end{multline*}
and
$\widehat{\bv{d}_i}(\btheta_{n}(i)) \to \widehat{\bv{d}_i}(\bbtheta(i))$ in $[C([0,T])]^{\dd}$,
$i=1,\ldots,\nd$. Notice that here we have used, for instance, $\btheta(i)$ to denote the $i$-th component of vector $\btheta$.
This and the fact that the last terms in $\calJ$
are weakly lower semicontinuous (see~\cite[Theorem 2.12]{book:T10}) yields
\begin{align*}
\min_{(\balpha,\btheta)\in \calH_{ad}\times\calV_{ad}} \calJ(\balpha,\btheta) = \liminf_{n\to\infty} \calJ(\balpha_n,\btheta_n)
  \geq\calJ(\bbalpha,\bbtheta),
\end{align*}
which concludes the proof.
\end{proof}

By applying the same technique as above and using that
$\brho$ belongs to $C^1([0,T])$, we obtain the following existence result for problem \eqref{eq:PJ2}.
\begin{theorem}[existence of minimizers]\label{thm:exist_min2}
There exist at least one solution $(\bbalpha,\bbphi)$ to the minimization problem \eqref{eq:PJ2}.
\end{theorem}
%

\subsection{First order optimality conditions}

To state first order optimality conditions we follow standard arguments
because of the fact that $\calJ : [\rH^1(0,T)]^2 \rightarrow \mathbb{R}$ is Fr\'echet differentiable.
We first notice that, in view of \eqref{eq:h_mag_1}, we can also rewrite $\bH$ as
\[
\bH(\balpha,\btheta)=\sum_{i=1}^{\nd}\mathbb{\bv{A}}(\alpha_j)\widehat{\bv{d}}_j(\theta_j)
\]
where $\mathbb{\bv{A}}(\alpha_i)\in [\rH^1(0,T;C^{\infty}(\overline{\Omega}))]^{\dd\times \dd}$
and is given by
\begin{equation*}
\mathbb{\bv{A}}(\alpha_i):=\alpha_i\del{\dd\dfrac{(\bx-\bx_i)(\bx-\bx_i)^\top}{|\bx-\bx_i|^2}-\mathbb{I}}
\dfrac{1}{|\bx-\bx_i|^\dd} , \quad i=1,\ldots, \nd.
\end{equation*}
Thus, the Kelvin force in terms of $\balpha$ and $\btheta$ can be written as follows
\begin{align*}
\nabla|\bH|^2 &=\sum_{m,n=1}^{\nd}\nabla\left(\widehat{\bv{d}}^\top_m(\theta_m)
(\mathbb{\bv{A}}^\top(\alpha_m)\mathbb{\bv{A}}(\alpha_n))\widehat{\bv{d}}_n(\theta_n)\right)
\\
&=\sum_{m,n=1}^{\nd}
\begin{pmatrix}\widehat{\bv{d}}^\top_m(\theta_m)\mathbb{Q}_1(\alpha_m,\alpha_n)\widehat{\bv{d}}_n(\theta_n)\\
\vdots
\\
\widehat{\bv{d}}^\top_m(\theta_m)\mathbb{Q}_\dd(\alpha_m,\alpha_n)\widehat{\bv{d}}_n(\theta_n)\end{pmatrix}
\end{align*}
with
$\mathbb{Q}_k(\alpha_m,\alpha_n):=\partial_{x_k}(\mathbb{\bv{A}}^\top(\alpha_m)\mathbb{\bv{A}}(\alpha_n))
\in [\rH^1(0,T;C^{\infty}(\overline{\Omega}))]^{\dd\times \dd}$ for $k=1,\ldots,\dd$.

\begin{theorem}[first order optimality condition for $\calJ$]\label{var_ineq}
If $(\bbalpha,\bbtheta)\in \calH_{ad}\times \calV_{ad}$  denotes an optimal control, given
by Theorem~\ref{thm:exist_min}, then the first order necessary optimality condition satisfied by
$(\bbalpha,\bbtheta)$ is
\[
\calJ'(\bbalpha,\bbtheta)\pair{\balpha - \bbalpha,\btheta-\bbtheta}\geq 0\qquad \forall (\balpha,\btheta)
\in \calH_{ad}\times \calV_{ad}
\]
where, for $(\delta\balpha,\delta\btheta)=(\balpha - \bbalpha,\btheta-\bbtheta)$ we have
 \begin{align*}
 \calJ'(\bbalpha,\bbtheta)&\pair{\delta\balpha,\delta \btheta}=\int_0^T \Bigg(\sum_{i=1}^\dd \int_{D_t}
 \left(\bbalpha^\top\mathbb{B}_i(\bbtheta)\bbalpha-\rmvel_{i}\right)
 \left(2\bbalpha^\top\mathbb{B}_i(\bbtheta)\delta\balpha\right) d\bx \\
 &+\sum_{i=1}^\dd \int_{D_t}
 \left(\bbalpha^\top\mathbb{B}_i(\bbtheta)\bbalpha-\rmvel_{i}\right)
 \left(2\sum_{m,n=1}^{\nd}(\widehat{\bv{d}}^\top_m)'(\theta_m)\mathbb{Q}_i(\alpha_m,\alpha_n)
 \widehat{\bv{d}}_n(\theta_n)\delta\theta_m\right) d\bx\\
 &+\lambda \partialt\bbalpha^\top \partialt\delta\balpha+\eta \partialt\bbtheta^\top \partialt\delta\btheta  \Bigg) dt.
 \end{align*}
\end{theorem}
Similarly, we obtain the first order optmimality conditions for $\calF$.

\begin{theorem}[first order optimality condition for $\calF$]\label{var_ineqF}
If $(\bbalpha,\bbphi)\in \calH_{ad}\times \calU_{ad}$  denotes an optimal control, given
by Theorem~\ref{thm:exist_min2}, then the first order necessary optimality condition satisfied by
$(\bbalpha,\bbphi)$ is
\[
\calF'(\bbalpha,\bbphi)\pair{\balpha - \bbalpha,\bphi-\bbphi}\geq 0\qquad \forall (\balpha,\bphi)
\in \calH_{ad}\times \calV_{ad}
\]
where, for $(\delta\balpha,\delta\bphi)=(\balpha - \bbalpha,\bphi-\bbphi)$ we have
\begin{align*}
\calF'(\bbalpha,\bbphi)\pair{\delta\balpha,\delta \bphi}=&\int_0^T \Bigg(\sum_{i=1}^\dd \int_{D_t}
\left(\bbalpha^\top\mathbb{G}_i(\bbphi)\bbalpha-\rmvel_{i}\right)
\left(2\bbalpha^\top\mathbb{G}_i(\bbphi)\delta\balpha\right) d\bx \\
&+\sum_{i=1}^\dd \int_{D_t}
\left(\bbalpha^\top\mathbb{G}_i(\bbphi)\bbalpha-\rmvel_{i}\right)
\left(\sum_{l=1}^{\nd}\bbalpha^\top\mathbb{G}_i(\bphi)'\pair{\delta\bphi_l}\bbalpha\right) d\bx\\
&+\lambda \partialt\bbalpha^\top \partialt\delta\balpha+\beta \partialt\bbphi^\top \partialt\delta\bphi  \Bigg) dt.
\end{align*}
\end{theorem}

\begin{remark}[unknown final time]
In some applications an important quantity to account for is the time it takes for
$D_t$ to arrive at its final destination. In this case, the final time is an unknown. 
This problem can be handled, following \cite{KS2011}, by reformulating
the minimization in terms of the arc length which has a fixed final value.
This approach can be applied to problems \eqref{eq:PJ1} and \eqref{eq:PJ2},
for details see \cite{ANV2016}.
\end{remark}

\section{Discretization}\label{s:disc}

In this section we propose and study a numerical approximation of
minimization problems \eqref{eq:PJ1} and \eqref{eq:PJ2}. 
We introduce a
parametrization of the moving domain $D_t$
in terms of a fixed domain $\hD\subset \R^\dd$ (see also \cite{ANV2016}).   With this in mind,
we define a map $\bX: [0,T] \times \overline{\hD}\to\overline{\O}$, such that for all $t\in [0,T]$
\begin{align*}
\bX(t,\cdot):&\overline{\hD}\to \overline{D}_t\\
      &\hbx\to \bx = \bX(t,\hbx),
\end{align*}
is a one-to-one correspondence which satisfies $\bX(t,\hD)=D_t$. For simplicity, we assume
\[
\bX(t,\hbx)=\bvarphi(t)+\psi(t)\hbx,
\]
where
\begin{equation}\label{eq:psi_reg}
\bvarphi:[0,T]\to \R^\dd,\,\quad \psi:[0,T]\to (0,+\infty),\quad \bvarphi, \psi \in \rH^1(0,T)
\end{equation}
are functions such that
$\bvarphi(0)=\boldsymbol{0}$ and $\psi(0)=1$, namely $\hD=D_{0}$.
In case $\psi(t)=\widehat{\psi}\in \R^+$ for all $t\in [0,T]$, $\bvarphi$
can be viewed as a parameterization of a desired path that the scaled domain $\widehat{\psi}\widehat{D}$
traverses from an initial position to a final position.

\subsection{Problem 1: Controlling intensities and directions}\label{s:contr_dir}

We consider the previous parametrization and rewrite the minimization problem
\eqref{eq:PJ1} in terms of $\hD$.
 Therefore, for $i=1,\dots, \dd$
\begin{align*}
\int_0^T\int_{D_t}&\left(\nabla |\bH(\balpha,\btheta)|^2-\bvel(t,\bx)\right)^2 d\bx\, dt\\
=&\sum_{i=1}^{\dd}\int_0^T\int_{D_t}\left(\balpha(t)^\top\mathbb{B}_i(\btheta,\bx)\balpha(t)-\rmvel_{i}(t,\bx)\right)^2 d\bx\, dt\\
=&
\sum_{i=1}^{\dd}\int_0^T\int_{\hD}\left(\balpha(t)^\top\mathbb{B}_i(\btheta,\bX(t,\hbx))\psi(t)^{\dd/2}\balpha(t)-\rmvel_{i}(t,\bX(t,\hbx))\psi(t)^{\dd/2}\right)^2 d\hbx\, dt,
\end{align*}
because $\det(\nabla_{\hbx}\bX(t,\hbx))=\psi(t)^\dd$.
Then, we rewrite  $\calJ$ as
\begin{align}\label{eq:calJ_D}
\calJ(\balpha,\btheta)=&\dfrac{1}{2}\sum_{i=1}^\dd\int_0^T\|\balpha^\top\hbP_i(\btheta)\balpha-\hbvel_i\|^2_{\rL^2(\hD)}
+\dfrac{\lambda}{2}\int_0^T|\partialt\balpha|^2+\dfrac{\eta}{2}\int_0^T|\partialt\btheta|^2\nonumber\\
=&\calJ^1(\balpha,\btheta)+\calJ^2(\balpha)+\calJ^3(\btheta)
\end{align}
with $\hbP_i(\btheta,\hbx):=\mathbb{B}_i(\btheta,\bX(t,\hbx))\psi(t)^{\dd/2}$
and $\hbvel_{i}(t,\hbx)=\rmvel_{i}(t,\bX(t,\hbx))\psi(t)^{\dd/2}$, $i=1,\dots, \dd$.

Next, we introduce a time discretization of problem \eqref{eq:PJ1} upon using $\calJ$ as defined in \eqref{eq:calJ_D}.
Let us fix $N\in \N$ and let $\Dt := T/N$ be the time step.
Now, for $n = 1,\ldots, N$, we define $t^n := n\Dt$,
$\hbP^n_i(\cdot)=\hbP_i(\cdot,\bX(t^n,\hbx))$ and $\hbvel^n_{i}$ to be
\begin{equation}\label{eq:discr_v}
\hbvel^n_{i}(\cdot)=\dfrac{1}{\tau}\int_{t^{n-1}}^{t^n}\hbvel_{i}(t,\cdot)dt,\qquad i=1,\ldots,\dd,
\end{equation}
which in turn allows us  to incorporate a general $\bvel$.
Then we consider a time discrete version of \eqref{eq:PJ1}: given
the initial condition $(\balpha_0,\btheta_0)=:(\balpha_\tau(0),\btheta_\tau(0))$,
find $(\bbalpha_\tau,\bbtheta_\tau) \in \calH^\tau_{ad}\times \calV^\tau_{ad}$ solving
\begin{align}
\label{eq:PJ1_disc}
(\bbalpha_\tau,\bbtheta_\tau)=\argmin_{(\balpha_\tau,\btheta_\tau)\in\calH^\tau_{ad}\times \calV^\tau_{ad} } \calJ_\tau(\balpha_\tau,\btheta_\tau),
\end{align}
where
\begin{equation*}
\calJ_\tau(\balpha_\tau,\btheta_\tau) := \calJ^1_\tau(\balpha_\tau,\btheta_\tau)
+\calJ^2_\tau(\balpha_\tau)+\calJ^3(\btheta_\tau)
\end{equation*}
and
\begin{align*}
&\calJ^1_\tau(\balpha_\tau,\btheta_\tau) :=
\Dt\sum_{n=1}^N\dfrac{1}{2}\sum_{i=1}^\dd\|(\balpha^n_\tau)^\top
\hbP^n_i(\btheta^n_\tau)\balpha^n_\tau-\hbvel^n_{i}\|^2_{\rL^2(\hD)} \\
&\calJ^2_\tau(\balpha_\tau) :=
\Dt\sum_{n=1}^N\dfrac{\lambda}{2\Dt^2}|\balpha^n_\tau-\balpha^{n-1}_\tau|^2, \qquad
\calJ^3_\tau(\btheta_\tau) :=
\Dt\sum_{n=1}^N\dfrac{\eta}{2\Dt^2}|\btheta^n_\tau-\btheta^{n-1}_\tau|^2
\end{align*}
and $\calH^\tau_{ad}=\mathbb{P}^1_\tau\cap \calH_{ad}$, $\calV^\tau_{ad}=\mathbb{P}^1_\tau\cap \calV_{ad}$
with
$\mathbb{P}^1_\tau := \{\bv{v} \in C[0,T] : \bv{v}|_{[t^{n-1},t^n]}\in \mathbb{P}^1, \ n = 1, \dots, N\}$. 
Hereafter $\mathbb{P}^1$ is the space of polynomials of degree at most 1.
Moreover, by applying the same arguments of Theorem~\ref{thm:exist_min}, it follows that
there exists $(\bbalpha_\tau,\bbtheta_\tau)\in \calH^\tau_{ad}\times\calV^\tau_{ad}$ a solution to problem
\eqref{eq:PJ1_disc}.

Next we prove the convergence of the discrete problem to a minimizer
of the continuous problem.
Such a proof is motivated by $\Gamma$-convergence theory \cite{Maso93,Braides2013}.
\begin{theorem}[Problem 1: convergence to a minimizer]\label{thm:Gconv}
The family of  minimizers $\left\{(\bbalpha_{\tau},\bbtheta_{\tau})\right\}_{\tau>0}$ to
\eqref{eq:PJ1_disc} is uniformly bounded in  $[\rH^1(0,T)]^2$
and it contains a subsequence that converges weakly to
$(\bbalpha,\bbtheta)$ in $[\rH^1(0,T)]^2$, a  solution to the minimization problem \eqref{eq:PJ1},
and $\lim_{\tau\to 0}\calJ_\tau(\bbalpha_\tau,\bbtheta_\tau)=\calJ(\bbalpha,\bbtheta)$.
\end{theorem}
\begin{proof}
We proceed is several steps.
\begin{enumerate}[1.-]
\item \textit{Boundedness of $\left\{(\bbalpha_\tau,\bbtheta_\tau)\right\}_{\tau>0}$ in $[\rH^1(0,T)]^2$:}
 This follows immediately from the fact that $(\bbalpha_{\tau},\bbtheta_{\tau})$ minimizes
$\calJ_{\tau}$ and $\lambda,\eta>0$: given that the constant function
$(\balpha_0(t),\btheta_0(t))=(\balpha_0,\btheta_0)$  belongs to $\calH^\tau_{ad}$, we have
\[
\calJ_{\tau}(\bbalpha_{\tau},\bbtheta_{\tau})
\leq \calJ_{\tau}(\balpha_0,\btheta_0)\leq
C\bigg(|\balpha_0|^4\Big(\sum_{i=1}^{\nd}|\widehat{\bv{d}}_i(\btheta_0(i) )|^4\Big)+\|\bvel\|^2_{\rL^2(0,T;\rL^2(\O))}\bigg).
\]
This implies the existence of a (not relabeled) weakly convergent subsequence such that
$(\bbalpha_\tau,\bbtheta_\tau) \weakto (\bbalpha,\bbtheta)$  in $[\rH^1(0,T)]^2$ and
$(\bbalpha,\bbtheta) \in \calH_{ad}\times\calV_{ad}$.
It remains to prove that $(\bbalpha,\bbtheta)$ solves \eqref{eq:PJ1} and
$\lim_{\tau\to 0}\calJ_\tau(\bbalpha_\tau,\bbtheta_\tau)=\calJ(\bbalpha,\bbtheta)$.
\item \textit{Lower bound inequality}:
We show that
\begin{equation}\label{eq:Lower_J2}
\calJ(\balpha,\btheta)\leq \underset{\tau\to 0}{\lim \inf}\calJ_\tau(\balpha_\tau,\btheta_\tau)
\end{equation}
for all $\left\{(\balpha_\tau,\btheta_\tau)\right\}_{\tau>0}\subset \calH^\tau_{ad}\times\calV^\tau_{ad}$
converging to $(\balpha,\btheta)$  weakly in $[\rH^1(0,T)]^2$.
Consequently $(\balpha_\tau,\btheta_\tau)\to(\balpha,\btheta)$ strongly in $[\rL^2(0,T)]^2$ for a
subsequence (not relabeled). Let $\overline{\Pi}_\tau$ be the piecewise constant interpolation
operator at the nodes $\left\{t^n\right\}_{n>0}$. Then by straightforward computations it follows that
\begin{equation*}
    \| \overline{\Pi}_\tau\balpha_\tau - \balpha \|_{\rL^2(0,T)}
    \rightarrow 0\quad \mbox{and}\quad \| \overline{\Pi}_\tau\btheta_\tau - \btheta \|_{\rL^2(0,T)}
    \rightarrow 0.
\end{equation*}
In view of the smoothness of $\widehat{\bv{d}}_i$, $i=1,\ldots, \nd$,
then  $\hbP_j(\overline{\Pi}_\tau\btheta_\tau) \rightarrow \hbP_j(\btheta)$ in $\rL^2(0,T; \rL^2(\widehat{D}))$,
 $j = 1, \dots, d$. Collecting these results and making use of the
bounds $\| \balpha_\tau\|_{\rL^\infty(0,T)}\leq \balpha^*$
 and $\| \btheta_\tau\|_{\rL^\infty(0,T)}\leq \btheta^*$
 we obtain (cf~\eqref{eq:PJ1_1conv})
\begin{equation*}
\overline{\Pi}_\tau\balpha_\tau^\top\hbP_i(\overline{\Pi}_\tau\btheta_\tau)\overline{\Pi}_\tau\balpha_\tau
\to\balpha^\top\hbP_i(\btheta)\balpha\quad \mbox{in  } \rL^2(0,T; \rL^2(\widehat{D})) ,
 \end{equation*}
which in conjunction with the regularity of $\bvel$ leads to 
\begin{equation}\label{eq:Low_bnd_1}
 \calJ^1_\tau(\balpha_\tau,\btheta_\tau)\to\calJ^1(\balpha,\btheta).
\end{equation}
On the other hand, given that $\partialt\balpha_\tau$ converges weakly to
 $\partialt\balpha$ in $\rL^2(0,T)$,
from the weak lower semi-continuity of the semi-norm it follows that
\begin{equation}\label{eq:Low_bnd_2}
\calJ^2(\balpha)+\calJ^3(\btheta)\leq \underset{\tau\to 0}{\lim \inf}\calJ^2_\tau(\balpha_\tau)
+ \underset{\tau\to 0}{\lim \inf}\calJ^3_\tau(\btheta_\tau).
\end{equation}
From \eqref{eq:Low_bnd_1} and \eqref{eq:Low_bnd_2} we conclude
\begin{equation*}
\calJ(\balpha,\btheta)\leq \underset{\tau\to 0}{\lim }\calJ^1_\tau(\balpha_\tau,\btheta_\tau)
+\underset{\tau\to 0}{\lim \inf}\calJ^2_\tau(\balpha_\tau)
+\underset{\tau\to 0}{\lim \inf}\calJ^3_\tau(\btheta_\tau)\leq \underset{\tau\to 0}{\lim \inf}\calJ_\tau(\balpha_\tau).
\end{equation*}
\item \textit{Existence of a recovery sequence}:
Let $(\balpha, \btheta)\in\calH_{ad}\times\calV_{ad}$ be given. Then, the piecewise linear
Lagrange interpolant $(\Pi_{\Dt}\balpha, \Pi_{\Dt}\btheta)$ of $(\balpha, \btheta$)
belongs to $(\calH_{ad}^\tau,\calV_{ad}^\tau)$. Since
$\partialt(\Pi_{\Dt}\balpha) \rightarrow \partialt\balpha$,
$\partialt(\Pi_{\Dt}\btheta) \rightarrow \partialt\btheta$
and $\Pi_\tau\balpha^\top_{\Dt}\hbP_i(\Pi_{\Dt}\btheta_\tau)\Pi_{\Dt}\balpha_\tau\to\balpha^\top\hbP_i(\btheta)\balpha$
in $\rL^2(0,T;\rL^2(\hD))$ because of the box constrains, we obtain
\begin{multline*}
 \underset{\tau\to 0}{\limsup}\calJ_\tau(\Pi_{\Dt}\balpha,\Pi_{\Dt}\btheta)\\
 \le\underset{\tau\to 0}{\limsup}\calJ^1_\tau(\Pi_{\Dt}\balpha,  \Pi_{\Dt}\btheta) +
 \underset{\tau\to 0}{\limsup}\calJ^2_\tau(\Pi_{\Dt}\balpha)
 +\underset{\tau\to 0}{\limsup}\calJ^3_\tau(\Pi_{\Dt}\btheta)\leq\calJ(\balpha,\btheta).
\end{multline*}
\item \textit{$(\bbalpha,\bbtheta)$ is a minimizer for problem \eqref{eq:PJ1} :} We need to show
\begin{equation}\label{balpha_min}
 \calJ(\bv{v}, \bv{w})\geq \calJ(\bbalpha,\bbtheta) \qquad \forall (\bv{v},\bv{w})\in \calH_{ad}\times\calV_{ad} .
\end{equation}
From step 3 there exists $\{ (\bv{v}_\tau,\bv{w}_\tau) \}_{\tau>0}$ such that
$\bv{v}_\tau \rightharpoonup \bv{v}$ and $\bv{w}_\tau \rightharpoonup \bv{w}$ in $\rH^1(0,T)$ and
\begin{align*}
\calJ(\bv{v},\bv{w})
& \ge\underset{\tau\to 0}{\lim \sup}\calJ_\tau(\bv{v}_\tau,\bv{w}_\tau)
\\
& \geq \underset{\tau\to 0}{\liminf} \calJ_\tau(\bv{v}_\tau,\bv{w}_\tau)
\geq \underset{\tau\to 0}{\liminf} \calJ_\tau(\bbalpha_\tau,\bv{w}_\tau)\geq \calJ(\bbalpha,\bbtheta)
\end{align*}
where we have used that $\bbalpha_\tau$ is a minimizer for
$\calJ_\tau$ together with \eqref{eq:Lower_J2}.
\item  \textit{Convergence:} Since $(\bbalpha_\tau,\bbtheta_\tau)$ is a minimizer we deduce
the inequality $\calJ_\tau(\bbalpha_\tau,\bbtheta_\tau)\leq \calJ_\tau(\Pi_\tau\bbalpha,\Pi_\tau\bbtheta)$, whence
applying first step 3 and next step 2 we see that
\[
\underset{\tau\to 0}{\lim\sup}\calJ_\tau(\bbalpha_\tau,\bbtheta_\tau)
\leq\underset{\tau\to 0}{\lim\sup}\calJ_\tau(\Pi_\tau\bbalpha,\Pi_\tau\bbtheta)
\leq \calJ(\bbalpha,\bbtheta)
\leq \underset{\tau\to 0}{\liminf} \calJ_\tau(\bbalpha_\tau,\bbtheta_\tau).
\]
This implies $\lim_{\tau\to 0}\calJ_\tau(\bbalpha_\tau,\bbtheta_\tau)=\calJ(\bbalpha,\bbtheta)$. In addition
 $\left\{\bbalpha_{\tau},\bbtheta_{\tau}\right\}_{\tau>0}$ converges $\rL^2$-strongly and
 $\rH^1$-weakly to $(\bbalpha,\bbtheta)$, a  minimizer of \eqref{eq:PJ1}.
\end{enumerate}
This concludes the proof.
\end{proof}

\subsection{Problem 2: Controlling intensities and positions}\label{s:contr_pos}

Like in the previous section,  we consider the fixed domain $\hD$
and  rewrite  $\calF$ (cf.~\eqref{eq:calF}) as
\begin{align}\label{eq:calF_D}
\calF(\balpha,\bphi)=&\dfrac{1}{2}\sum_{i=1}^\dd\int_0^T\|\balpha^\top
\widehat{\mathbb{G}}_i(\bphi)\balpha-\hbvel_{i}\|^2_{\rL^2(\hD)}
+\dfrac{\lambda}{2}\int_0^T|\partialt\balpha|^2+\dfrac{\beta}{2}\int_0^T|\partialt\bphi|^2
\end{align}
with $\widehat{\mathbb{G}}_i(\bphi,\hbx):=\mathbb{G}_i(\bphi,\bX(t,\hbx))\psi(t)^{\dd/2}$, $i=1,\dots, \dd$.

Next, we consider the previous definition $\calF$ in order to introduce a time
discretization of  problem \eqref{eq:PJ2}. Let us fix $N\in \N$ and let $\Dt := T/N$ be the time step.
Then we consider the following time discrete version of \eqref{eq:PJ2}: given
the initial condition $(\balpha_0,\bphi_0)=:(\bbalpha_\tau(0),\bbphi_\tau(0))$,
find $(\bbalpha_\tau,\bbphi_\tau) \in \calH^\tau_{ad}\times\calU^\tau_{ad}$ solving
\begin{align}
\label{eq:PJ2_disc}
(\bbalpha_\tau,\bbphi_\tau)=\argmin_{(\balpha_\tau,\bphi_\tau)\in\calH^\tau_{ad}\times \calU^\tau_{ad} }
\calF_\tau(\balpha_\tau,\btheta_\tau),
\end{align}
where $\calU^\tau_{ad}=\mathbb{P}^1_\tau\cap \calU_{ad}$,
\[
\calF_\tau(\balpha_\tau,\btheta_\tau) :=\calF^1_\tau(\balpha_\tau,\bphi_\tau)
+\calF^2_\tau(\balpha_\tau)+\calF^3\bphi_\tau),
\]
and
\begin{align*}
&\calF^1_\tau(\balpha_\tau,\btheta_\tau) :=
\Dt\sum_{n=1}^N\dfrac{1}{2}\sum_{i=1}^\dd\|(\balpha^n_\tau)^\top
\widehat{\mathbb{G}}^n_i(\btheta^n_\tau)\balpha^n_\tau-\hbvel^n_{i}\|^2_{\rL^2(\hD)} \\
&\calF^2_\tau(\balpha_\tau) :=
\Dt\sum_{n=1}^N\dfrac{\lambda}{2\Dt^2}|\balpha^n_\tau-\balpha^{n-1}_\tau|^2, \qquad
\calF^3_\tau(\btheta_\tau) :=
\Dt\sum_{n=1}^N\dfrac{\eta}{2\Dt^2}|\bphi^n_\tau-\bphi^{n-1}_\tau|^2.
\end{align*}
Moreover, it is straightforward to prove the following convergence result.
\begin{theorem}[Problem 2: convergence to minimizers]
The family of  minimizers $\left\{(\bbalpha_{\tau},\bbphi_{\tau})\right\}_{\tau>0}$ to
\eqref{eq:PJ2_disc} is uniformly bounded in  $[\rH^1(0,T)]^2$
and it contains a subsequence that converges weakly to
$(\bbalpha,\bbphi)$ in $[\rH^1(0,T)]^2$, a solution to the minimization problem \eqref{eq:PJ2},
and $\lim_{\tau\to 0}\calF_\tau(\bbalpha_\tau,\bbphi_\tau)=\calF(\bbalpha,\bbphi)$.
\end{theorem}

\section{Numerical example}\label{sec:num_dip}

In this section we illustrate the performance of the proposed discrete
schemes  \eqref{eq:PJ1_disc} and \eqref{eq:PJ2_disc}.
We consider the approximation of two constant vector fields $\bvel$  on $D_t$.
For all examples $\Omega\subset \R^2$ is a ball of unit radius centered at
$(0,0)$ and the dipoles belong to a ball of  radius 1.2 centered at $(0,0)$.

To solve the minimization problem, we use projected BFGS with Armijo  line search
\cite{K1999}; alternative strategies such as semi-smooth Newton \cite{HIK2002} can be immediately applied as well.
The algorithm terminates when the $l^2$-norm of the projected gradient
is less or equal to $10^{-6}$. In addition to the initial condition $(\balpha_0,\btheta_0)\in \R^{2\nd}$
for problem \eqref{eq:PJ1_disc}
or $(\balpha_0,\bphi_0)\in \R^{2\nd}$ for \eqref{eq:PJ2_disc}, the optimization algorithm requires
an initial guess $(\bbalpha,\bbtheta)$ and $(\bbalpha,\bbphi)\in \R^{2N\nd}$, for \eqref{eq:PJ1_disc} and \eqref{eq:PJ2_disc}, respectively.
Following \cite{ANV2016} we propose below Algorithm \ref{alg:algorithm_1} in order to compute an initial guess
$(\balphaI,\bthetaI)\in \calH^\tau_{ad}\times\calV^\tau_{ad}$ for problem \eqref{eq:PJ1_disc}.
A similar strategy can be used for problem \eqref{eq:PJ2_disc}.

Given $\bv{a}_*,\bv{a}^*, \bv{b}_*, \bv{b}^*\in \R^{2N\nd}$ , by $\mbox{Proj}_{[\bv{a}_*,\bv{a}^*, \bv{b}_*, \bv{b}^*]}$
 we denote pointwise projection on the interval $[\bv{a}_*,\bv{a}^*]\times[\bv{b}_*,\bv{b}^*]$
\[
\mbox{Proj}_{[\bv{a}_*,\bv{a}^*,\bv{b}_*,\bv{b}^*]}(\mathbf{x},\mathbf{y})
=\min\left\{(\bv{a}^*,\bv{b}^*),\max\left\{(\mathbf{x},\mathbf{y}),(\bv{a}_*,\bv{b}_*)\right\}\right\},
\]
where $\min$ and $\max$ are interpreted componentwise.
Notice that the above algorithm, related to model predictive control (MPC),
is the so-called finite horizon model
predictive or instantaneous optimization algorithm  \cite{MR1233904,MR1708955,HK1999,AHNSWsub}:
the minimization takes place over one time step only.
This cheaper strategy provide us an ``accurate" initial guess to solve
 \eqref{eq:PJ1_disc}.
\begin{algorithm}[!h]
\caption{: Initialization algorithm}
\begin{algorithmic}[1]

\STATE $\textbf{Input:} \, \balpha_0,\, \btheta_0,\,\balpha_*,\,\btheta_*,\,\balpha^*,\,\btheta^*,\, \lambda,\,\Ds,\, \eta,\,
\hD,\, \verb"tol",\, \hbP^n_i,\, \hbvel^n_{i},\, n=1,\ldots, N,\, i=1,\ldots,\dd $
\STATE Set $(\mathbf{x}^0,\mathbf{y}^0):=(\balpha_0,\btheta_0)$
\FOR{$n=1,\ldots,N$}
    \STATE Solve for $(\mathbf{x},\mathbf{y})\in \R^{2N\nd}$
    \[
    \displaystyle \underset{(\balpha_*,\btheta_*)\leq(\mathbf{x},\mathbf{y})\leq(\balpha^*,\btheta^*)}{\min_{(\mathbf{x},\mathbf{y})\in\R^{2N\nd}}}
    F(\mathbf{x},\mathbf{y})
    \]
    \[
        F(\mathbf{x},\mathbf{y})=
\dfrac{1}{2}\sum_{i=1}^\dd\|\mathbf{x}^\top\hbP^n_i(\mathbf{y})\mathbf{x}-\hbvel^n_{i}\|^2_{\rL^2(\hD)}
+\dfrac{\lambda}{2\Ds^2}|\mathbf{x}-\mathbf{x}^{0}|^2
+\dfrac{\eta}{2\Ds^2}|\mathbf{y}-\mathbf{y}^{0}|^2
\]
 with termination criterion: $|(\mathbf{x}, \mathbf{y})-\mbox{Proj}_{[\balpha_*,\balpha^*,\btheta_*,\btheta^*]}
 ((\mathbf{x}, \mathbf{y})-\nabla F(\mathbf{x},\mathbf{y}))| < \verb"tol"$.
\STATE $\balphaI(n\Ds)=\mathbf{x}$,\, $\bthetaI(n\Ds)=\mathbf{y}$
\STATE $\mathbf{x}^0 \gets \mathbf{x}$,\, $\mathbf{y}^0 \gets \mathbf{y}$
\ENDFOR
\end{algorithmic}
\label{alg:algorithm_1}
\end{algorithm}
%

\subsection{Approximation of a vector field}\label{sec:example1}%

In this section we focus on the approximation of two vectors fields
 $\bvel$ defined on moving domains $D_t$.
 With this aim
we solve the discrete minimization problems \eqref{eq:PJ1_disc} and \eqref{eq:PJ2_disc}.
We will achieve our goal by creating a magnetic force which allows us to ``magnetically inject" nanoparticles inside the domain. 
  We first consider the approximation of a  constant vector field
  $\bvel_1(\bx,t) = (\frac{1}{\sqrt{2}},-\frac{1}{\sqrt{2}})^\top$ by controlling the  directions and intensities
  of the $4$ dipoles.
  For the second example we are interested in the approximation of the field given
  by $\bvel_2(\bx,t)= (1, 0)^\top$ by controlling the positions and intensities
  of $3$ dipoles.
  The moving domains $D_{i,t}$ are such that $D_{i,t}=\bx(t,\hD)$, with
$\bx(t,\hbx)=\bvarphi_i(t)+\hbx$, $i=1,2$.  Here $\bvarphi_1(t)=\frac{0.6}{T}(t, -t)^\top$
  and  $\bvarphi_2(t)=(t, 0)^\top$.
  The trajectories of the barycenters of
$D_{1,t}$ and $D_{2,t}$ are shown in  Figure~\ref{fig:control_domain} (left and right).
For both examples we consider final time $T=0.75$.
In the first case  the reference domain $\hD$ is a ball of radius $0.2$ centered at
$\bx_I=(-0.6,0.6)$ whereas for the second one, $\hD$ is a ball of radius $0.2$ centered at $\bx_I=(-0.75,0)$.
In both cases, $D_{i,t}$ moves from $\bx_I$ to $\bx_F=(0,0)$.
\begin{figure}[h]
\centering
\includegraphics[height=4.5cm]{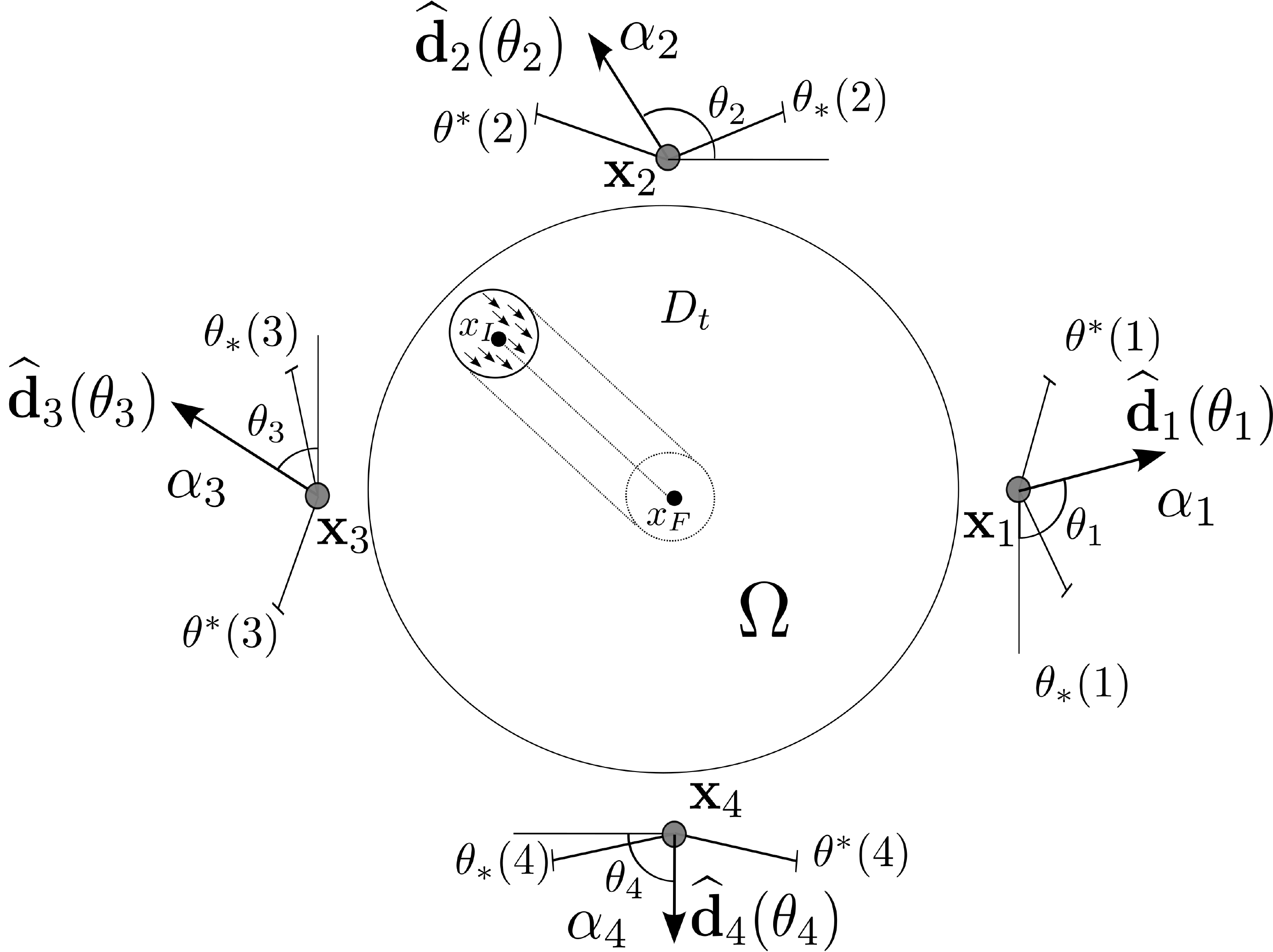}\qquad
\includegraphics[height=4.5cm]{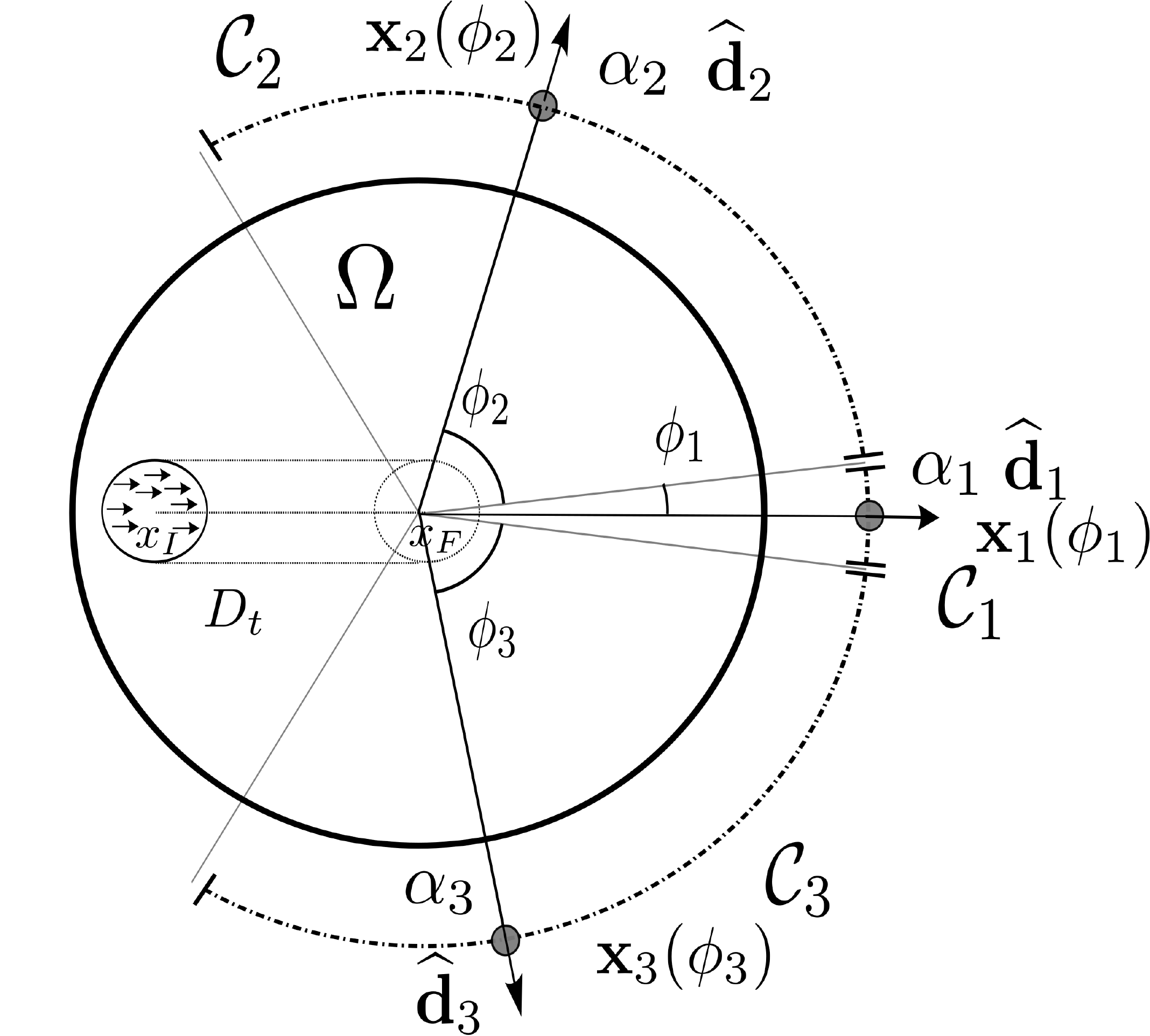}
\caption{Dipole configuration and two different moving subdomains $D_t$ within
$\O\subset\R^2$.  These domains, $D_{1,t}$ (left) and $D_{2,t}$ (right), do not deform, are initialized at $\bx_I$
and travel to their destinations $\bx_F$ along different curves.
The target vector fields $\left\{\bvel_i\right\}_{i=1}^2$, which are tangent to
the curve $\calC$, are  represented by arrows
for the initial configuration.}
\label{fig:control_domain}
\end{figure}
For each of these configurations we have solved the discrete problems with
  $N=100$  time intervals and $\lambda =\eta= 10^{-5}$.
 In order to solve \eqref{eq:PJ1_disc}, we consider an admissible set  characterized
 by initial conditions
 $\balpha_0=(2,0,0,2), \btheta_0=(0,\pi/2,3\pi/2,3\pi/2)\in\R^4$,
 and  by upper and lower bounds
 for the intensities and angles given by:
$\balpha^*=(2,\ldots,2),\balpha_*=(-2,\ldots,-2),\btheta^*=(2\pi,\ldots,2\pi), \btheta_*=(0,\ldots,0)\in\R^4$, respectively.
The initial guesses $(\balphaI,\bthetaI)$ are computed  by Algorithm~\ref{alg:algorithm_1} with  $\verb"tol"=10^{-3}$
with a total of 525 iterations. Notice that, at each time step $n$,  the iterations of the minimization problem in Algorithm~\ref{alg:algorithm_1}
depends on $\nd$ unknowns.
We recall that, due to the non-convexity of the cost functional $\calJ$ (and $\calJ_\tau$), we may converge
to different local minima  depending on the choice of the initial guess.

The solution $\bbalpha_\tau(t)=(\bar{\alpha}_{i,\tau}(t))_{i=1}^4, \bbtheta_\tau(t)=(\bar{\theta}_{i,\tau}(t))_{i=1}^4$
of problem \eqref{eq:PJ1_disc} with initial guesses  given
by  $(\balphaI,\bthetaI)$  is depicted in Figure~\ref{fig:control_ini_dir}.
\begin{figure}
\centering
\includegraphics[height=4cm]{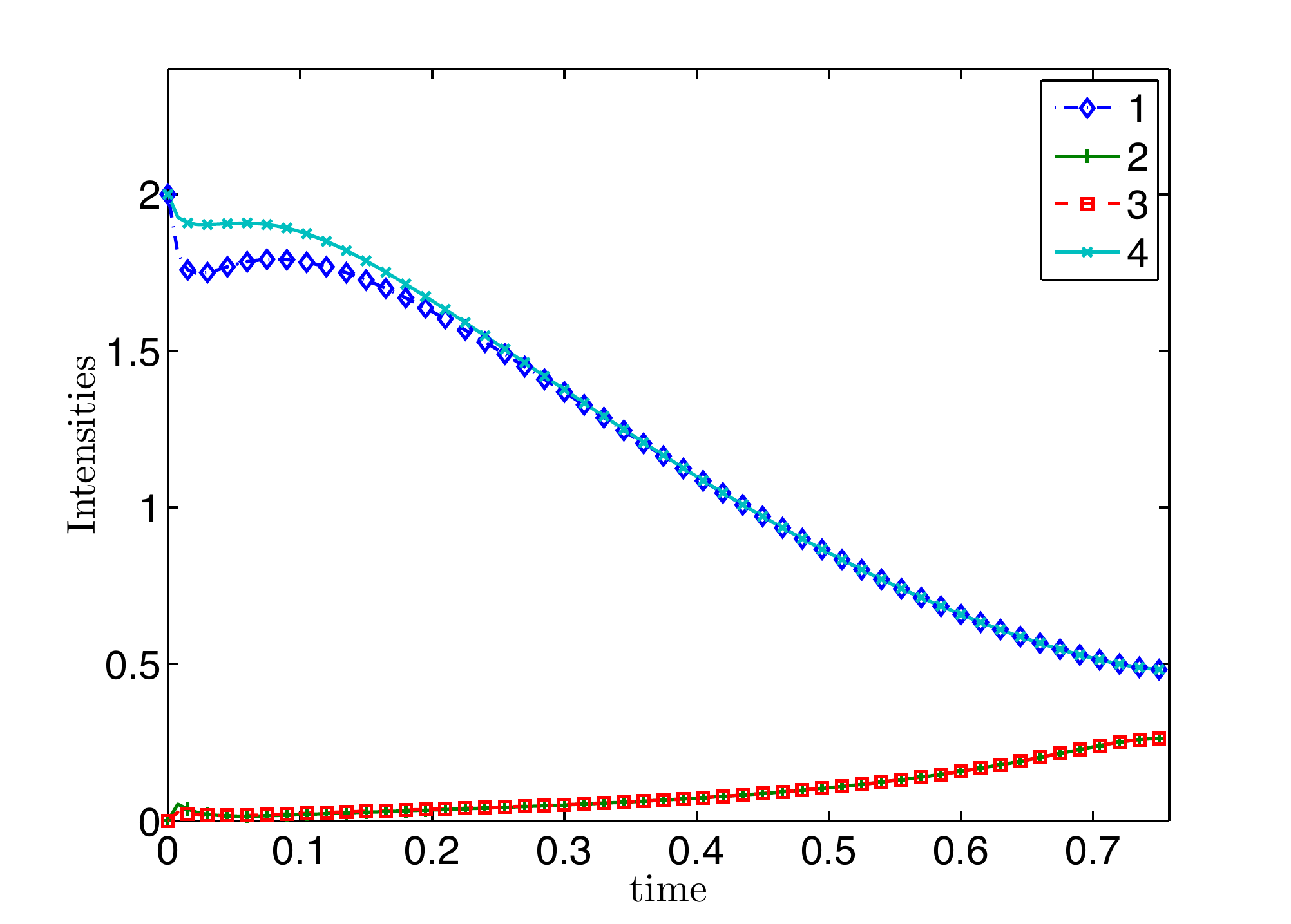}\qquad
\includegraphics[height=4cm]{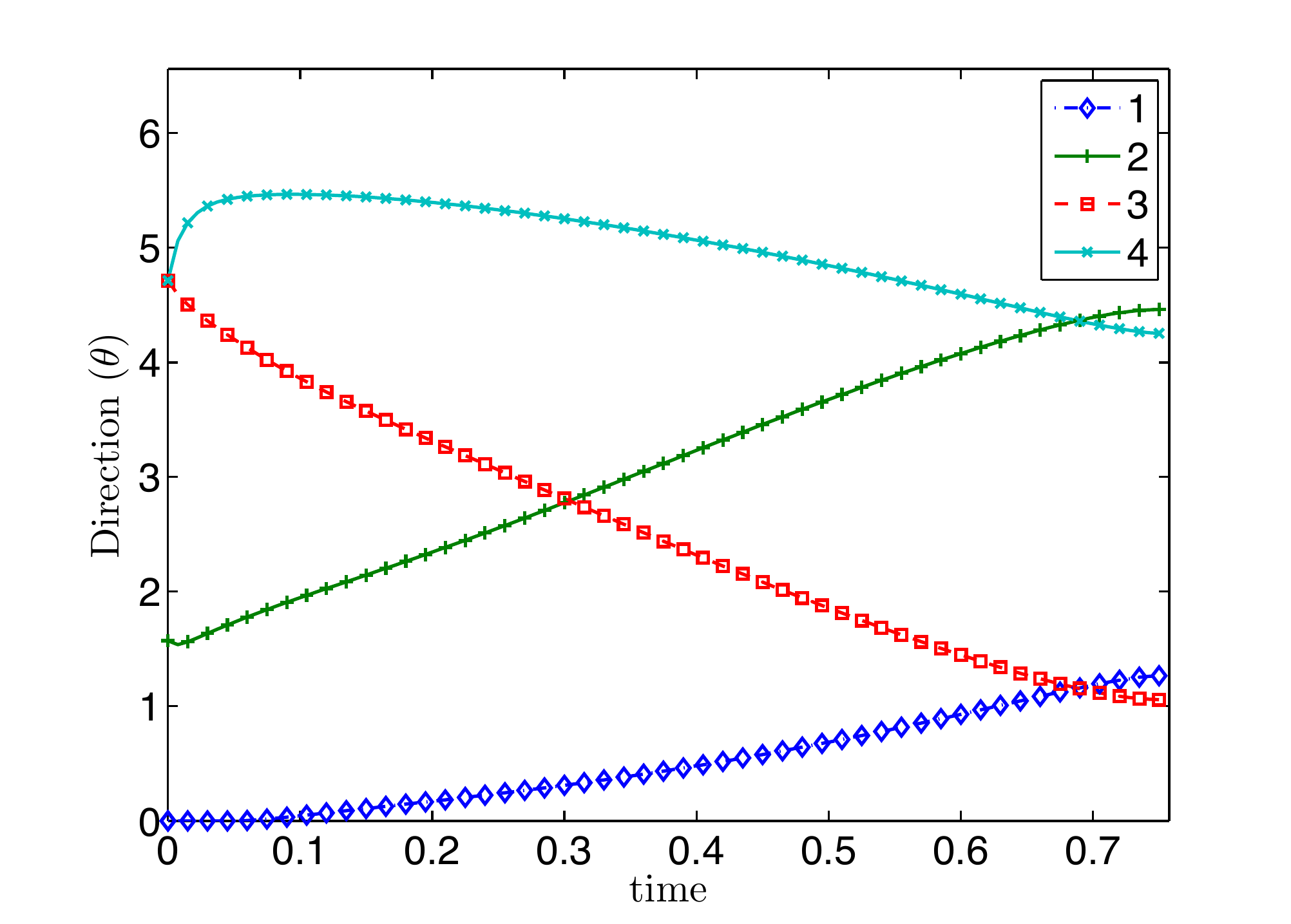}
\caption{Optimal solution $(\bbalpha_\tau,\bbtheta_\tau)=(\bar{\alpha}_{i,\tau},\bar{\theta}_{i,\tau})_{i=1}^4$ to problem \eqref{eq:PJ1_disc}.
The evolution of the intensities (left) and direction, characterized by the angle $\theta $(right) are shown for each dipole $i=1,\ldots, 4$.}
\label{fig:control_ini_dir}
\end{figure}
From Figure~\ref{fig:control_ini_dir} (left) we observe that  the optimal intensity $\balpha_i$ due to $i^{th}$ dipole, is smaller
 when $D_t$ is close to the $i^{th}$ dipole and $\bvel$ is pointing in a  direction opposite to the dipole position.
This is due to the fact that the magnetic forces on the boundary of the domain $\Omega$ are much higher
than inside $\Omega$, thus making it difficult  to approximate the magnetic force when $D_t$
is close to the boundary.
In particular, it can be seen that the intensities of dipole 2 and 3 are small when $t$ is close to 0.
Such a behavior is expected because $D_{1,t}$ is close
 to the boundary of $\Omega$, where the magnetic field generated by these dipoles is large, thus  it is
 difficult for dipoles 2 and 3 to ``push" in the  $\bvel_1$ direction. We also notice that dipoles 1 and 4,
 which can create an attractive field in the  $\bvel_1$ direction, have the
 largest intensities  at initial times and decrease when the time is close to $0.75$. Figure~\ref{fig:control_ini_dir} (right)
shows the dynamics of the angles $(\bar{\theta}_{i,\tau}(t))_{i=1}^4$.

Figure~\ref{fig:ex_dir_line} shows the approximate field and dipole configuration for three time instances
$t=  0.007, 0.375,   0.75$.  Here,  the magnitude of magnetic force $|\nabla|\bH|^2|$  in logarithmic scale
(in the background) and the magnetic force represented by arrows are depicted. The bottom figures  show
 the dipoles direction represented  by the arrows outside the domain $\Omega$.
 The top figures
 illustrate the normalized magnetic force  in $\Omega$, but in the bottom figures the force is restricted to $D_t$.
It can be seen that the magnetic force is close (i.e., almost constant) to $\bvel_1$ in $D_{1,t}$ as expected
whereas it is quite far from constant in the entire domain.
\begin{figure}
\centering
\hfil
\includegraphics[width=0.255\linewidth]{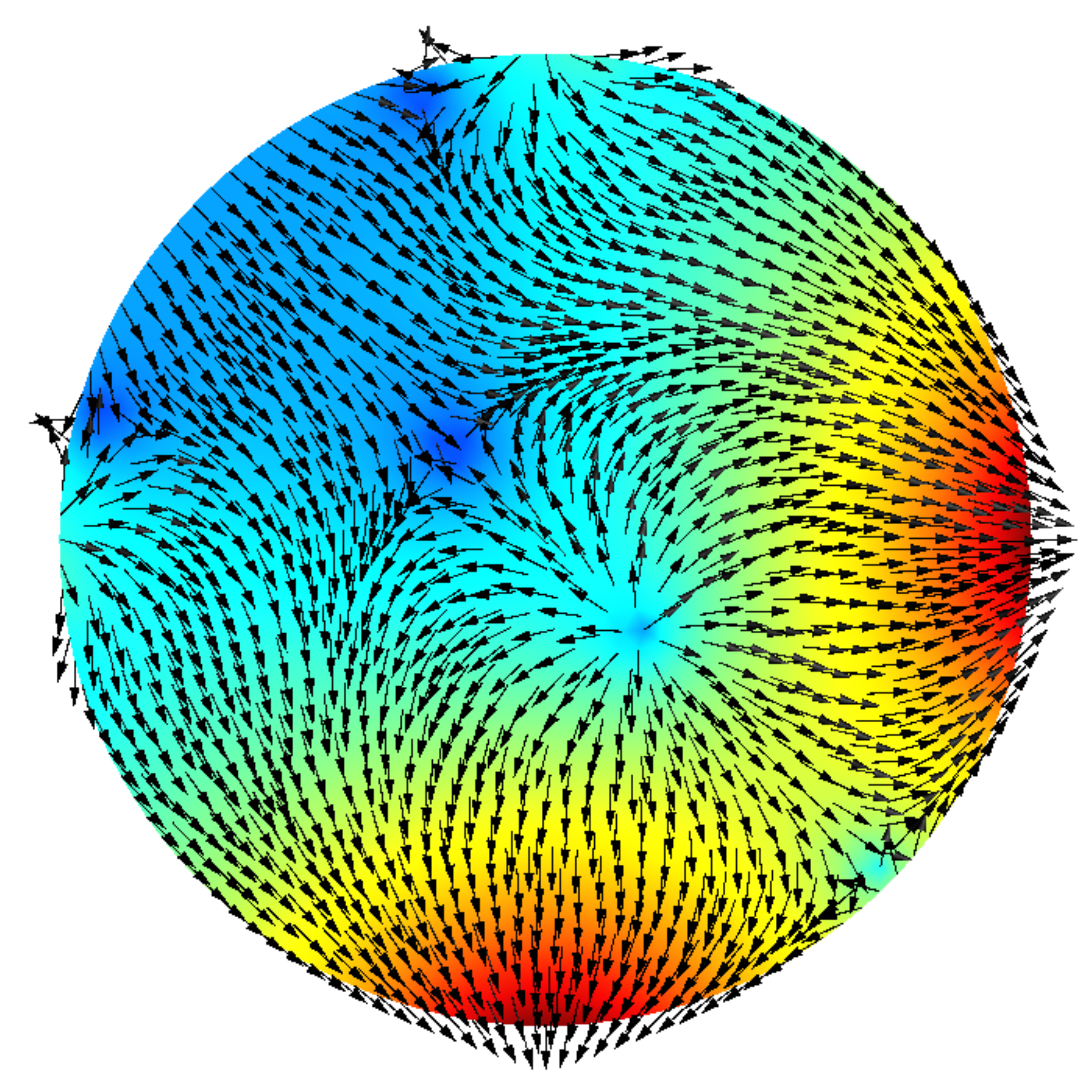}
\hskip1.cm
\includegraphics[width=0.255\linewidth]{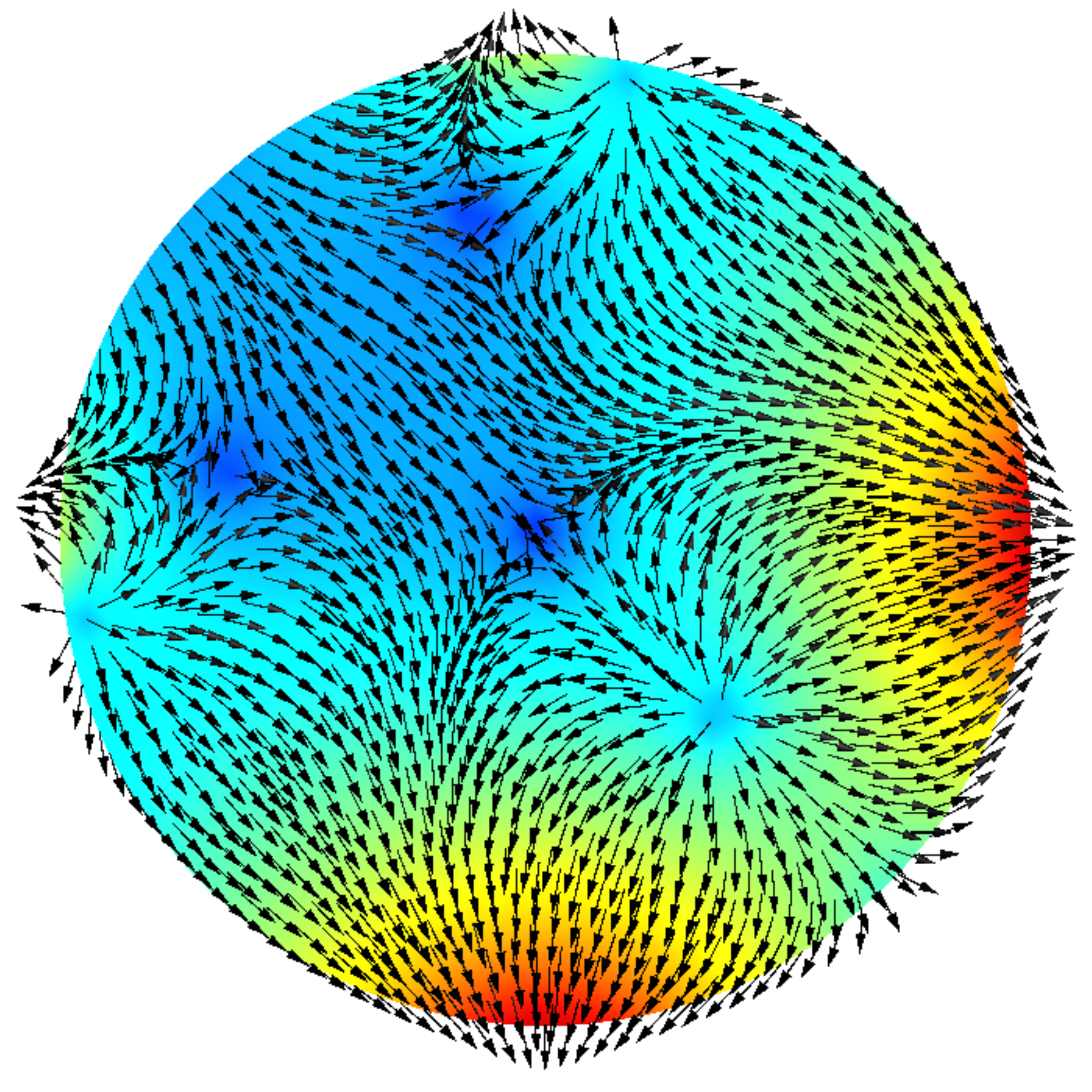}
\hskip1.cm
\includegraphics[width=0.255\linewidth]{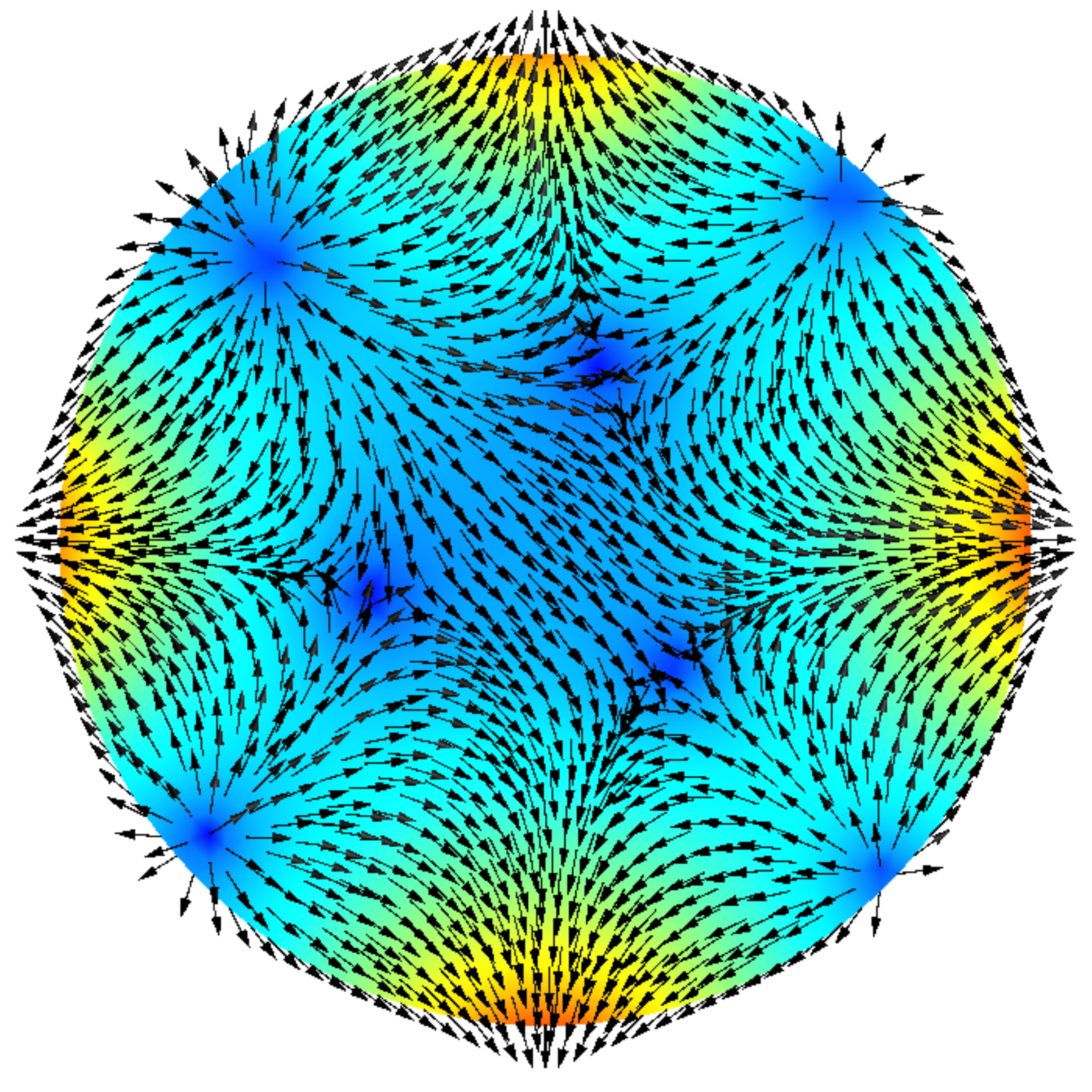}\\
\includegraphics[width=0.325\linewidth]{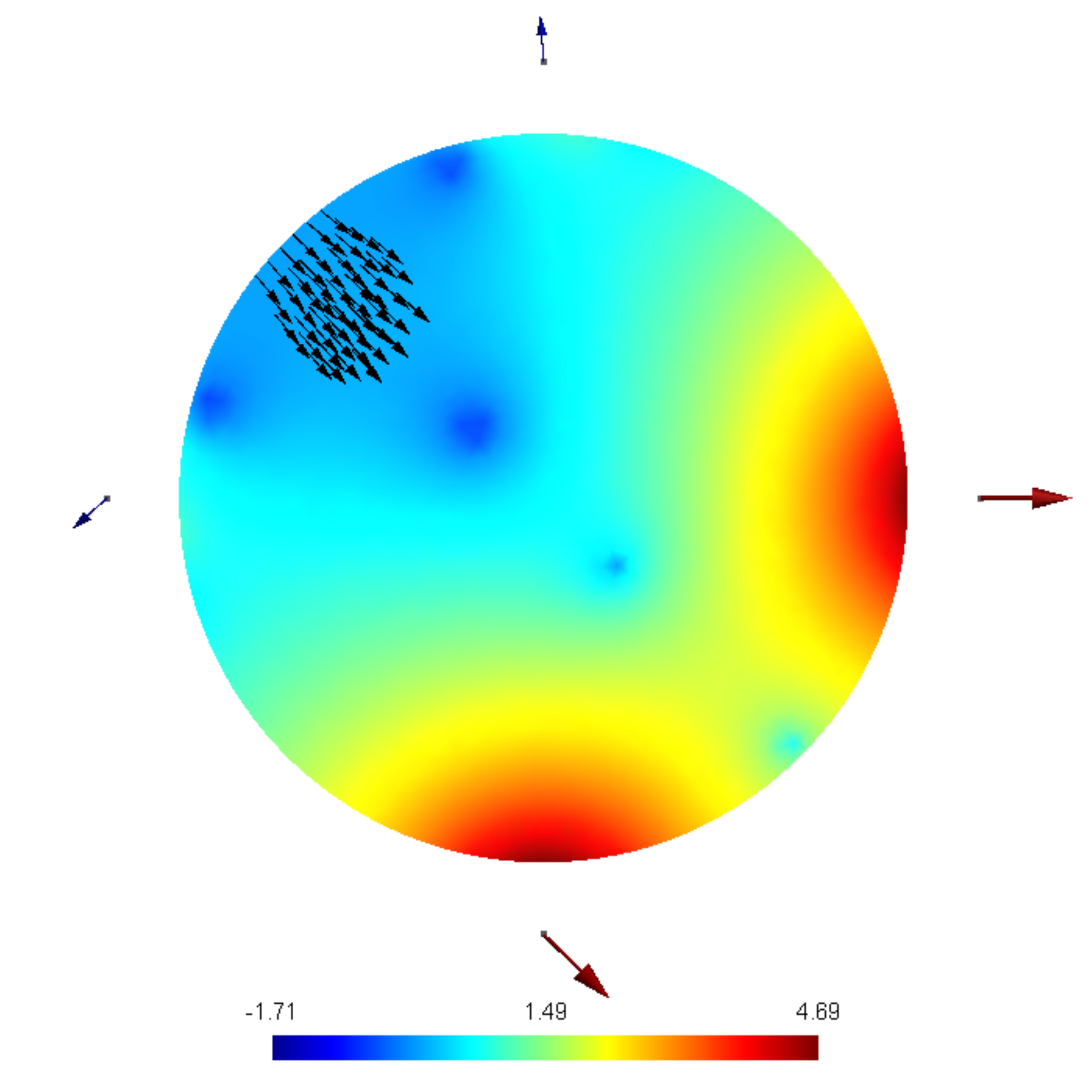}
\includegraphics[width=0.325\linewidth]{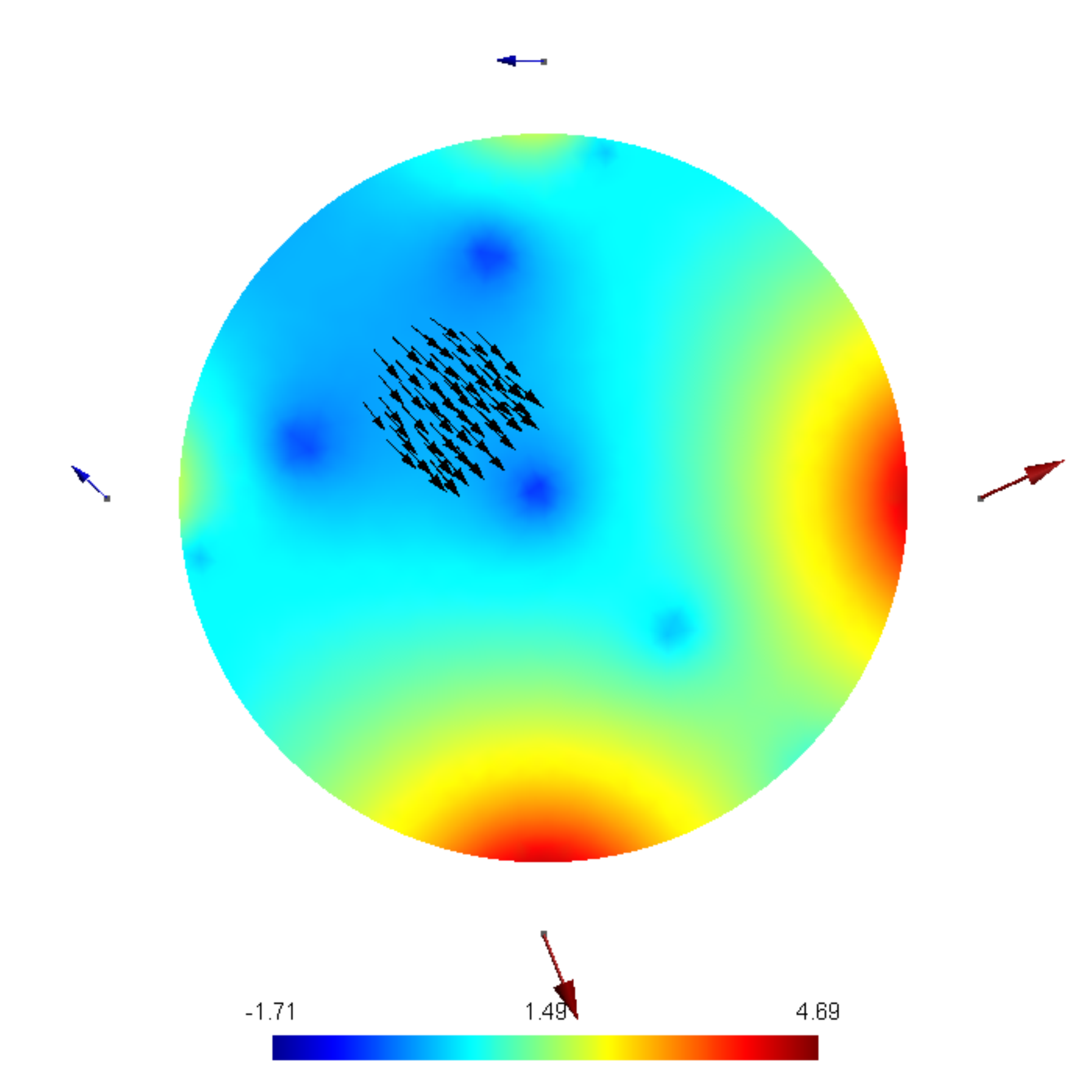}
\includegraphics[width=0.325\linewidth]{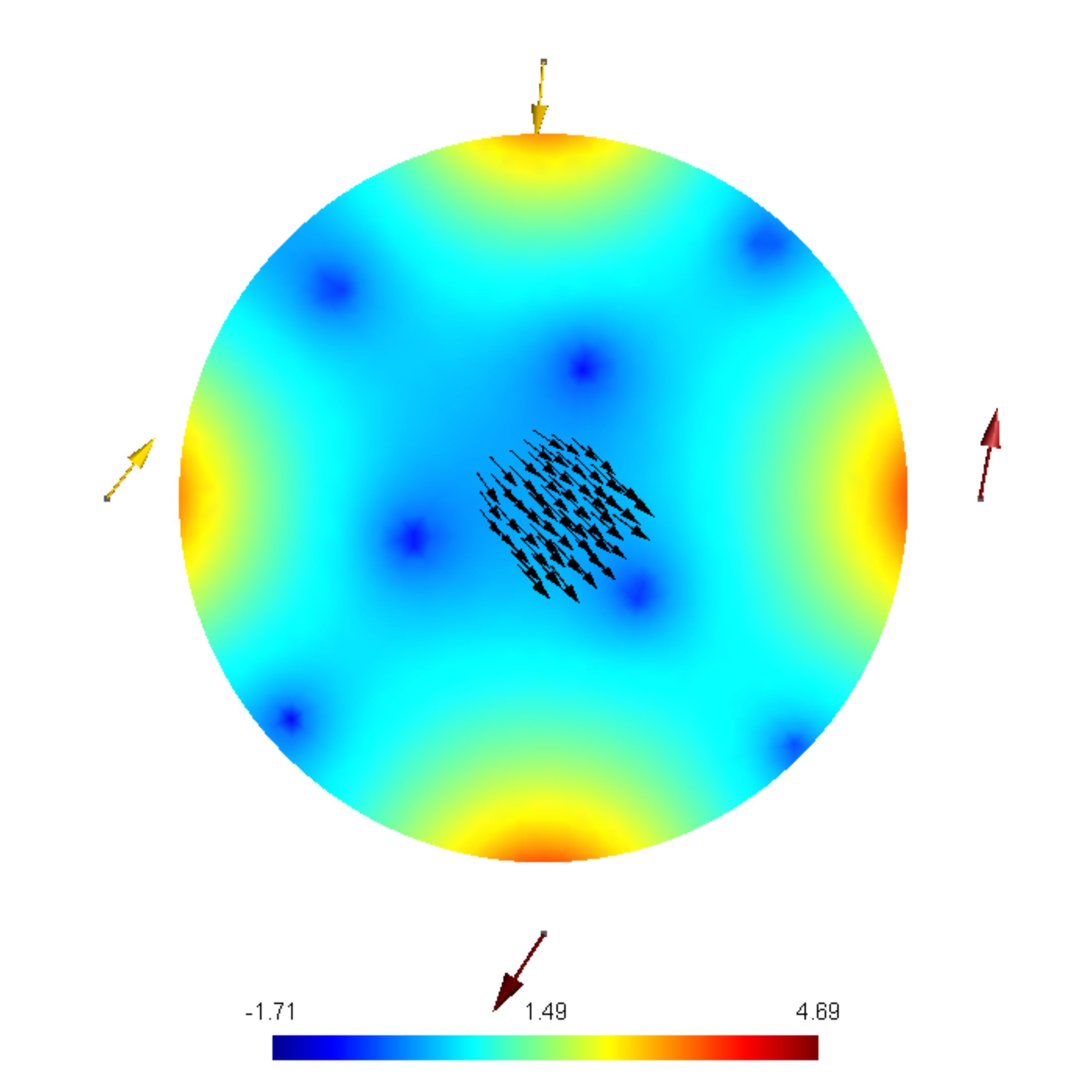}
\caption{Magnetic force solution to the minimization problem \eqref{eq:PJ1_disc} with
$D_{1,t}$ and vector field $\bvel_1$ (force fields shown by black arrows).
Top figure shows the normalized force  at three different times
$t= 0.007, 0.375$ and $0.75$ from left to right, respectively.
In the bottom figure  the magnetic forces are shown only on $D_{1,t}$ for the same time instances.
The arrows outside $\Omega$ represent the dipoles with fixed position and variable direction.
The magnetic force magnitude  $|\nabla|\bH|^2|$ is shown by the background coloring on a logarithmic scale.}
\label{fig:ex_dir_line}
\end{figure}

Next, we approximate the vector field $\bvel_2$ defined on $D_{2,t}$. Here,
given that we want to create a uniform field $\bvel_2(\bx,t) = (1, 0)^\top$
on a domain which moves from left to the center of $\Omega$ (see Figure~\ref{fig:control_domain} (right)),
we restrict the position of the dipoles to
the right of the domain. Namely, we assume that each dipole $i$ moves along a prescribed
trajectory $\calC_i=1.2[\cos(\phi_i),\sin(\phi_i)], i=1,2,3$,  where $\phi_1(t)\in (-\pi/90,\pi/90)$,
$\phi_2(t)\in (\pi/90,3\pi/4)$ and $\phi_3(t)\in (-3\pi/4,-\pi/90)$.
We solve \eqref{eq:PJ2_disc} for an admissible set  characterized
 by initial conditions
 $\balpha_0=(-2,0,0), \bphi_0=(0,2\pi/3,4\pi/3)\in\R^3$,
 and  by upper and lower bounds
 for the intensities and angles given by:
$\balpha^*=(2,2,2),\balpha_*=(-2,-2,-2),\bphi^*=(\pi/90,3\pi/4,2\pi), \btheta_*=(-\pi/90,\pi/90,5\pi/4)\in\R^3$, respectively.

\begin{figure}
\centering
\includegraphics[height=4cm]{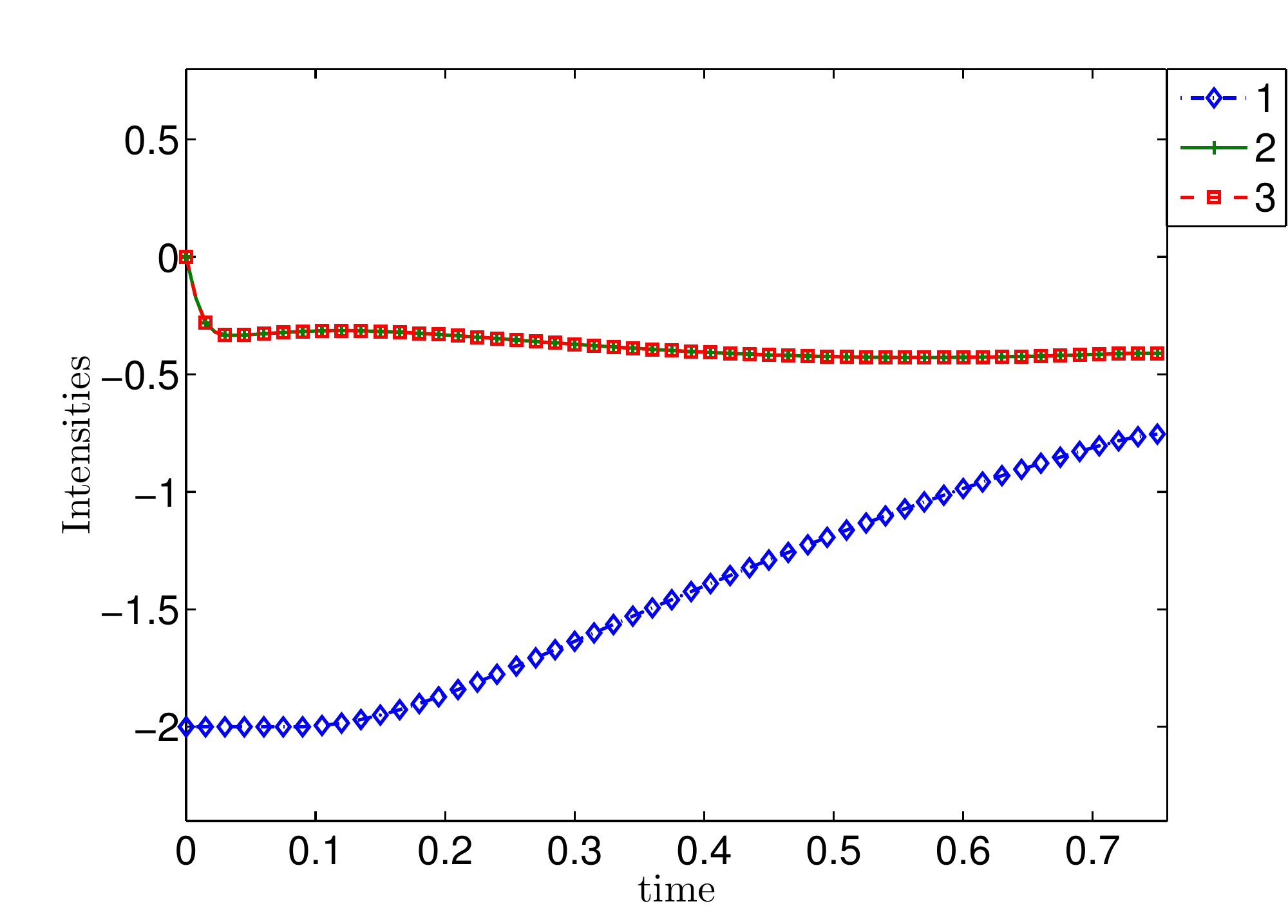}\qquad
\includegraphics[height=4cm]{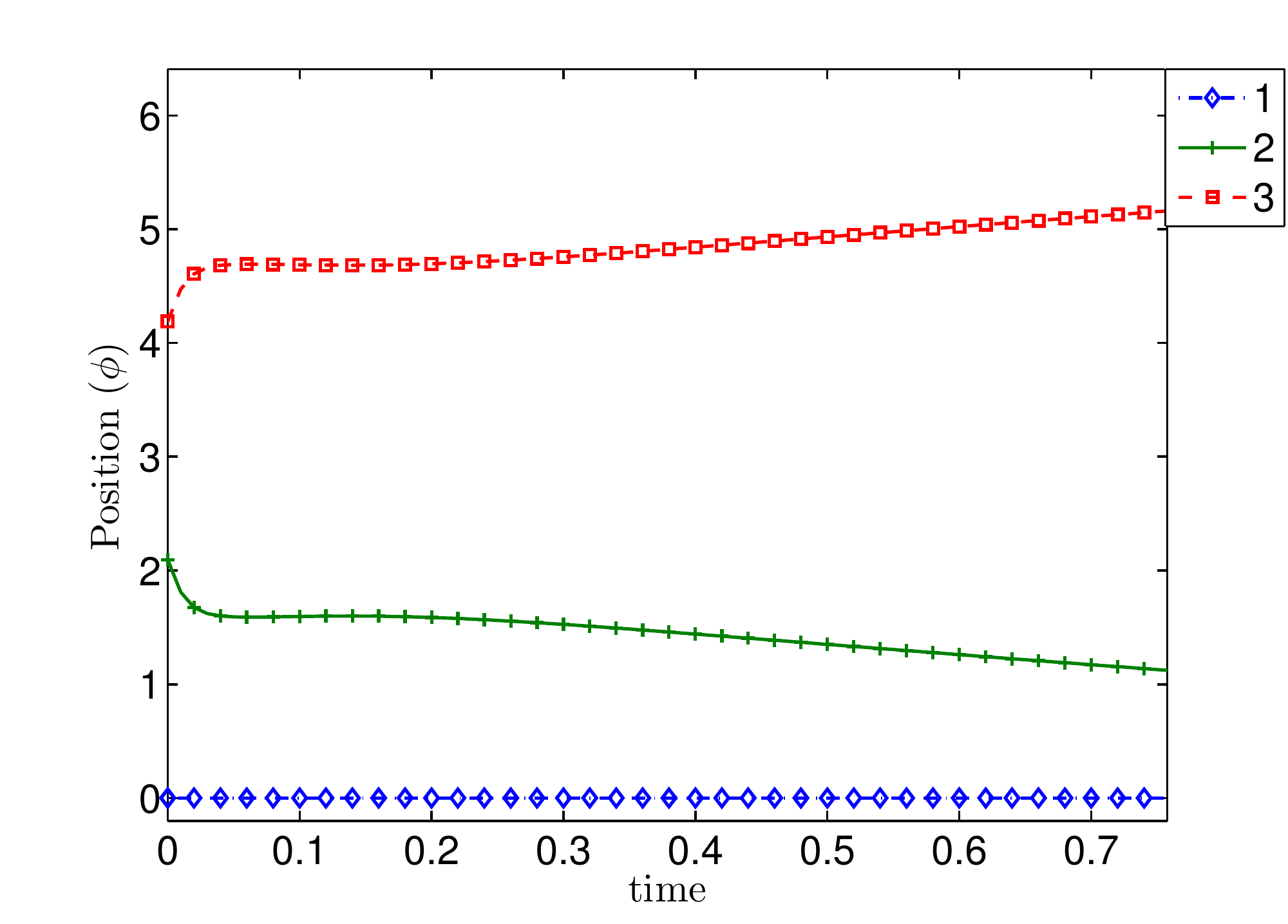}
\caption{Optimal solution $(\bbalpha_\tau,\bbphi_\tau)=(\bar{\alpha}_{i,\tau},\bar{\phi}_{i,\tau})_{i=1}^3$ to problem \eqref{eq:PJ2_disc}.
The evolution of the intensities (left) and position, characterized by the angle $\phi$, (right) are shown for each dipole $i=1,\ldots, 3$.}
\label{fig:control_ini_pos}
\end{figure}

\begin{figure}
\centering
\hfil
\includegraphics[width=0.255\linewidth]{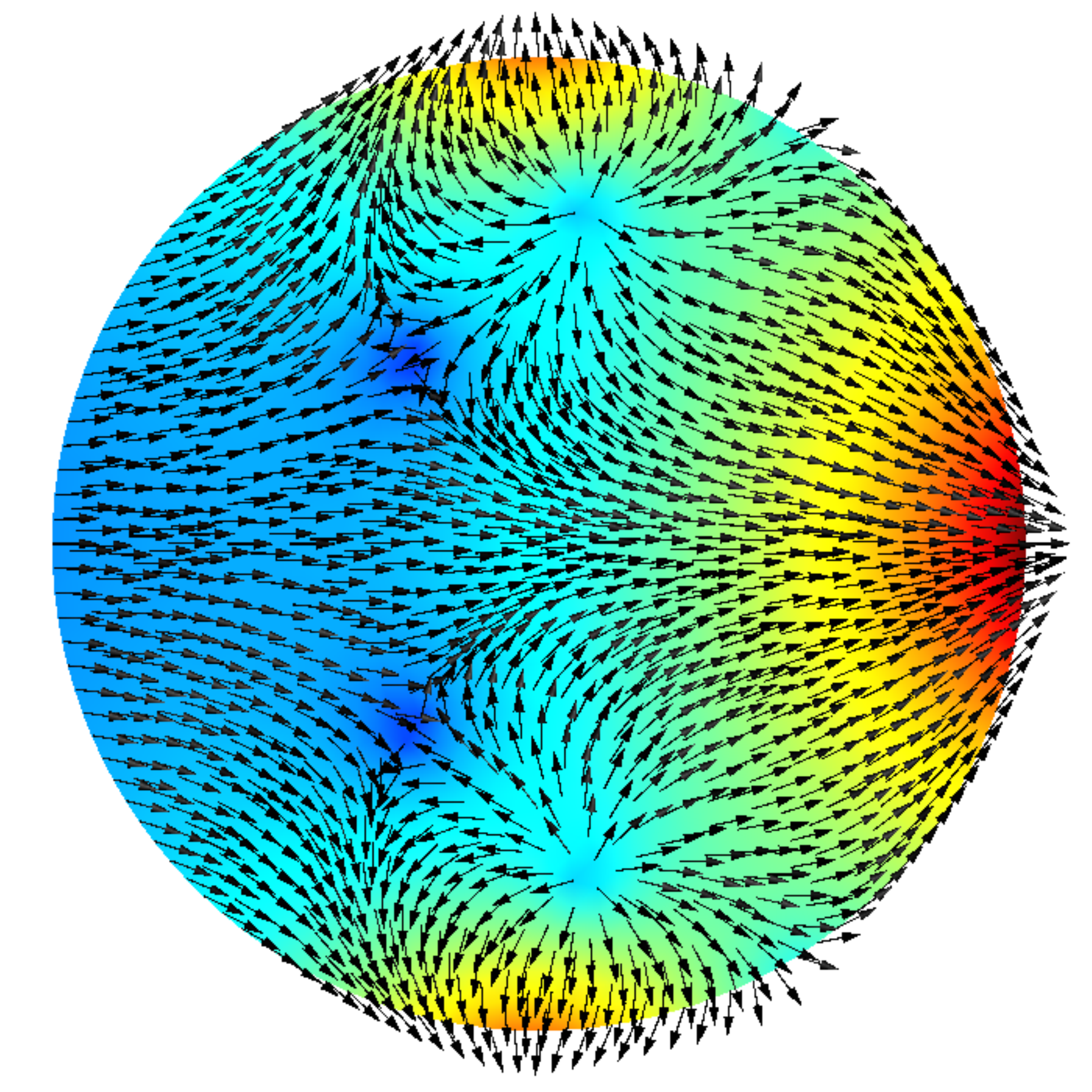}
\hskip1.cm
\includegraphics[width=0.255\linewidth]{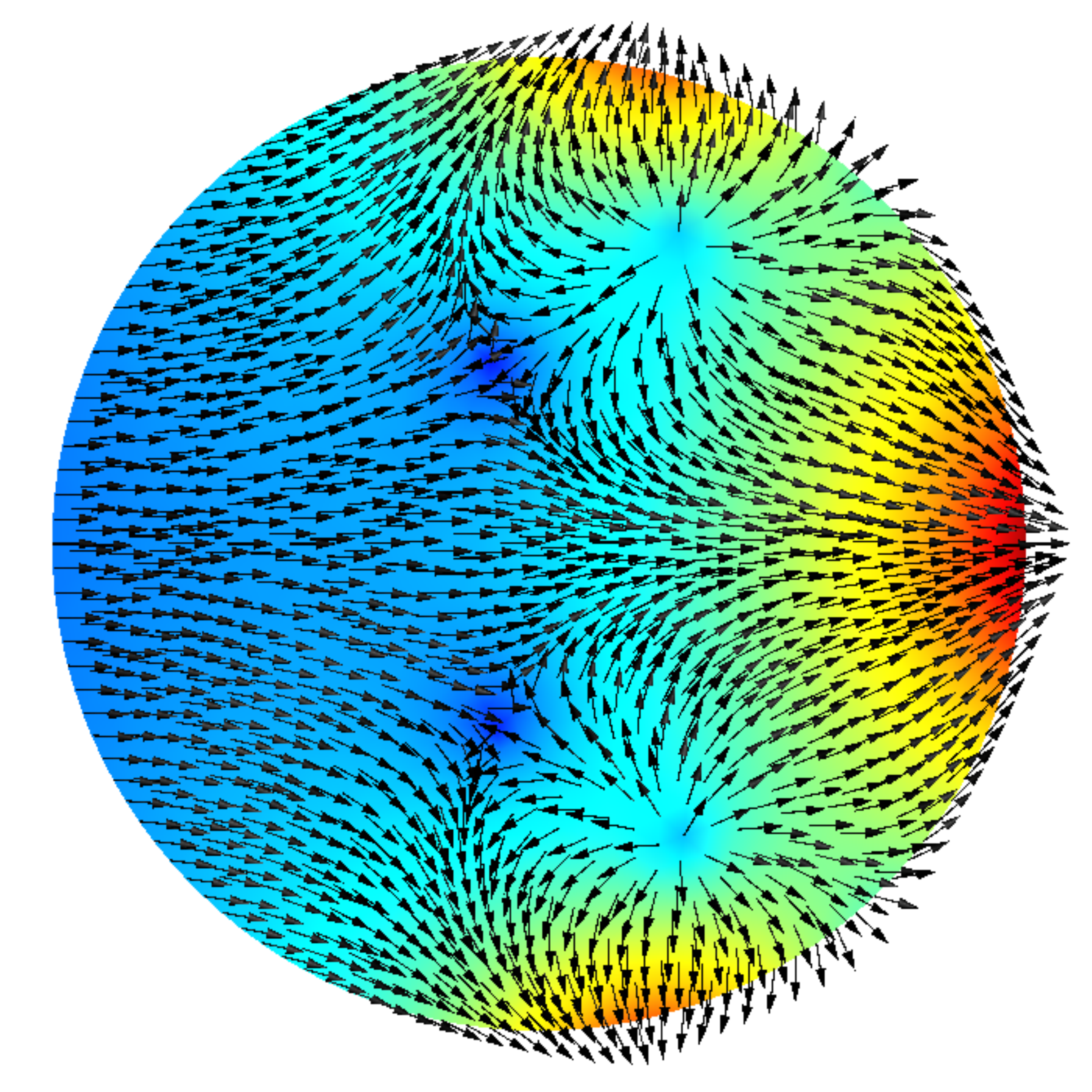}
\hskip1.cm
\includegraphics[width=0.255\linewidth]{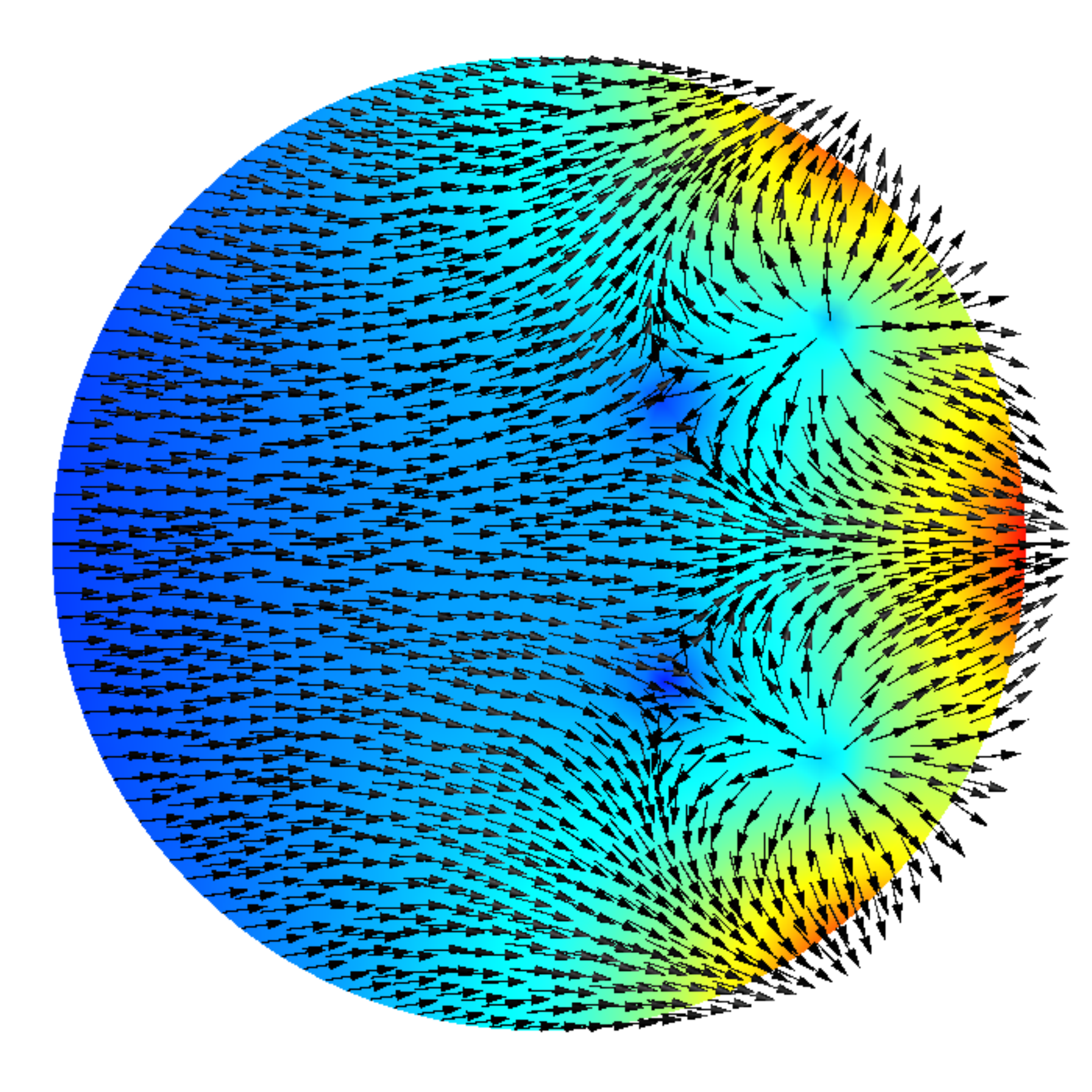}\\
\includegraphics[width=0.325\linewidth]{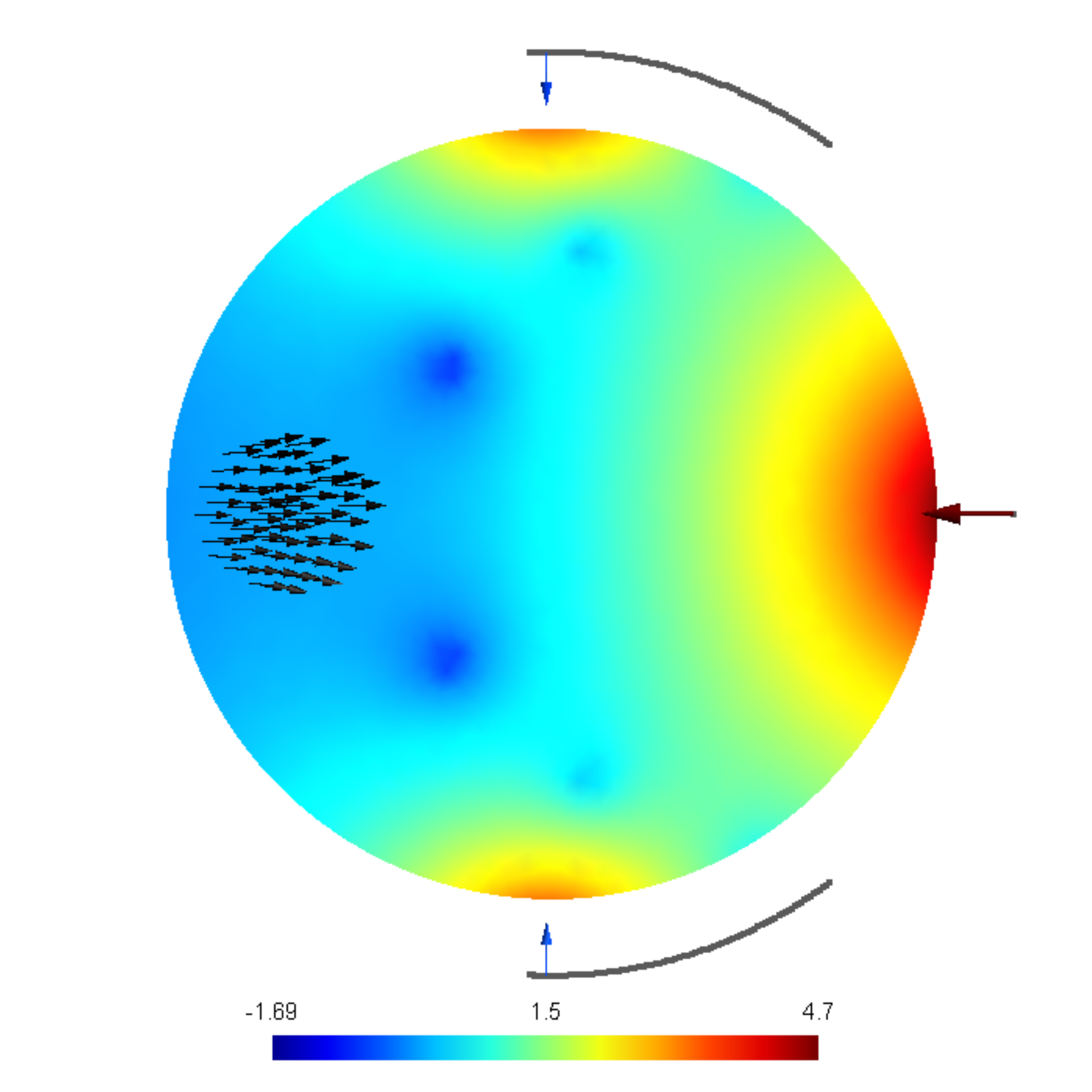}
\includegraphics[width=0.325\linewidth]{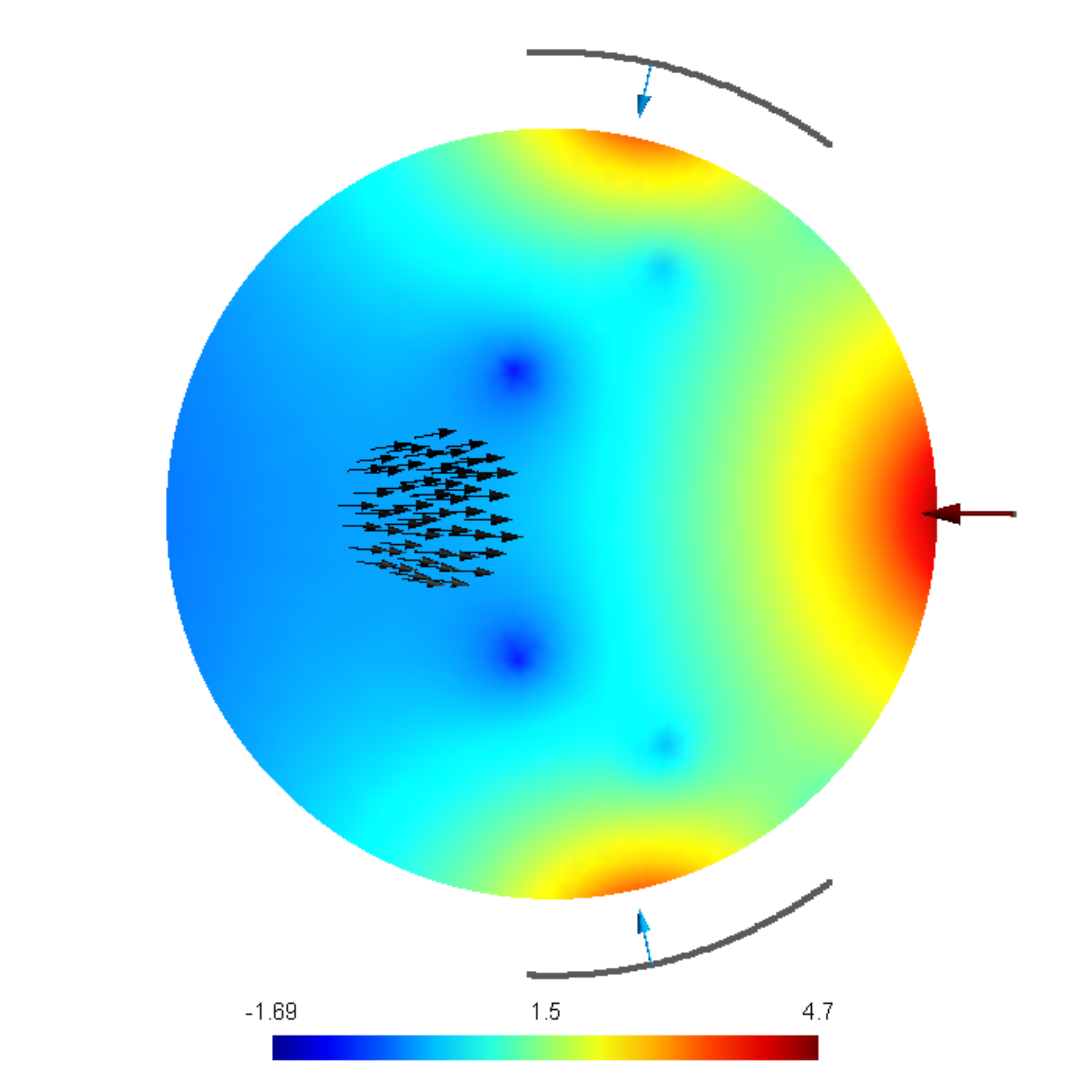}
\includegraphics[width=0.325\linewidth]{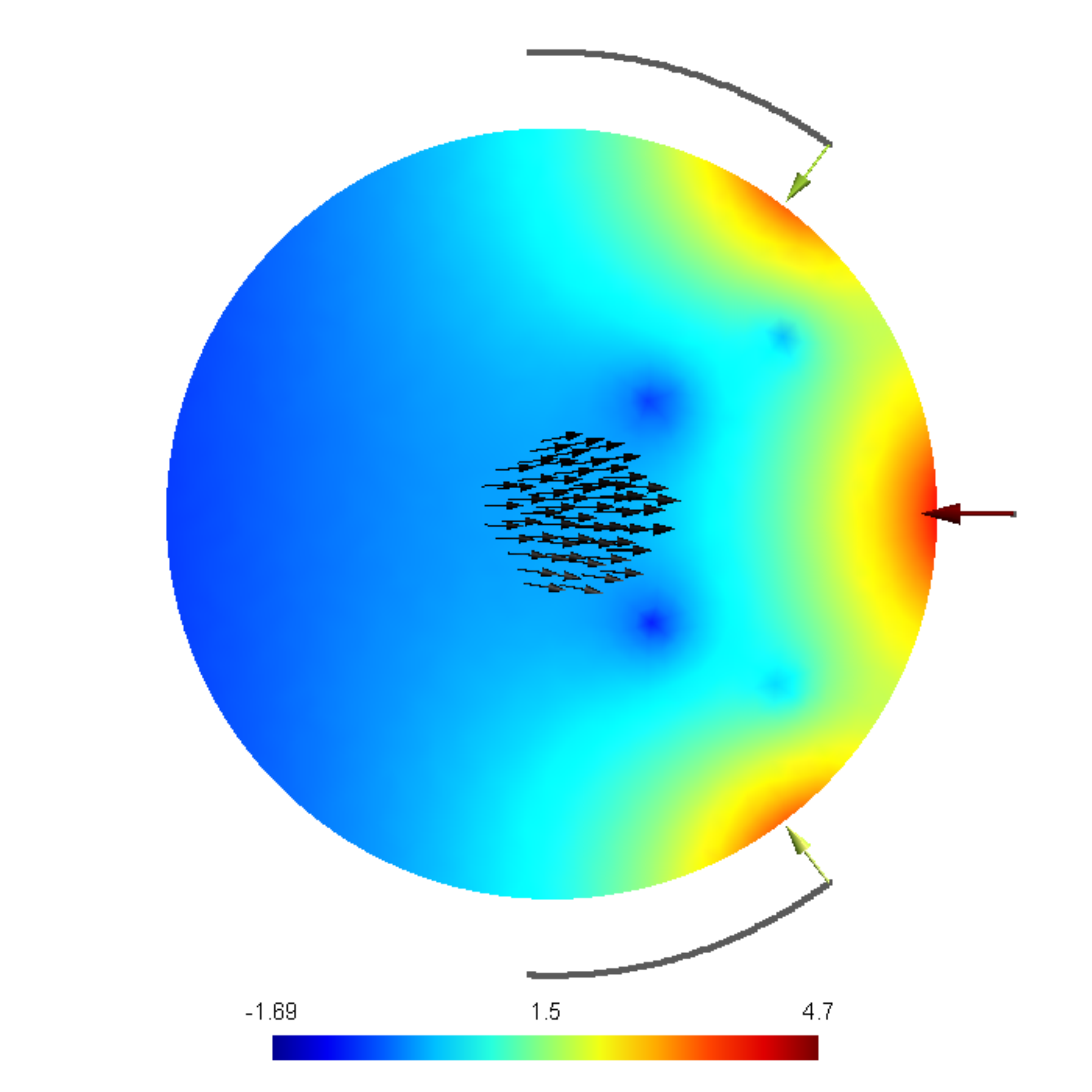}
\caption{Magnetic force solution to the minimization problem \eqref{eq:PJ2_disc} with
$D_{2,t}$ and vector field $\bvel_2$ (force fields shown by black arrows).
Top figures shows the normalized force  at three different times
$t= 0.007, 0.375$ and $0.75$ from left to right, respectively.
In the bottom figures  the magnetic forces is defined only on $D_{2,t}$
for the same time instances.
The arrows outside $\Omega$ represent the dipoles with fixed direction and variable position.
The path covered by each dipole is depicted by the line outside $\Omega$.
The magnetic force magnitude  $|\nabla|\bH|^2|$ is shown by the background coloring on a logarithmic scale.}
\label{fig:ex_posit_line}
\end{figure}

The vector intensities $(\bar{\alpha}_{i,\tau})_{i=1}^3$
and the parametrization of the position $(\bar{\phi}_{i,\tau})_{i=1}^3$
solutions to the minimization problem \eqref{eq:PJ2_disc} are shown in Figure \ref{fig:control_ini_pos} left and right, respectively .
We notice that the first dipole
attains a fixed position and intensity (the maximum intensity), whereas dipoles 2 and 3
are moving but their intensity is almost constant.
Given that we want to approximate a
uniform and unitary vector field  on $D_{1,t}$, the intensity of dipole 1 decreases as the domain
approaches dipole 1. This is due to the fact that the force applied to the edge of $D_{1,t}$,
close to dipole 1, is larger than the one applied to the farther
edge, and this difference increases as $D_{1,t}$ is close to the source.

The magnetic force and dipole positions are shown  for three time instances
$t= 0.007, 0.375$ and $0.75$  in Figures~\ref{fig:ex_posit_line}. Here dipoles 2 and 3 ``follows" the trajectory of
$D_{2,t}$ which allows to obtain an almost constant field on the moving domain.
From the previous figures it can be seen that our proposed minimization procedure
produces a vector field that is close to the target field.

\section{Numerical solution of the convection-dominated diffusion equation}\label{s:addiff_numerics}

From the previous section  it follows that an almost uniform magnetic force field
can be attained in a small region for two different dipole configurations.
In this section we move a concentration of ferrofluid modeled by
\eqref{eq:strong_1}-\eqref{eq:strong_max} where the advection term is computed
either by solving the minimization problem \eqref{eq:PJ1} or \eqref{eq:PJ2} depending
on the dipole configuration. The aim of these examples is to study the feasibility of
the dipole configurations to create a magnetic force that allow us to ``push" the
concentration inside the domain.

For simplicity, we assume $\gamma=\nu=1$, $\bv{u}=\bv{0}$.
To further minimize the spreading we choose a small diffusion coefficient,
in particular, we set $A=\varepsilon\mathbb{I}$, with $\varepsilon\ll1$.
As a result the advection diffusion equation \eqref{eq:strong_1} and \eqref{eq:strong_2} can be written as
\begin{subequations}
\begin{align}
\partial_t c+\di \left( -\varepsilon\nabla c   +c\nabla|\bH|^2\right)=0 \quad \mbox{in } \widetilde{\Omega}\times(0,T)&\label{eq:pb_strong_a}\\
(-\varepsilon\nabla c   +c\nabla|\bH|^2)\cdot\boldsymbol{n}=0\quad \mbox{on } \partial\widetilde{\Omega}\times(0,T)\qquad c(x, 0) = c_0
\quad \mbox{in } \widetilde{\Omega}&\label{eq:pb_strong_b}
\end{align}
\end{subequations}
Multiplying \eqref{eq:pb_strong_a} by a test function $v\in\rH^1(\widetilde{\Omega})$,
integrating by parts, and using the boundary condition in
\eqref{eq:pb_strong_b} leads us to the following problem: Given $c_0\in \rL^2(\widetilde{\Omega})$, find $c\in \rL^2(0,T;\rH^1(\widetilde{\Omega}))\cap \rH^1(0,T;\rH^1(\widetilde{\Omega})')$ such that
\begin{subequations}\label{eq:weak}
\begin{align}
\langle\partial_t c,v\rangle+\left( \varepsilon\nabla c  - c\nabla|\bH|^2,\nabla v\right)&=0 \quad \forall v\in \rH^1(\widetilde{\Omega}),\\
c(x, 0) &= c_0\quad \mbox{in } \widetilde{\Omega},
\end{align}
\end{subequations}
here $\langle\cdot,\cdot\rangle$ is the duality product between $\rH^1(\widetilde{\Omega})$ and $\rH^1(\widetilde{\Omega})'$.
The existence and uniqueness of solution to this problems follows, for instance, from \cite[Theorem~6.6]{EG04}. Notice that in case of the Dirichlet boundary condition $c=0$ we must replace $\rH^1(\widetilde\Omega)$ by $\rH^1_0(\widetilde\Omega)$ in the above calculations.

Below we present two numerical examples. In the first example we consider a simplified
$\widetilde\Omega$ whose size is comparable to $\Omega$ where the goal is to
magnetically inject the concentration of nanoparticles. 
With the current configuration, 
the concentration may reach the boundary and,
as is mentioned in \cite{NBBS2011} for a similar problem, boundary layers may appear.
On the other hand, given that we are interested in the case $\varepsilon\ll1$,
 a proper space discretization has to be considered in order to avoid numerical oscillations. 
We adopt a monotone scheme for space discretization: edge-averaged finite element (EAFE) \cite{XZ1999}. Other alternatives
such as SUPG stabilization are equally
valid. Given that the EAFE scheme is a type of upwinding scheme, we observe a significant diffusion in all cases.
Thus making it difficult to control the concentration in case of complicated geometries, for instance the flow around
an obstacle.

As a second example we study the feasibility of moving the concentration around
an obstacle. In order to ensure the uniformity of the field we consider a small, but complex, domain $\widetilde{\Omega}$
(see Figure~\ref{fig:pde_control_domain} (right)).
Moreover, it is shown that Dirichlet boundary condition $c=0$
on $\partial\widetilde\Omega$ in \eqref{eq:pb_strong_b} is meaningful in this example.
For the numerical approximation we consider piecewise linear finite element space discretization and
explicit Euler scheme for time discretization. In order to avoid having to solve a linear
system at each time step we consider  mass lumping with a correction technique to account for the dispersive effects (see \cite{GP2013}). 
In section~\ref{s:eafe}
we discuss the EAFE scheme and in section~\ref{s:explicit} we discuss the explicit
scheme. Numerical examples are presented in each case.

\subsection{Edge-averaged Finite Element (EAFE) Method and Magnetic Injection}
\label{s:eafe}

 It is well--known that
the standard finite element method applied to problem \eqref{eq:weak}, yields solution oscillations when
$\varepsilon\ll \|\nabla|\bH|^2\|_{L^\infty(0,T;\widetilde\Omega)}$.
However, there is no universal approach to treat such problems.
In order to choose a numerical scheme to approximate problem~\eqref{eq:weak}
first notice an interesting
feature of the flux $J(c)$ of \eqref{eq:pb_strong_a}:
\begin{equation}\label{eq:adv-1}
J(c):=\varepsilon\nabla c - \nabla|\bH|^2c
= e^{|\bH|^2\varepsilon^{-1}}\nabla\left(\varepsilon \, e^{-|\bH|^2\varepsilon^{-1}}c\right)
= a \nabla u
\end{equation}
i.e., it is symmetrizable. Here
\begin{equation}\label{eq:adv0}
 a = e^{|\bH|^2\varepsilon^{-1}} , \quad \psi = -|\bH|^2\varepsilon^{-1}, \quad
 u = \varepsilon e^{\psi}c .
\end{equation}
Equations with the above property have been studied by many authors  (see, for instance, \cite{BMP1989,MZ1988}).
They also appear in applications such as semiconductors device simulation.
In order to exploit the structure of the equation, for the numerical
approximation we consider the edge-average finite element (EAFE), a
\textit{monotone scheme} introduced in \cite{XZ1999} for stationary
advection-diffusion equations. It was later extended to time-dependent
advection-diffusion equations in \cite{BK2015} and
compared with SUPG, another well-known stabilization technique.
With this in mind we adapt the scheme presented in \cite{BK2015} to our particular case.

We begin by recalling the EAFE scheme of \cite{XZ1999} and apply it to
the symmetric operator in \eqref{eq:adv-1}. We denote by $\calT_h$
a shape regular triangulation of $\widetilde{\Omega}$ and by
$\rH_h\subset \rH^1(\Omega)$ the space of piecewise linear finite
elements with hat basis $\{\phi_i\}_{i=1}^N$. Notice that the scheme
proposed below is equally applicable to both the Neumann and the Dirichlet
problems.
We designate by $E = E_{ij}$ the edge connecting the nodes $x_i$ and
$x_j$. The discrete weak form of $-\textrm{div }(a\nabla u)$ reads
\begin{equation}\label{eq:adv1}
  \int_\Omega a \nabla u_h \cdot \nabla v_h =
  \sum_{T\in\calT_h} \int_\Omega a \nabla u_h \cdot \nabla v_h
  = - \sum_{T\in\calT_h} \sum_{E\subset\partial T} a_E^T
  \delta_{E} (e^\psi c_h) \delta_{E} v_h,
\end{equation}
where $\delta_{E} (w_h) = w_i-w_j$ and
\[
a_E^T = \left(\ \int_T \varepsilon a \nabla \phi_i \cdot \nabla\phi_j\,dx \right)
= \left(\frac{1}{|T|} \int_T \varepsilon e^{-\psi} \, dx \right)
\left(\int_{T} \nabla\phi_i \cdot \nabla\phi_j \, dx \right)
\]
because $\nabla\phi_i$ is piecewise constant. We now replace the first
factor by the so-called harmonic average over the edge $E$,
namely
\[
\alpha_E := \left( \frac{1}{|E|} \int_E \varepsilon^{-1} e^{\psi} \, dx
\right)^{-1} \approx \frac{1}{|T|} \int_T \varepsilon e^{-\psi} \, dx.
\]
Upon setting $\omega_E^T := \int_T \nabla\phi_i\cdot\nabla\phi_j$, we
realize that \eqref{eq:adv1} can be approximated by
\begin{equation}\label{eq:B}
\mathcal{B}_h(c_h,v_h):= - \sum_{T\in\calT_h}\left(\sum_{E\subset \partial T}
\omega_E^T \alpha_E
\delta_E(e^{\psi}c_h)\delta_E (v_h)\right)
\quad \forall\, c_h,v_h\in \rH_h,
\end{equation}
which is the bilinear form corresponding to the EAFE scheme of \cite{XZ1999}.

For the time interval $[0, T]$ we introduce a uniform partition $0 = t_0 < t_1 <\ldots <
t_N = T$, with $N \geq 1$, and denote by $\Delta t=T/N$ the time step size.
By applying the implicit Euler scheme for time
discretization, we obtain the following fully discrete scheme of problem $\eqref{eq:weak}$:
Given $c_0\in \rL^2(\widetilde{\Omega})$, for $n=1,\ldots,N$ find $c^n_h\in \rH_h$ such that
\begin{subequations}\label{eq:discrete2}
\begin{align}
(c^n_h,v_h)+\Delta t \mathcal{B}_h(c^n_h,v_h) &=(c^{n-1}_h,v_h) \quad \forall v_h\in \rH_h\\
c^0_h &= c_0
\end{align}
\end{subequations}
Error estimates of a similar discrete problem  can be found in \cite{BK2015}.

In the following subsections we will study the feasibility of the
manipulation of nanoparticles concentration
by controlling the dipole intensity and direction. To fix ideas, let
$\Omega$ be a ball of unit radius centered
at $(0, 0)$. The configuration of the 4 dipoles is same as in the previous example (cf. Section~\ref{sec:example1}).

In our first example we focus on magnetically injecting the initial concentration $c_0$ of nanoparticles so that it arrives at a desired location. 
 Instead of just the intensity (cf. \cite{ANV2016,KS2011}) the control variables here are both the intensities and the directions of the dipoles. Additionally, the concentration is close to the boundary for certain time instances. A standard Galerkin scheme in space usually leads to oscillations in this case. We assume that the domain $\widetilde{\Omega}\subset\Omega$ is a $\pi/4$ clockwise rotation of the rectangle $[-0.9,0.9]\times[-0.3,0.3]$ (cf. Figure~\ref{fig:pde_control_domain}, left). We set the final time $T = 0.75$ and
$c_0(x,y) = \exp\left(3\times10^{-5}((x+0.53)^2+(y-0.53)^2)\right)$ (cf. Figures~\ref{fig:pde_control_domain} and \ref{fig:pde_dir_line} (left)).

\begin{figure}[h!]
\centering
\includegraphics[height=3.3cm]{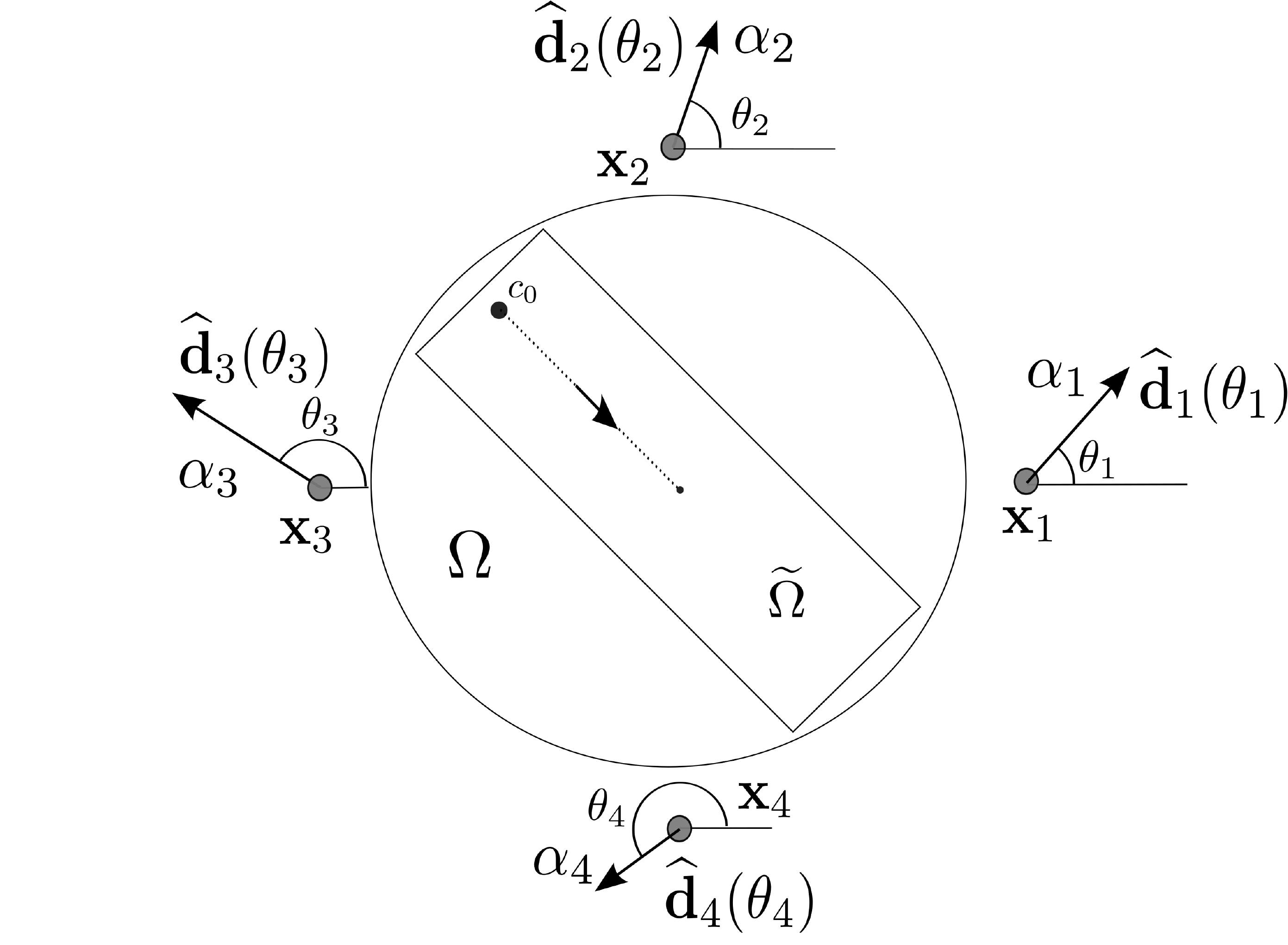}\quad
\includegraphics[height=3.3cm]{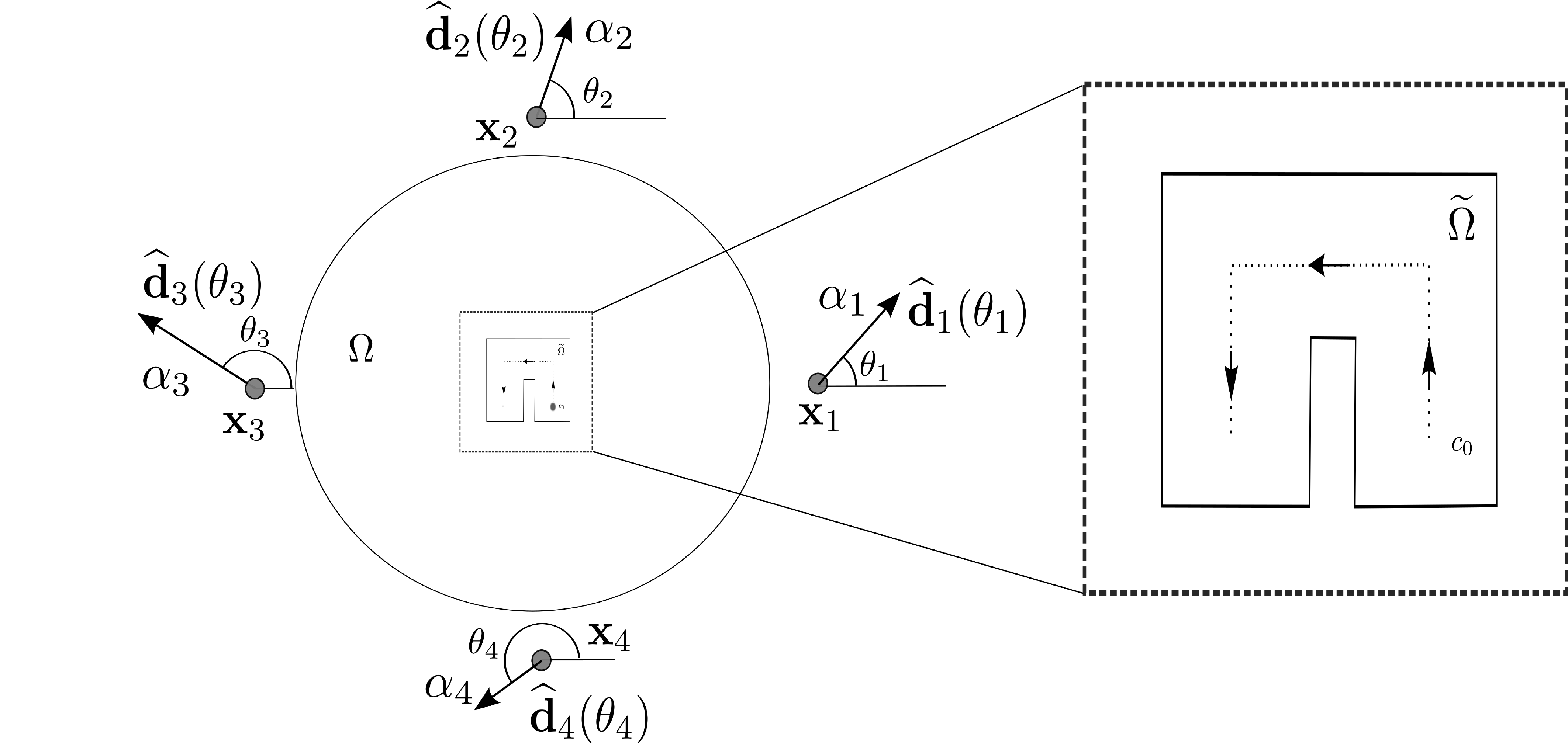}
\caption{Computational domain $\widetilde{\Omega}\subset \Omega$ and dipole configuration with fixed position and variable
direction and intensity.}
\label{fig:pde_control_domain}
\end{figure}

We use the optimal magnetic force computed in the first example of Section~\ref{sec:example1}  (cf.~ Figure~\ref{fig:ex_dir_line}) as an input to \eqref{eq:discrete2}.
Here we have set $\varepsilon=10^{-5}$, $\Delta t=7.5\times 10^{-3}$ and $h=0.0065$.
The initial concentration is located at $(-0.53,0.53)$ and the final location is $(0,0)$, see Figure~\ref{fig:pde_control_domain} (left).
Figure~\ref{fig:pde_dir_line} shows the snapshots of concentration at three time instances.
\begin{figure}
\includegraphics[width=0.325\linewidth]{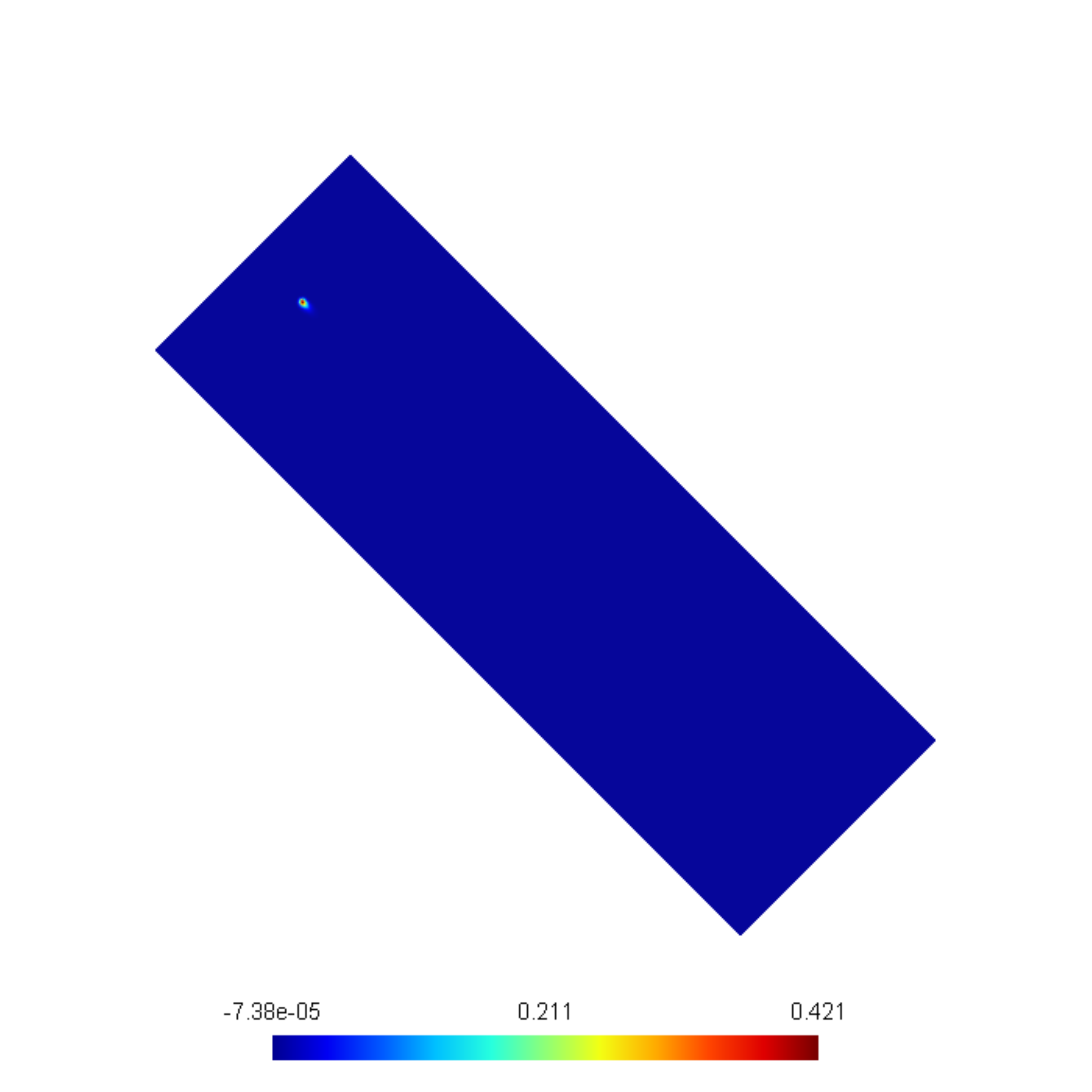}
\includegraphics[width=0.325\linewidth]{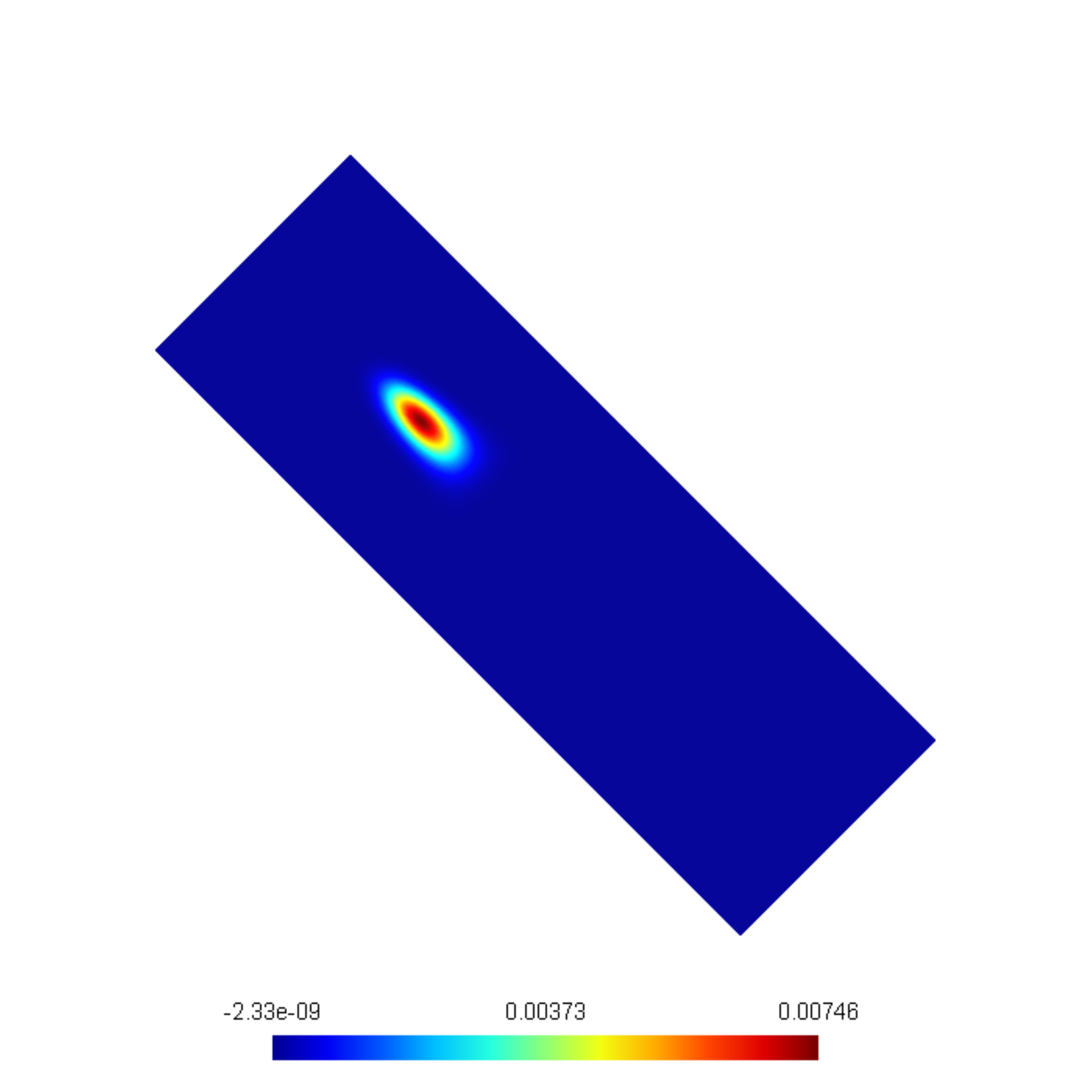}
\includegraphics[width=0.325\linewidth]{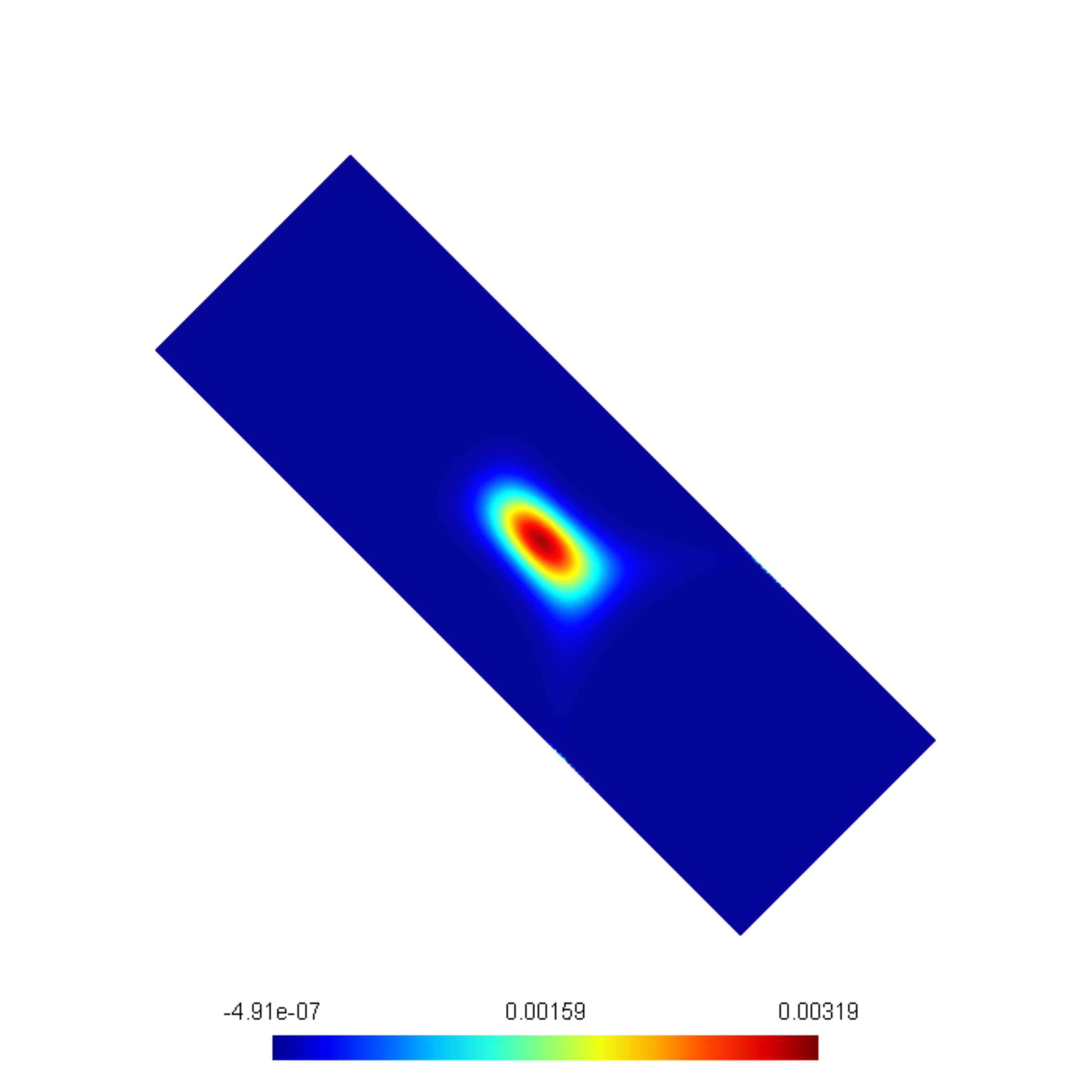}
\caption{Evolution of the concentration in $\widetilde{\Omega}$  at times $t=0, 0.375, 0.75$ with $\eps=10^{-5}$}
\label{fig:pde_dir_line}
\end{figure}

Notice that (cf. Figure~\ref{fig:pde_dir_line}) the concentration reaches the desired location. Since we are using 4 dipoles instead of 8 dipoles used in (cf. \cite{ANV2016,KS2011}) we find it difficult to control the diffusion, also notice that we have Neumann boundary conditions here. Indeed the spikes in Figure~\ref{fig:pde_dir_line} (right) can be explained by the behavior of the magnetic force given in Figures~\ref{fig:control_ini_dir} (left) and \ref{fig:ex_dir_line} (right).

\subsection{Explicit Scheme and Magnetic Injection Around an Obstacle}
\label{s:explicit}

Our next example explores the possibility of moving the concentration around an obstacle.
Indeed a larger $\varepsilon$ leads to significant diffusion and it is not possible to
accomplish our goal. On the other hand, a smaller $\varepsilon$ can computationally
lead to an ill-posed problem; recall the Neumann boundary condition:
\[
(-\varepsilon\nabla c   +c\nabla|\bH|^2)\cdot\boldsymbol{n}=0\quad \mbox{on } \partial\widetilde{\Omega}\times(0,T) .
\]
So when $\varepsilon$ is small then an appropriate boundary condition is Dirichlet (assuming that $\nabla|\bH|^2\cdot\boldsymbol{n}  \neq 0$).

Motivated by this, we consider the advection-diffusion equation \eqref{eq:pb_strong_a}
but with the Dirichlet boundary condition $c=0$ on $\partial\widetilde\Omega\times(0,T)$.
The latter is also physically meaningful if the concentration stays away from the boundary at all times.
Even for small parameter $\varepsilon$ we observe a relatively large diffusion both in
case of EAFE scheme and stabilized implicit (in time) scheme \cite{ANV2016}.
For that reason and, in order to ensure that the concentration is far from the boundary,
we have adopted a piecewise linear finite element scheme for the space approximation and the explicit Euler scheme
for the time discretization.

The fully discrete problem is: Given $c_0\in \rL^2(\widetilde{\Omega})$,
for $n=1,\ldots,N$ find $c^n_h\in \rH_{h,0} \subset \rH_0^1(\widetilde{\Omega})$ such that
\begin{subequations}\label{eq:discrete_explicit}
\begin{align}
(c^n_h,v_h) &=(c^{n-1}_h,v_h) - \Delta t a(c^{n-1}_h,v_h) \quad \forall v_h\in \rH_{h,0}\\
(c^0_h,\phi) &= (c_0,\phi) \quad \forall \phi \in \rH_{h,0} ,
\end{align}
\end{subequations}
with $a(c^{n-1}_h,v_h) = \int_\Omega \varepsilon \nabla c_h^{n-1} \cdot \nabla v_h
- c_h^{n-1}\nabla|\bH|^2 \cdot \nabla v_h \; dx$.

In the matrix vector notation the system in \eqref{eq:discrete_explicit} is given by
\begin{subequations}\label{eq:unlump}
\begin{align}
{\boldsymbol M} {\boldsymbol c}^n &= {\boldsymbol M} {\boldsymbol c}^{n-1} - \Delta t {\boldsymbol A} {\boldsymbol c}^{n-1} \\
{\boldsymbol M} {\boldsymbol c}^0 &= {\boldsymbol M} {\boldsymbol c}_0   ,
\end{align}
\end{subequations}
where ${\boldsymbol M} = (m_{ij})_{i,j}$ with $m_{ij} = \int_\Omega\phi_i\phi_j$ and ${\boldsymbol A}$ are the mass and the stiffness matrices,
respectively. Moreover, for every $k\in\{1,\dots,N\}$ the vector ${\boldsymbol c}^k$ contains the finite element nodal values.
Clearly \eqref{eq:unlump} requires solving a linear system at every time. To avoid such a linear solve we employ the
mass lumping (with correction) of \cite{GP2013}.
Let $\overline{\boldsymbol M} = (\overline{m}_{ij})_{i,j}$ denotes a diagonal lumped mass matrix with
diagonal entries given by $\overline{m}_{ii} = \sum_j m_{ij}$. Then we approximate ${\boldsymbol M}^{-1}$
 by replacing the inverse of the lumped mass matrix by
$(\mathbb{I}+{\boldsymbol B}_r)\overline{\boldsymbol M}^{-1}$ where ${\boldsymbol B}_r=\overline{\boldsymbol M}^{-1}(\overline{\boldsymbol M}-{\boldsymbol M})$.

Our domain of interest is
$\widetilde{\Omega}=[-0.18,0.18]^2\setminus [-0.02,0.02]\times[-0.18,0] $ depicted on
Figure~\ref{fig:pde_control_domain} (right) where the path traversed by the concentration is given by the dotted line.
Our aim is to move the initial concentration
$$
 c_0(x,y) = \exp\left(10^{-4}((x-0.1)^2+(y+0.1)^2)\right)
$$
from the initial location $(0.1,-0.1)$ at $t=0$
to the final location $(-0.1,-0.1)$ at $t=T=0.6$.
With this in mind, we first solve the optimal control problem \eqref{eq:PJ1_disc}
with  a ball of radius $0.2$ centered at $(0.1,-0.1)$ as a reference domain $\hD$.
 We  approximate  a  time dependent vector field given
  by $\bvel(\bx,t)$ such that: $\bvel(\bx,t)= (0,1)^\top$ if $0\leq t< 0.2$,
   $\bvel(\bx,t)=(-1,0)^\top$ if $0.2\leq t< 0.4$   and $\bvel(\bx,t)=(0,-1)^\top$ if $0.4\leq t\leq 0.6$.
  The moving domains $D_{t}$ is such that $D_{t}=\bx_I(t,\hD)$, with
$\bx_I(t,\hbx)=\bvarphi(t)+\hbx$.
Here $\bvarphi(t)$, which represents the
   trajectory of the barycenter of $D_{t}$, is defined as: $\bvarphi(t)= (0.1,-0.1)+t(0,1)^\top$ if $0\leq t< 0.2$,
   $\bvarphi(t)= (0.1,0.1)+(t-0.2)(-1,0)^\top$ if $0.2\leq t< 0.4$
   and $\bvarphi(t)= (-0.1,0.1)+(t-0.4)(0,-1)^\top$ if $0.4\leq t\leq 0.6$. 
Figure~\ref{f:adv_diff_ex2} (top row) shows the resulting magnetic force at four different times.

In order to steer $c_0$ around the obstacle shown in Figure~\ref{f:adv_diff_ex2},
we use the above optimal force as input to \eqref{eq:discrete_explicit}.
The bottom row in Figure~\ref{f:adv_diff_ex2} shows the concentration snapshots at four different times.
Here we have set $\varepsilon = 10^{-8}$,  $\Delta t = 3\times10^{-5}$ and mesh size $h=0.0016$.

\begin{figure}[h!]
\centering
\includegraphics[width=0.24\textwidth]{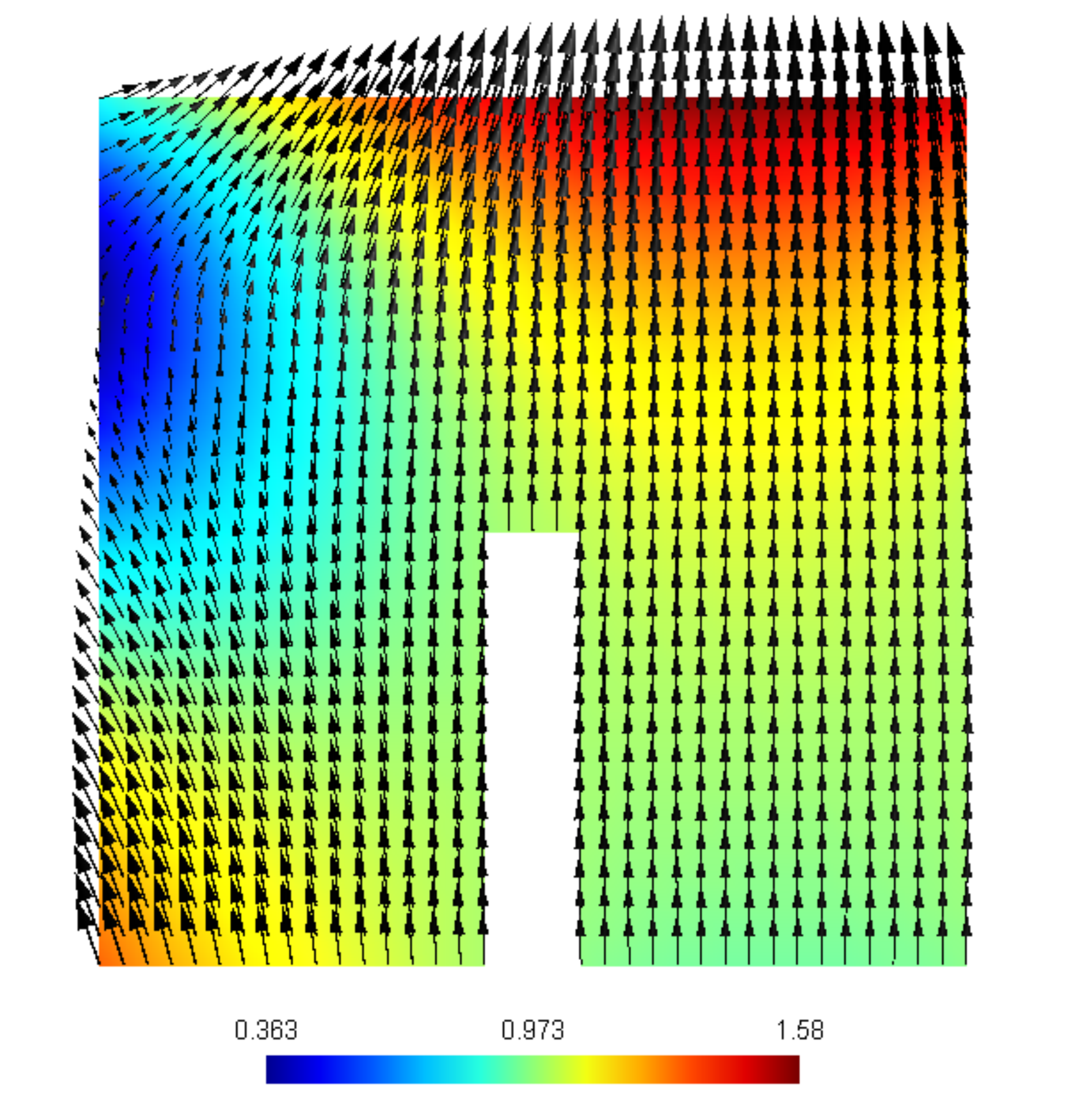}
\includegraphics[width=0.24\textwidth]{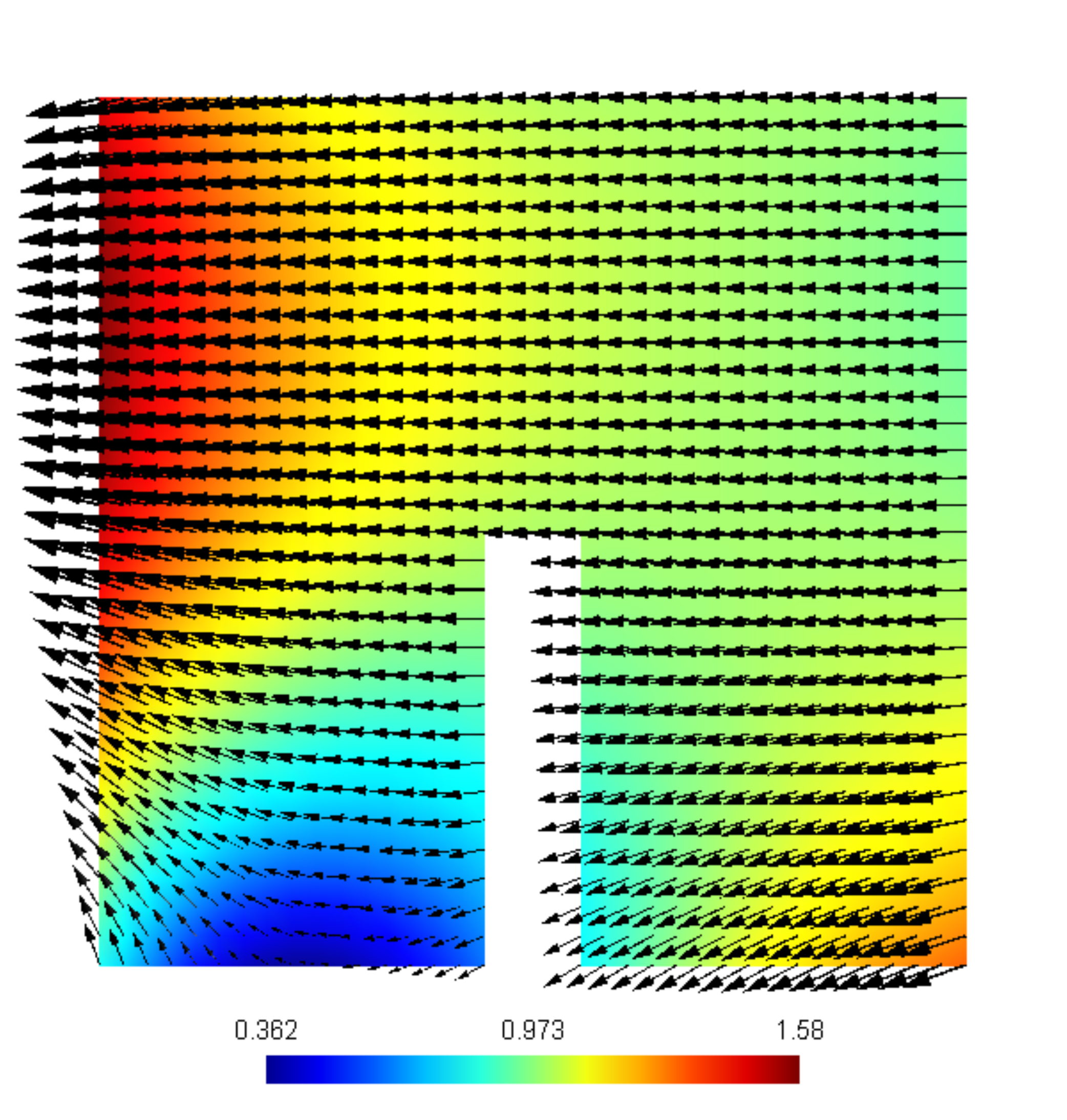}
\includegraphics[width=0.24\textwidth]{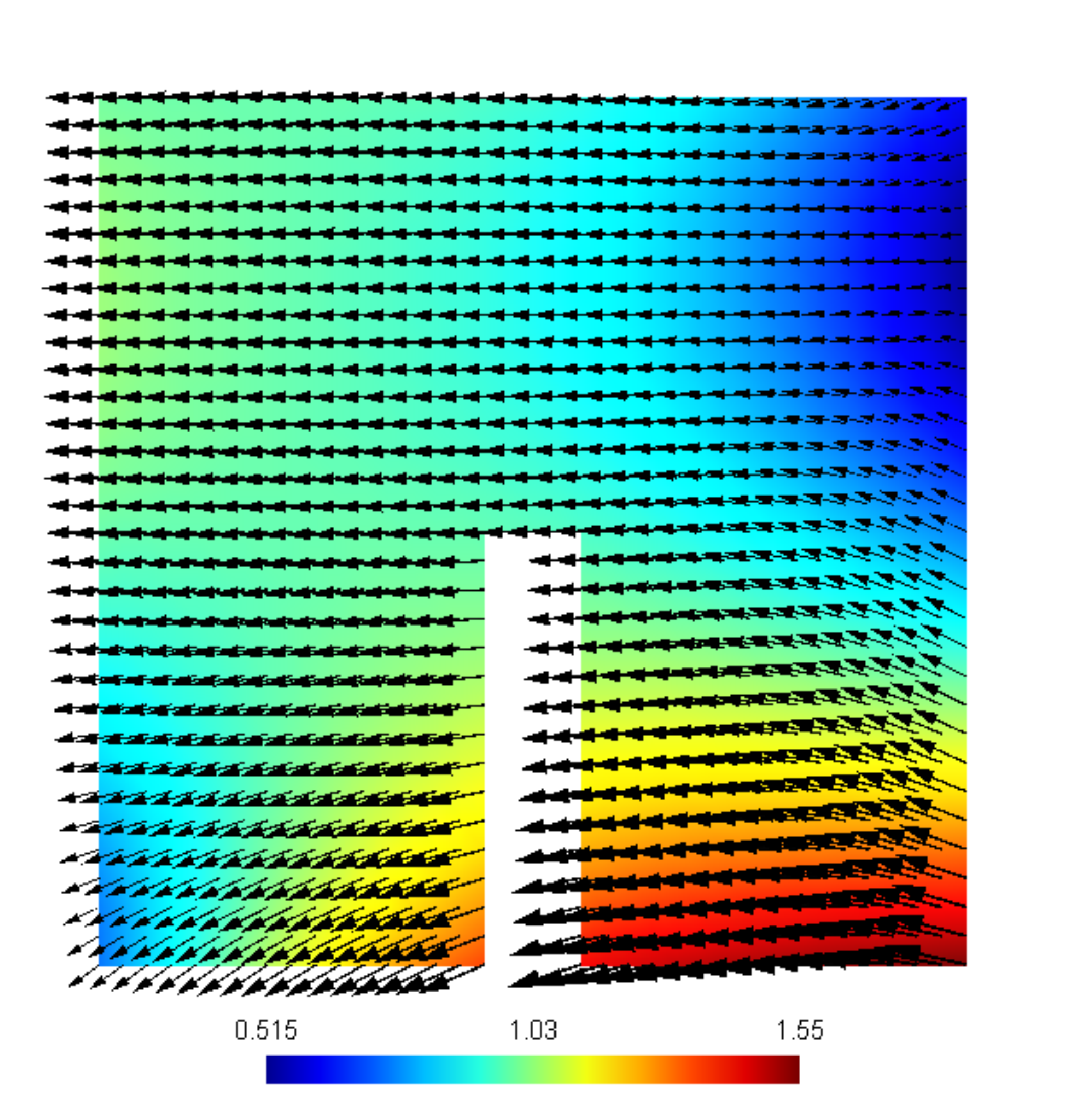}
\includegraphics[width=0.24\textwidth]{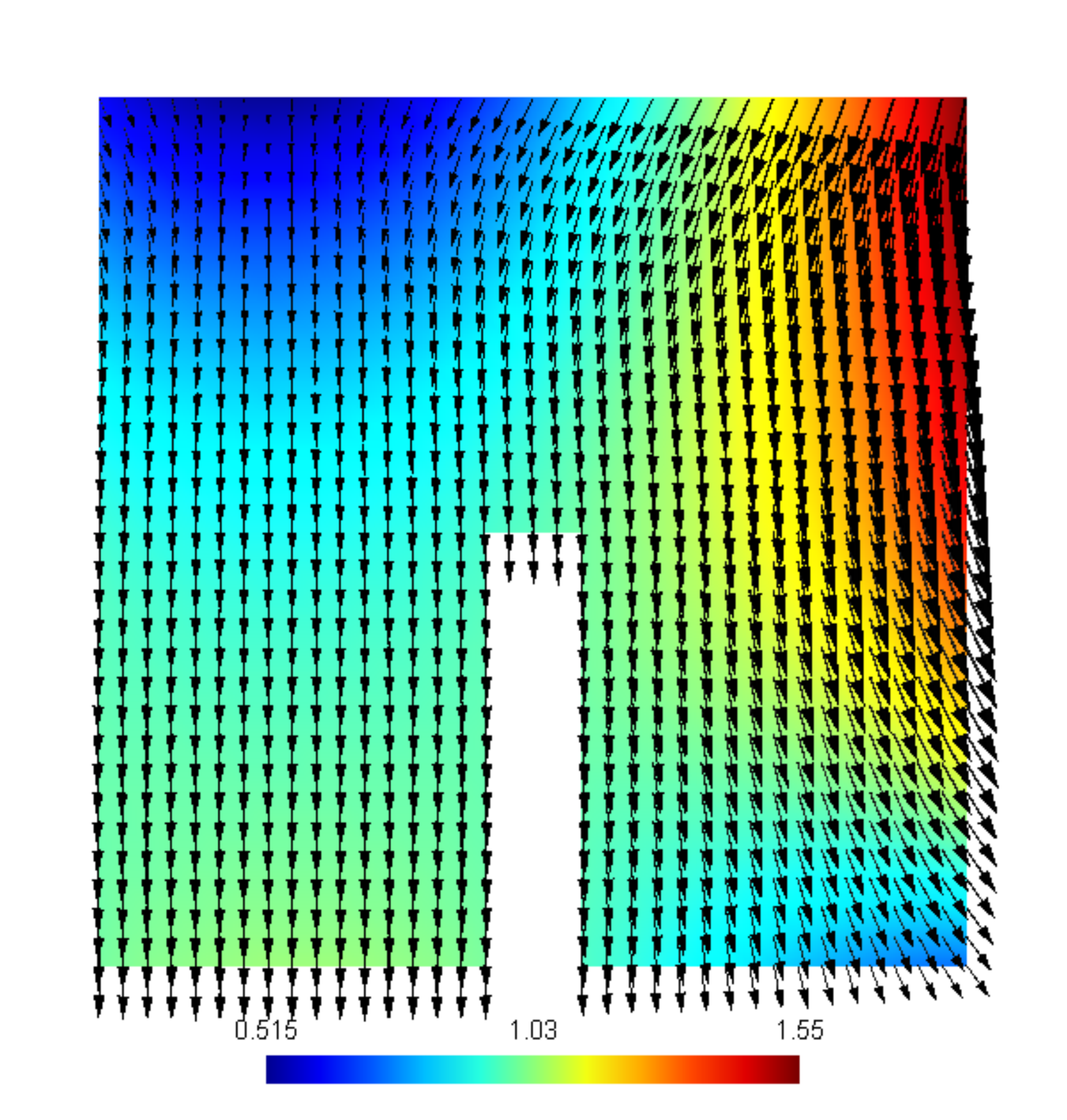}

\includegraphics[width=0.24\textwidth]{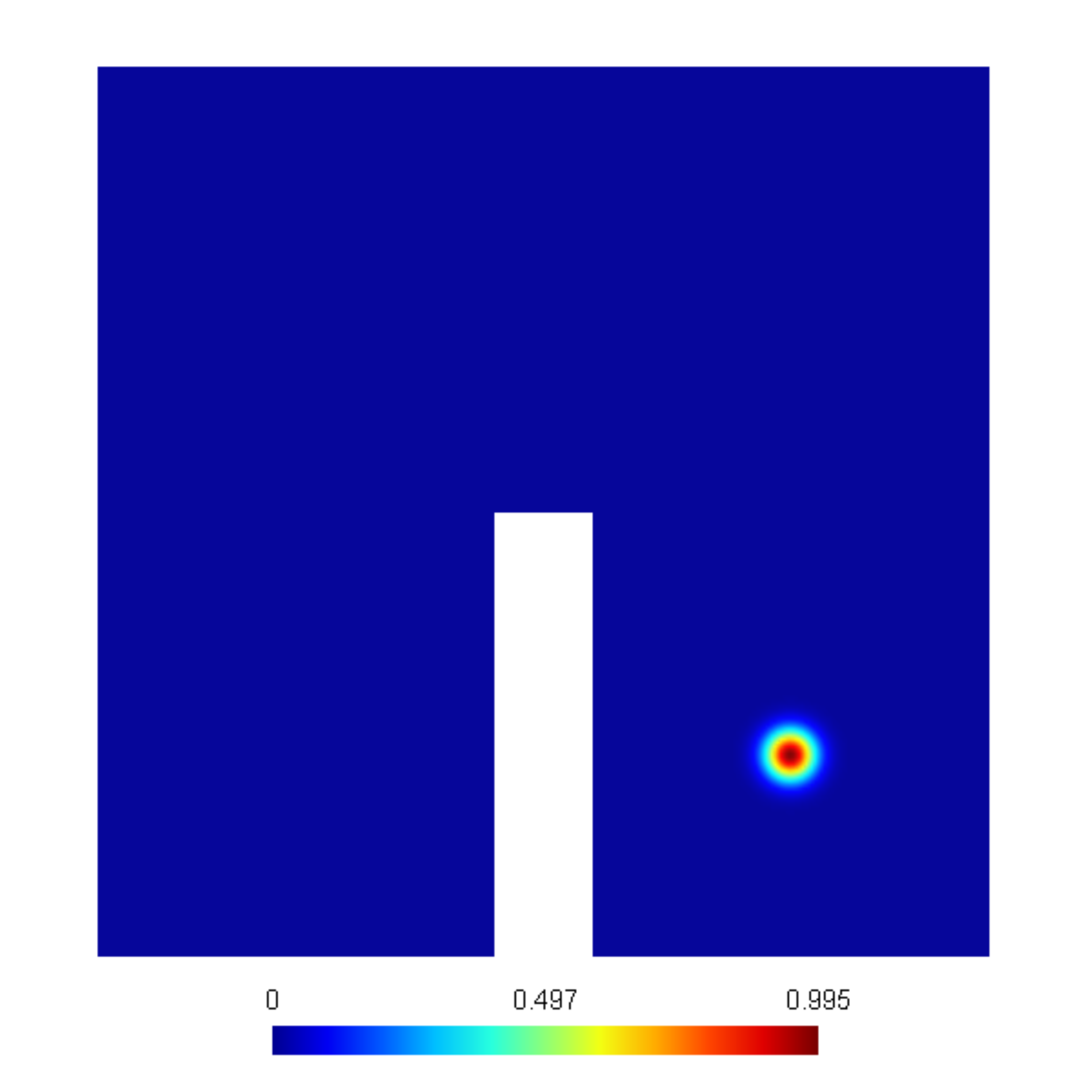}
\includegraphics[width=0.24\textwidth]{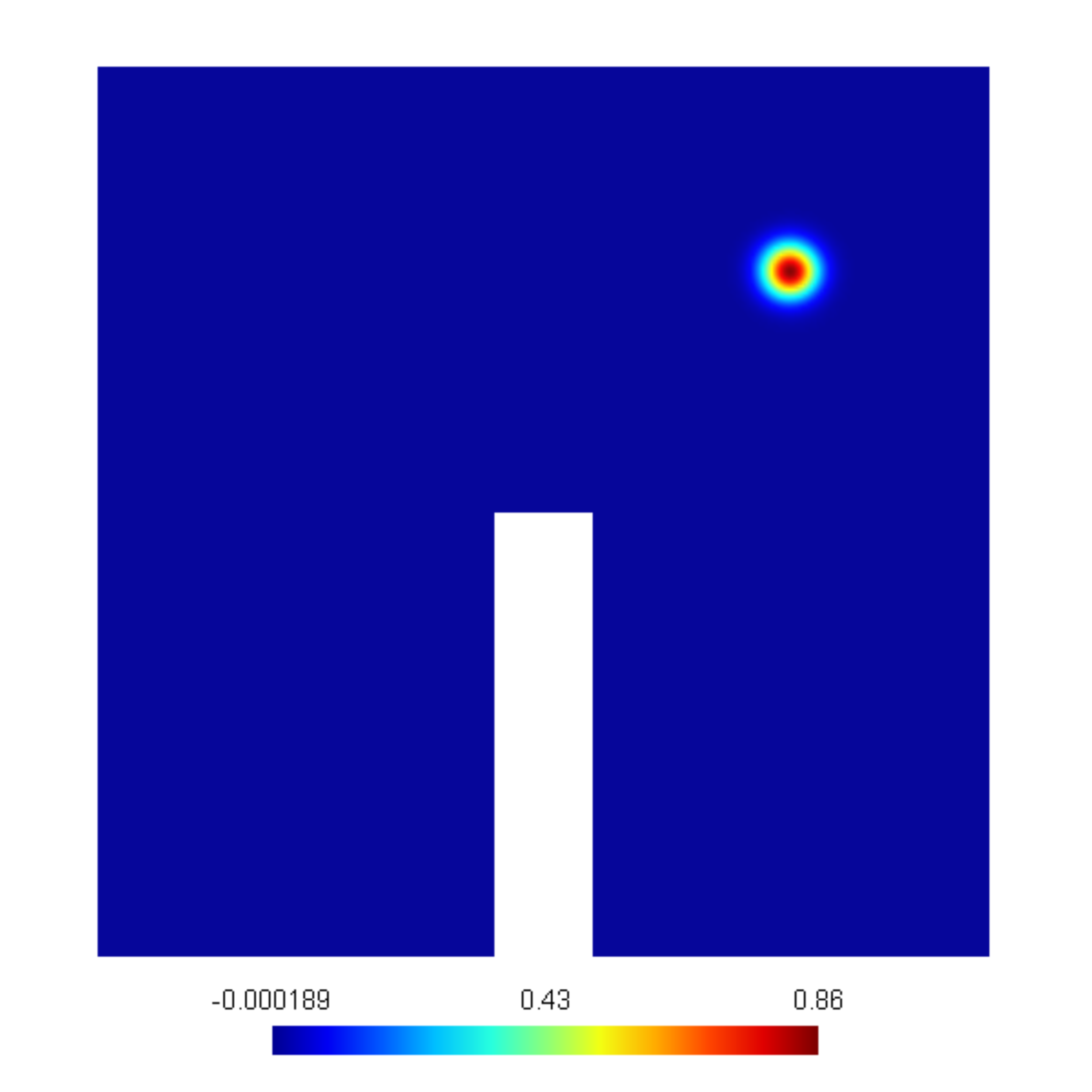}
\includegraphics[width=0.24\textwidth]{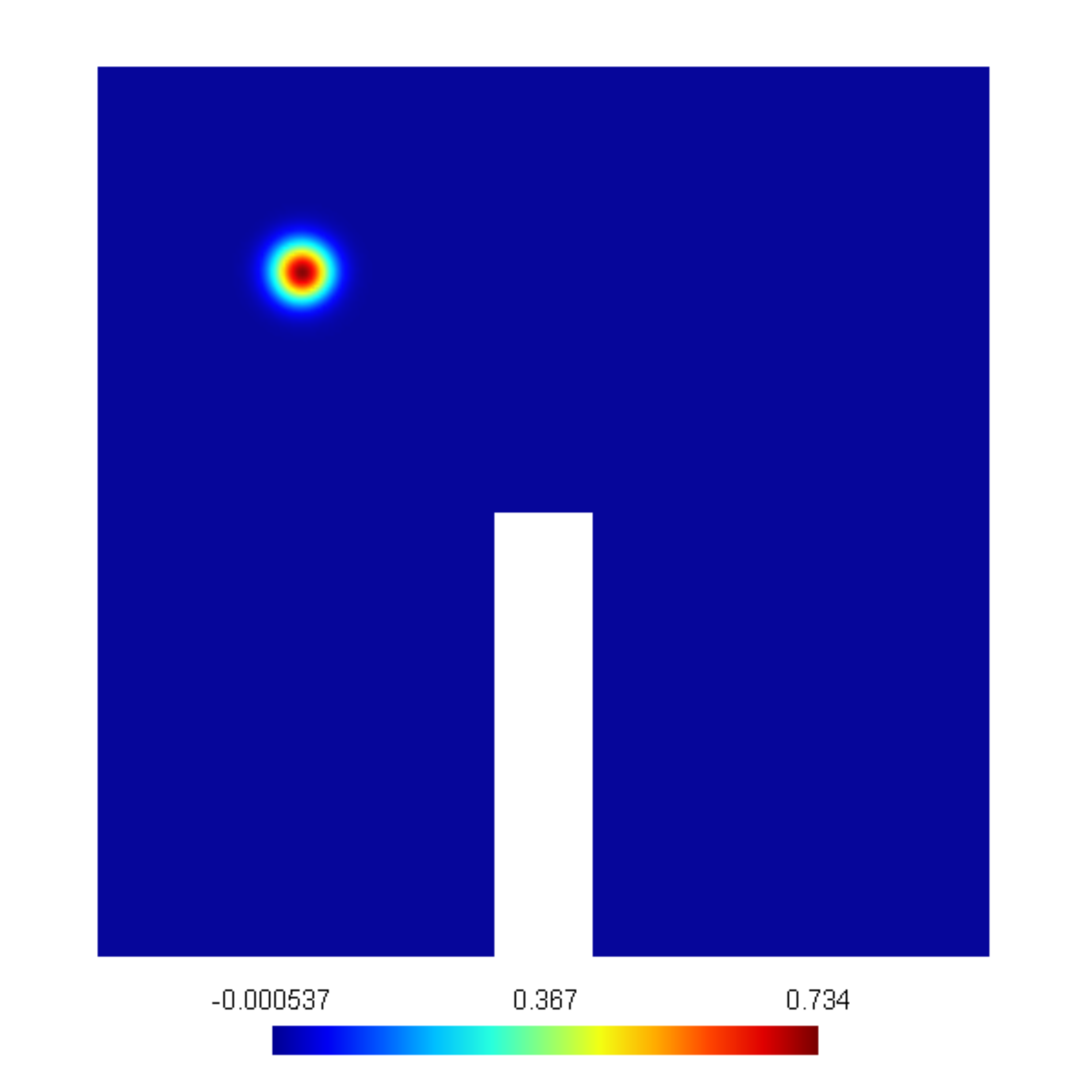}
\includegraphics[width=0.24\textwidth]{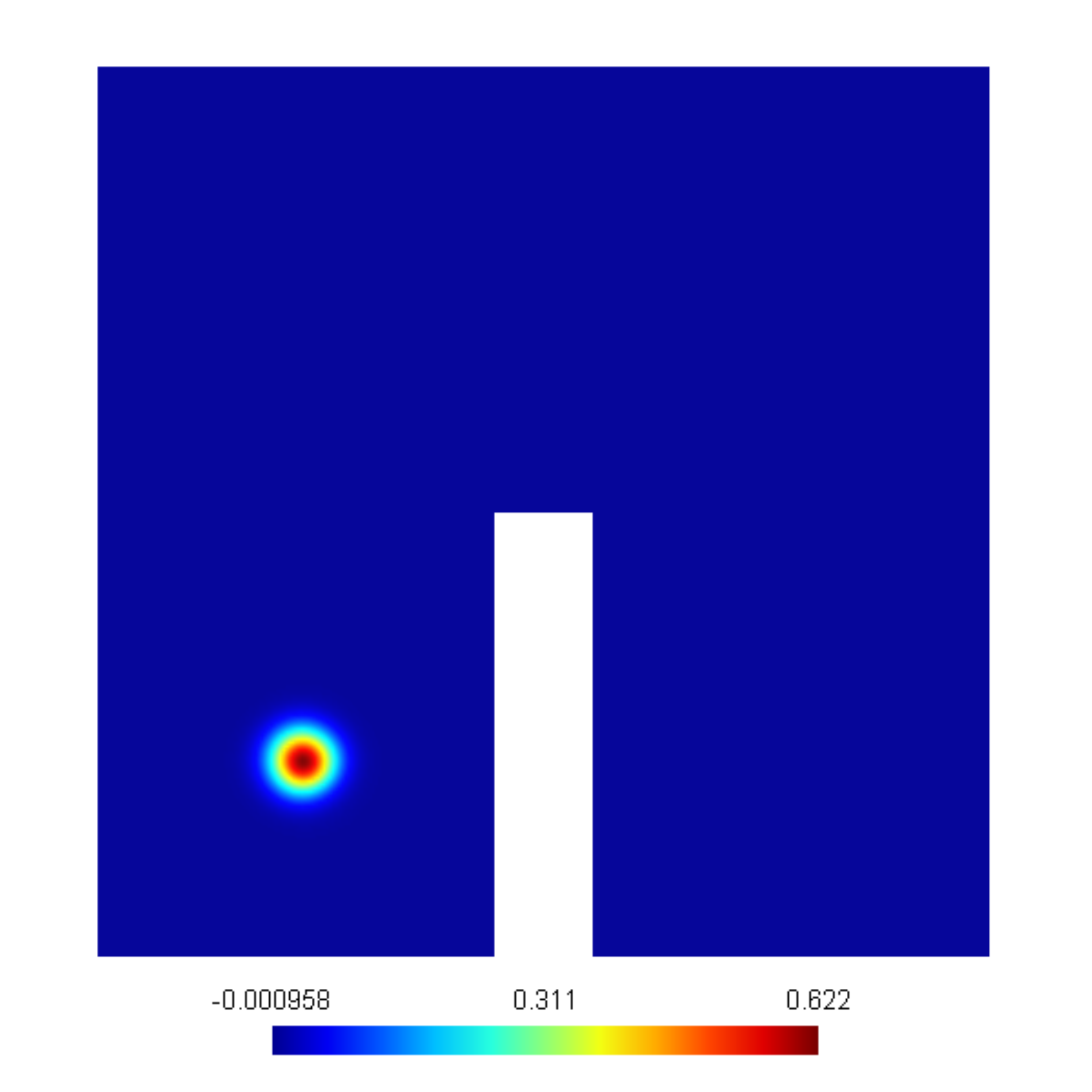}
\caption{\label{f:adv_diff_ex2}Top figures: optimal magnetic force computed by solving \eqref{eq:PJ1_disc}.
 The magnitude the magnetic force  is shown by the background
color. Bottom figures: Evolution of the concentration in $\widetilde{\Omega}$
when the optimal magnetic force is used as an input to \eqref{eq:discrete_explicit}. 
The columns correspond to $t =0, 0.2, 0.4$ and $0.6$ from left to right, respectively.} 
\end{figure}

The top row in Figure~\ref{f:adv_diff_ex2} shows the optimal magnetic force  obtained by solving \eqref{eq:PJ1_disc}.
The bottom row shows the evolution of the concentration at four different time instances.
Notice that the magnetic forces prevent the concentration from reaching the boundary  $\partial \widetilde{\Omega}$
and enables the concentration to reach the final location at  $(-0.1,0.1)$ while avoiding the obstacle.

%

\bibliographystyle{plain}
\bibliography{Ref}

\end{document}